\documentclass[]{interact}

\usepackage{epstopdf}

\usepackage[numbers,sort&compress]{natbib}
\bibpunct[, ]{[}{]}{,}{n}{,}{,}

\usepackage{amsfonts, bm, amsmath, graphicx, amssymb, amsthm, mathdots, mathtools}
\usepackage[utf8]{inputenc}
\usepackage[english]{babel}

\usepackage{enumerate}
\usepackage{float}
\usepackage[table]{xcolor}
\usepackage{hyperref}
\usepackage{comment}
\usepackage{subcaption}

\theoremstyle{plain}
\newtheorem{theorem}{Theorem}[section]
\newtheorem{lemma}[theorem]{Lemma}
\newtheorem{corollary}[theorem]{Corollary}

\newtheoremstyle{claims}
{3pt}
{3pt}
{\itshape}
{}
{\itshape}
{}
{.5em}
{}
\theoremstyle{claims}
\newtheorem{claim}{\textit{Claim}}

\theoremstyle{definition}
\newtheorem{definition}[theorem]{Definition}

\theoremstyle{remark}
\newtheorem{remark}{Remark}

\newcounter{cases}
\newcounter{subcases}[cases]
\newcounter{subsubcases}[subcases]
\newenvironment{mycases}
  {%
    \setcounter{cases}{0}%
    \setcounter{subcases}{0}%
    \setcounter{subsubcases}{0}%
    \def\case
      {%
        \par\noindent
        \refstepcounter{cases}%
        \textit{\textbf{Case (\thecases})}
      }%
    \def\subcase
      {%
        \par\noindent
        \refstepcounter{subcases}%
        \textbf{\textit{Case (\thecases.\thesubcases) }}
      }%
     \def\subsubcase
      {%
        \par\noindent
        \refstepcounter{subsubcases}%
        \textbf{\textit{Case (\thecases.\thesubcases.\thesubsubcases) }}
      }%
  }
  {%
    \par
  }
\renewcommand*\thecases{\arabic{cases}}
\renewcommand*\thesubcases{\arabic{subcases}}
\renewcommand*\thesubsubcases{\arabic{subsubcases}}

\newcommand*{\QED}{\hfill\ensuremath{\diamondsuit}}%
\newcommand*{\tagg}{\mathrm{tag}}%

\definecolor{dark-blue}{RGB}{60,80,170}
\definecolor{dark-red}{RGB}{240,50,60}
\definecolor{dark-orange}{RGB}{255,147,23}
\definecolor{dark-green}{RGB}{65, 173, 57}

\begin{document}


\title{$2$-nested matrices: towards understanding the structure of circle graphs}

\author{\name{Guillermo Dur\'an\textsuperscript{1,2,3}
  and Nina Pardal\textsuperscript{1,4}
  and Mart\'in D. Safe\textsuperscript{5,6} }
  \affil{\textsuperscript{1 }CONICET-Universidad de Buenos Aires, Instituto de C\'alculo (IC), Buenos Aires, Argentina\\  
 \textsuperscript{2} Universidad de Buenos Aires, Facultad de Ciencias Exactas y Naturales, Departamento de Matem\'atica, Buenos Aires, Argentina\\ 
  \textsuperscript{3} Departamento de Ingenier\'ia Industrial, Facultad de Ciencias F\'isicas y Matem\'aticas, Universidad de Chile, Santiago, Chile\\  
  \textsuperscript{4} CONICET-Universidad de Buenos Aires, Instituto de Investigaci\'on en Ciencias de la Computaci\'on (ICC), Buenos Aires, Argentina\\  
  \textsuperscript{5} Departamento de Matem\'atica, Universidad Nacional del Sur (UNS), Bah\'ia Blanca, Argentina\\
  \textsuperscript{6} INMABB, Universidad Nacional del Sur (UNS)-CONICET, Bah\'ia Blanca, Argentina} 
  }

\maketitle

\begin{abstract}

A $(0,1)$-matrix has the \emph{consecutive-ones property} (C$1$P) if its columns can be permuted to make the $1$'s in each row appear consecutively. This property was characterised in terms of forbidden submatrices by Tucker in 1972. Several graph classes were characterised by means of this property, including interval graphs and strongly chordal digraphs.

In this work, we define and characterise $2$-nested matrices, which are $(0,1)$-matrices with a variant of the C$1$P and for which there is also certain assignment of one of two colors to each block of consecutive $1$'s in each row.
The characterization of $2$-nested matrices in the present work is of key importance to characterise split graphs that are also circle by minimal forbidden induced subgraphs.

\end{abstract}


\section{Introduction}\label{section:intro}

A $(0,1)$-matrix has the \emph{consecutive-ones property} (C$1$P) if there is a permutation of its columns such that the $1$'s in each row appear consecutively.
It appears naturally in a wide range of applications, and more precisely, in any problem in which we are required to linearly arrange a set of objects with the restriction that some sets of objects must appear consecutively~\cite{AS95,L83,W97}.

The C$1$P has been widely studied. The matrices with this property were studied by Fulkerson and Gross in \cite{FG65}, where a characterization of interval graphs in terms of the C$1$P of their clique-matrices was given. 
The C$1$P, along with some variants of it, has been widely used to study structural properties of several other graph classes, such as proper interval graphs~\cite{R68}, proper interval bigraphs~\cite{SS94}, strongly chordal digraphs \cite{HHHL19}, and adjusted interval digraphs \cite{FHHR12}. In 1972, Tucker~\cite{T72} presented a characterization of the C$1$P in terms of forbidden submatrices. The corresponding forbidden matrices were later called \emph{Tucker matrices}.

In this work, we define and characterise $2$-nested matrices, which are $(0,1)$-matrices with a variant of the C$1$P and for which there is also a particular assignment of one of two colors to each block of consecutive $1$'s in each row. This characterization is a continuation of the work in~\cite{PDGS18}. 
A graph is \emph{circle}~\cite{EI71} if it is the intersection graph of a set of chords on a circle; if so, the set of chords is called a \emph{circle model} of the graph. A \emph{split graph}~\cite{FH76} is any graph $G$ admitting a \emph{split partition} $(K,S)$, i.e.\ a partition of its vertex set into a clique $K$ and a stable set $S$; we denote it by $G=(K,S)$. In~\cite{BDPS20,P20}, we addressed the problem of characterizing those split graphs that are also circle by strongly relying on the characterization of $2$-nested graphs proved in the current paper. We now briefly introduce the connection between split graphs that are circle graphs and $2$-nested matrices.

Let us first consider a split graph $G$ that is \emph{minimally non-circle}, i.e.\ $G$ is non-circle but any proper induced subgraph of $G$ is circle. \emph{Permutation graphs} are exactly those comparability graphs whose complement is also a comparability graph~\cite{EPL72}. 
Since permutation graphs are circle (see e.g.~\cite[p.\ 252]{G04}) and $G$ is non-circle, $G$ is not a permutation graph. Comparability graphs were characterised by forbidden induced subgraphs in~\cite{G67}. 
This characterization immediately leads to a forbidden induced subgraph characterization for permutation graphs.
By relying on the corresponding list of forbidden subgraphs and the fact that $G$ is also a split graph, we concluded that $G$ contains an induced subgraph $H$ isomorphic to one of the graphs in Figure~\ref{fig:forb_permsplit_base}~(see \cite{BDPS20}). 
\begin{figure}[h]
\centering
\includegraphics[scale=.25]{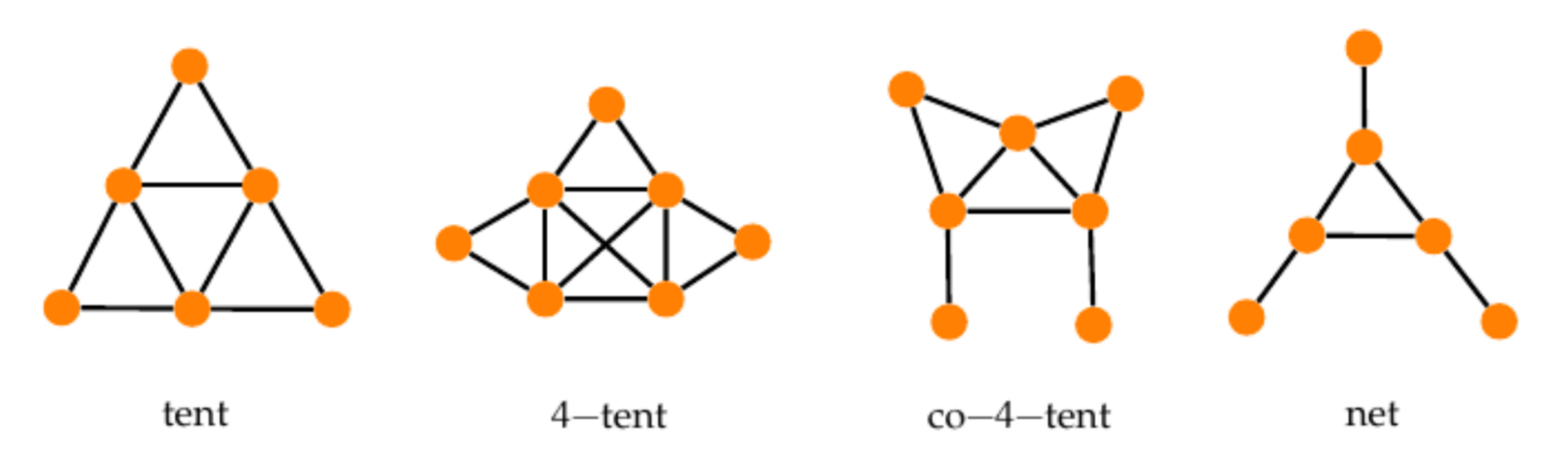}
\caption{Forbidden induced subgraphs for permutation graphs within split graphs.} \label{fig:forb_permsplit_base}
\end{figure}

Let us now consider an arbitrary split graph $G$ (either circle or not). Because of the discussion in the preceding paragraph, if $G$ contains none of the graphs in Figure~\ref{fig:forb_permsplit_base} as an induced subgraph, then $G$ is a circle graph. Thus, we assume next that $G$ contains an induced subgraph $H$ isomorphic to tent, $4$-tent or co-$4$-tent. Let $(K,S)$ be a split partition of $G$. 
We consider a partition $K_1, \ldots, K_j$ of the vertices of $K$, where each set $K_i$ consists of all the vertices of $K$ having a same set of neighbours in $V(H)\cap S$.
For each set $K_i$, let $S_i$ be the set of vertices in $S$ that are adjacent to at least one vertex of $K_i$. We consider the $(0,1)$-matrices $A(S_i,K_i)$ that represent the adjacencies between $S_i$ and $K_i$.
In~\cite{BDPS20}, 
we show that it suffices to study each of these matrices and the relationship between them to decide whether $G$ is a circle graph or not. More precisely, the key condition for $G$ to be a circle graph is that each of these matrices $A(S_i,K_i)$ is $2$-nested. 
This analysis allowed us to give a minimal forbidden induced subgraph characterization of those split graphs that are circle graphs~\cite{BDPS20}.

This work is organized as follows. In Section~\ref{section:basic_defs} we give some basic definitions and notation that will be useful thoughout this paper. In Section~\ref{section:motivation}, we give a motivating example and set the necessary bases to define $2$-nested matrices.
In Section~\ref{section:2nested_matrices}, we characterise $2$-nested matrices by minimal forbidden subconfigurations. 

\section{Basic definitions and notation} \label{section:basic_defs}

Let $A=(a_{ij})$ be a $n\times m$ $(0,1)$-matrix.
We denote $a_{i.}$ and $a_{.j}$ the $i$th row and the $j$th column of matrix $A$. From now on, we associate each row $a_{i.}$ with the set of columns in which $a_{i.}$ has a $1$. For example, the \emph{intersection} of two rows $a_{i.}$ and $a_{j.}$ is the subset of columns in which both rows have a $1$.
Two rows $a_{i.}$ and $a_{k.}$ are \emph{disjoint} if there is no $j$ such that $a_{ij} = a_{kj} = 1$.
We say that $a_{i.}$ is \emph{contained} in $a_{k.}$ if for each $j$ such that $a_{ij} = 1$ also $a_{kj} = 1$. We say that $a_{i.}$ and $a_{k.}$ are \emph{nested} if $a_{i.}$ is contained in $a_{k.}$ or $a_{k.}$ is contained in $a_{i.}$.
We say that a row $a_{i.}$ is \emph{empty} if every entry of $a_{i.}$ is $0$, and we say that $a_{i.}$ is \emph{nonempty} if there is at least one entry of $a_{i.}$ equal to $1$.
We say that two nonempty rows \emph{overlap} if they are non-disjoint and non-nested.
For every nonempty row $a_{i.}$, let $l_i = \min\{ j \colon\,a_{ij} = 1 \}$ and $r_i = \max\{ j \colon\,a_{ij} = 1 \}$ for each $i\in\{1,\ldots,n\}$.
Finally, we say that $a_{i.}$ and $a_{k.}$ \emph{start} (resp.\ \emph{end}) \emph{in the same column} if $l_i = l_k$ (resp.\ $r_i = r_k$), and we say $a_{i.}$ and $a_{k.}$ \emph{start (end) in different columns}, otherwise. The \emph{complement} of a row of a $(0,1)$-matrix arises by turning all the $0$'s into $1$'s and vice versa. We denote the complement of a row $r$ by $\overline r$.

All graphs in this work are simple, undirected, with no loops and no multiple edges. We denote the vertex and edge set of a graph by $V(G)$ and $E(G)$, respectively. A \emph{clique} is a set of pairwise adjacent vertices. A \emph{stable set} is a set of pairwise nonadjacent vertices. Let $G$ be a graph. If $W\subseteq V(G)$, the \emph{subgraph of $G$ induced by $W$}, denoted $G[W]$, is the graph with vertex set $W$ and whose edges are those of $G$ with both endpoints in $W$. If $H$ is an induced subgraph of $G$, we denote by $G-H$ the graph $G[V(G)-V(H)]$. 
We say a vertex $v$ is \emph{complete to} the set of vertices $X$ if $v$ is adjacent to every vertex in $X$, and we say $v$ is \emph{anticomplete to} $X$ if $v$ has no neighbour in $X$. We say that $v$ is \emph{adjacent to} $X$ if $v$ has at least one neighbour in $X$.
If $v$ is a vertex of $G$, we denote by $N(v)$ the set of vertices adjacent to $v$. 
Given two vertices $v_1$ and $v_2$ in $S$, we say that $v_1$ and $v_2$ are \emph{nested} if either $N(v_1) \subseteq N(v_2)$ or $N(v_2) \subseteq N(v_1)$. The \emph{length} of a path or cycle is the number of edges of the path joining consecutive vertices. A path or cycle is \emph{odd} or \emph{even} depending on whether its length is odd or even, respectively. A path or cycle is \emph{induced} if there is no edge joining two non-consecutive vertices. A \emph{$2$-coloring} of $G$ is a mapping that assigns one of two colors to each vertex of $G$. A $2$-coloring is \emph{proper} if each two adjacent vertices are assigned different colors.

Let $G=(K,S)$ be a split graph. Let $s_1, \ldots, s_n$ and $v_1, \ldots, v_m$ be linear orderings of $S$ and $K$, respectively. Let $A= A(S,K)$ be the $n\times m$ matrix defined by $A(i,j)=1$ if $s_i$ is adjacent to $v_j$ and $A(i,j)=0$, otherwise.
From now on, we associate the rows (resp.\ columns) of the adjacency matrix $A(S,K)$ with the corresponding vertex in $S$ (resp.\ vertex in $K$). Given a partition $K_1,K_2,\ldots,K_j$ of $K$ and a vertex $v$ in $S$, we denote $N_i(v) = N(v) \cap K_i$.

\section{Motivation: the connection between 2-nested matrices and circle graphs} \label{section:motivation}

In this section we first give motivating examples for the definition of nested and $2$-nested matrices. Afterwards, we define nested and $2$-nested matrices, which are of fundamental importance to describe a circle model for those split graphs that are also circle, as we will see in the following examples. 

Let us consider the split graph $G=(K,S)$ represented in Figure~\ref{fig:example_graph0}. Since $G$ contains an tent $H$ induced by $\{k_1,k_3,k_5,s_{13},s_{35},s_{51}\}$, we consider the partitions $K_1, K_2 \ldots, K_6$ of $K$ and $\{ S_{ij} \}_{1 \leq i,j \leq 6}$ of $S$, defined as follows.
\begin{itemize}
 \item For each $i\in\{1,3,5\}$, let $K_i$ be the set of vertices of $K$ whose neighbours in $V(H)\cap S$ are precisely $s_{(i-2)i}$ and $s_{i(i+2)}$ (where subindexes are modulo~$6$).
 \item For each $i\in\{2,4,6\}$, let $K_i$ be the set of vertices of $K$ whose only neighbour in $V(H)\cap S$ is $s_{(i-1)(i+1)}$ (where subindexes are modulo~$6$).
 \item For $i,j\in\{1,\ldots,6\}$, let $S_{ij}$ be the set of vertices of $S$ that are adjacent to some vertex in $K_i$ and some vertex in $K_j$, are complete to $K_{i+1}$, $K_{i+2}$ ,$\ldots, K_{j-1}$, and are anticomplete to $K_{j+1}$, $K_{j+2},\ldots, K_{i-1}$ (where subindexes are modulo~$6$).
\end{itemize}

\begin{figure}[h!] 
	\centering
	 \begin{subfigure}[b]{0.34\textwidth}
		\includegraphics[width=\textwidth]{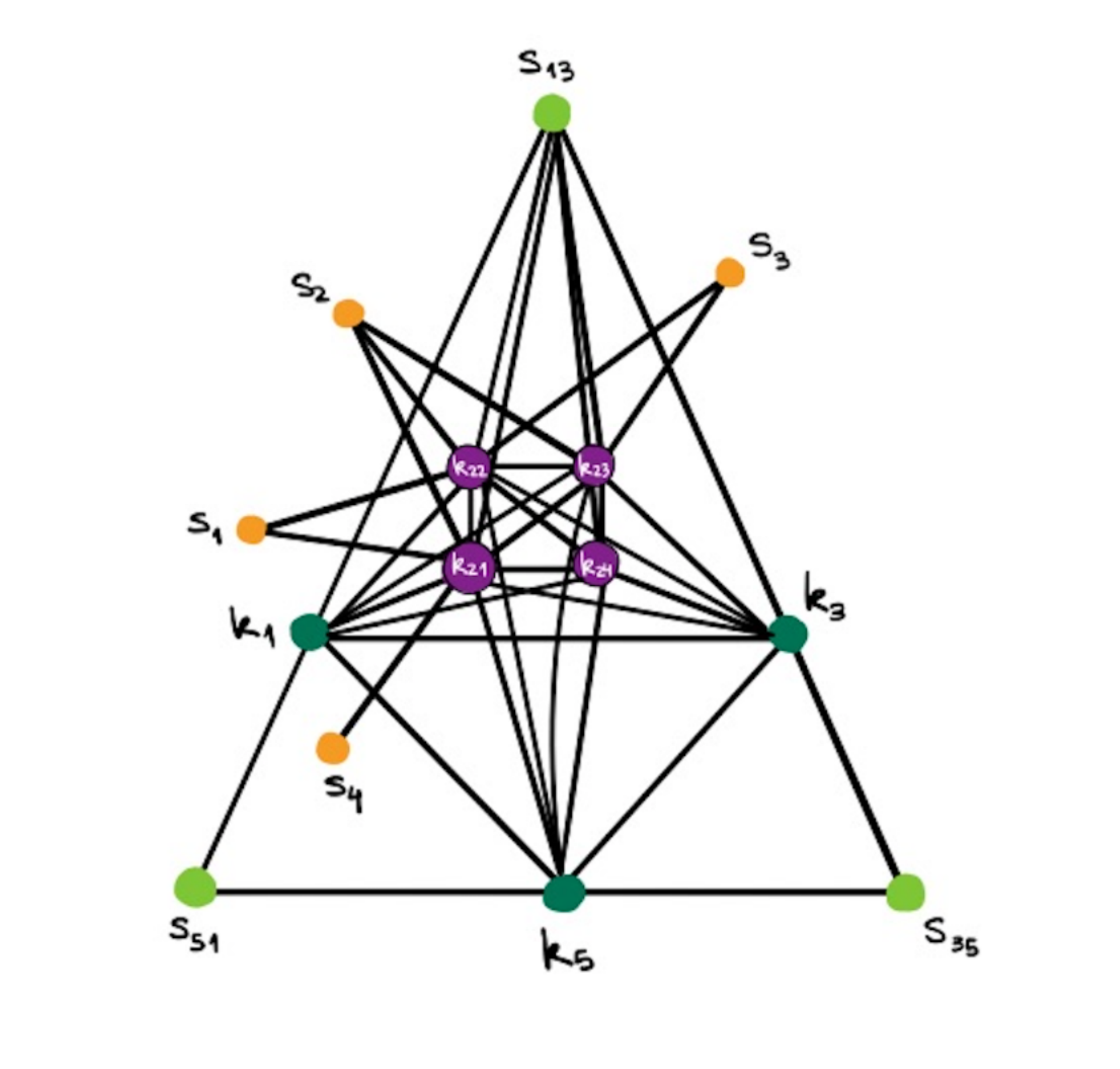}
		\caption{A split circle graph $G$.}
		\label{fig:example_graph0}
	\end{subfigure}
	\begin{subfigure}[b]{0.36\textwidth}
		\includegraphics[width=\textwidth]{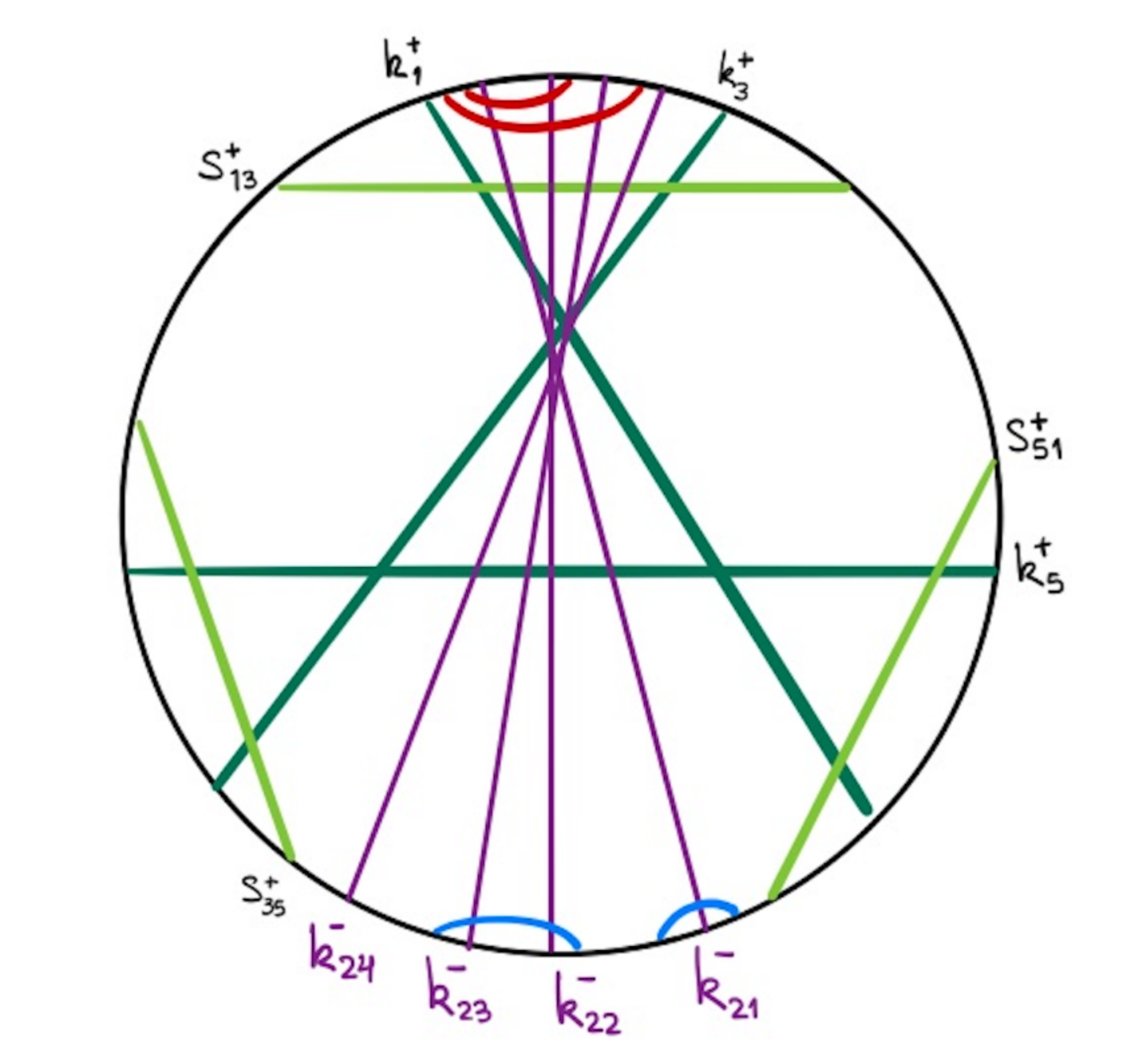}
		\caption{A circle model for $G$.}
		\label{fig:example_model0}
	\end{subfigure}
	\caption{Graph and circle model of Example $1$.}
\end{figure}

Notice that every vertex in $K \setminus V(H)$ lies in $K_2$, for the only adjacency of these vertices with regard to $V(H)\cap S$ is the vertex $s_{13}$. 
Thus, $K_2 = \{ k_{21}$, $k_{22}$, $k_{23}$, $k_{24}\}$. 
Moreover, the orange vertices are precisely $S\setminus V(H)$ and these vertices are adjacent only to vertices in $K_2$. Thus, they all lie in $S_{22}$.
We want to find properties that help us decide whether we can give a circle model for $G$ or not. The tent graph admits a unique circle model in the sense that circular ordering of the endpoints of the chords in the model is uniquely determined. This is because the tent is prime with respect to the split decomposition (see~\cite{B87}). Hence, let us begin by considering a circle model for $H$ as the one presented in Figure~\ref{fig:example_tentmodel}. 
We denote the arcs and chords of a model by their endpoints in clockwise order. For example, in Figure~\ref{fig:example_tentmodel} the arc $k_1^+ k_3^+$ is the portion of the circle that lies between $k_1$ and $k_3$ while traversing the circumference clockwise.

\begin{figure}[h!] 
	\centering
	\includegraphics[scale=.27]{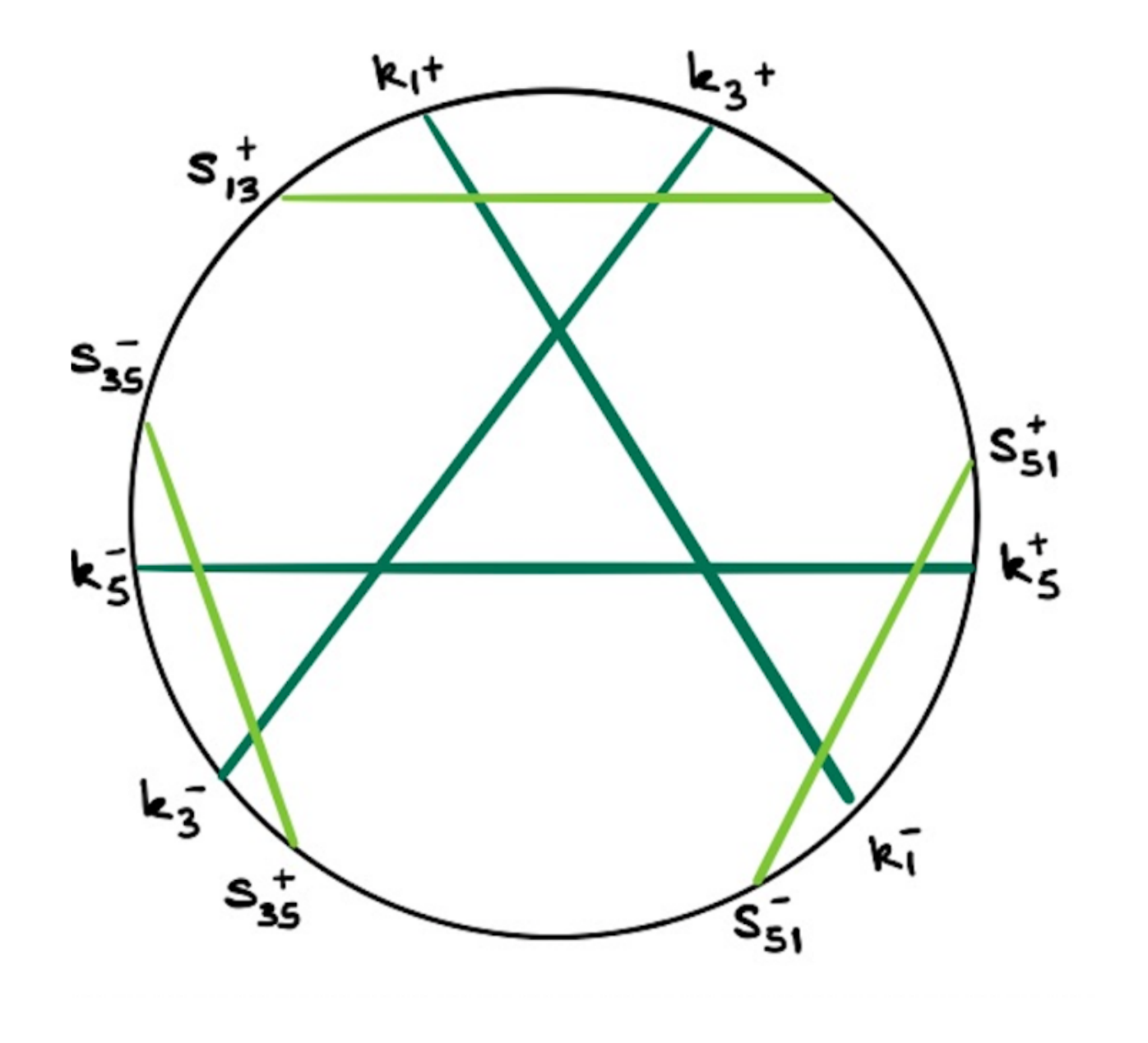}
	\caption{A circle model for the tent graph $H$.}
		\label{fig:example_tentmodel}
\end{figure}


In order to place the chords corresponding to each vertex in $S_{22}$, we need to place first the chords that represent every vertex in $K_2$. Notice that a chord re\-pre\-sen\-ting a vertex in $K_2$ has one endpoint between the arc $k_1^+ k_3^+$ and the other endpoint between the arc $s_{51}^- s_{35}^+$ 
while a chord representing a vertex in $S_{22}$ has either both endpoints inside the arc $k_1^+ k_3^+$ or both endpoints inside the arc $s_{51}^- s_{35}^+$, always intersecting at least one chord that represents a vertex of $K_2$.
Thus, in order to place the chords corresponding to the vertices of $K_2$ we need to establish a ``good ordering'' for these vertices, this is, one that respects the relationships among the neighbourhoods of the vertices in $S_{22}$. 
For example, since $N(s_1) \subseteq N(s_2)$, it follows that an ordering of the chords in $K_2$ that allows us to give a circle model must contain one of the following four subsequences: $(k_{21} < k_{22} <k_{23})$, $(k_{22}<k_{21}<k_{23})$, $(k_{23}<k_{21}<k_{22})$ or $(k_{23}<k_{22}<k_{21})$. 
Moreover, since $N(s_2) \cap N(s_3) \neq \emptyset$ and $N(s_2)$ and $N(s_3)$ are not nested, then the chords corresponding to $s_2$ and $s_3$ must be drawn in distinct portions of the circle model, for they represent vertices in $S$ and thus the chords cannot intersect. 
The vertex $s_4$ is adjacent only to $k_{21}$, thus $N(s_4)$ is contained in both $N(s_1)$ and $N(s_2)$ and is disjoint with $N(s_3)$. Hence, the chord that represents $s_4$ may be placed indistinctly in any of the two portions of the circle corresponding to the partition $S_{22}$.

Therefore, when considering the placement of the chords, we find ourselves facing two important decisions: (1) in which order should we place the chords corresponding to the vertices in $K_2$ so that it is also possible to draw the chords of those vertices in $S$ adjacent to $K_2$? and (2) in which portion of the circle model should we place both endpoints of the chords corresponding to vertices in $S_{22}$? We give a circle model for $G$ in Figure~\ref{fig:example_model0} 

Yet in this small example of a split graph that is circle, it becomes evident that there is a property that must hold for every pair of vertices in $S$ that have both of its endpoints placed within the same arc of the circumference. This led to the definition of nested matrices below, which was the first step towards expressing some of these problems regarding the placement of chords for a circle model in terms of certain properties of the adjacency matrix $A(S,K)$ (see Section~\ref{section:basic_defs} for the definition of $A(S,K)$). 

\begin{definition} \label{def:nested_m}
	Let $A$ be a $(0,1)$-matrix. We say $A$ is \emph{nested}~\cite{PDGS18} if it has the consecutive-ones property and every two rows are disjoint or nested.
\end{definition} 

\begin{definition} \label{def:nested_g}
A split graph $G = (K,S)$ is \emph{nested} if and only if $A(S,K)$ is a nested matrix.
\end{definition}

\begin{figure}[h!] 
	\centering  
	\[ 		\small{ \begin{pmatrix}
		 1 1 0 \cr
		0 1 1
		\end{pmatrix}} \]
	\caption{\mbox{The} $0$-gem \mbox{matrix}} \label{fig:forb_nested}
\end{figure}

Nested matrices were characterised in~\cite{PDGS18} by a single forbidden matrix. 

\begin{theorem}[\cite{PDGS18}] \label{teo:nested_caract}
A $(0,1)$-matrix is nested if and only if it contains no $0$-gem as a submatrix up to permutations of rows and/or columns (see Figure~\ref{fig:forb_nested}) 
\end{theorem}


Let us consider the matrix $A(S_{22},K_2)$ corresponding to the graph $G$ given in Figure~\ref{fig:example_graph0}. In this case, the rows are given by $s_1$, $s_2$, $s_3$ and $s_4$, and the columns are $k_{21}$, $k_{22}$, $k_{23}$ and $k_{24}$.
\vspace{-.5mm}
\[
A(S_{22},K_2) = \begin{pmatrix}
	1 1 0 0 \\
	1 1 1 0 \\
	0 1 1 0 \\
	1 0 0 0 
\end{pmatrix}
\]

Notice that the C$1$P for the matrix $A(S_{22},K_2)$ 
is a necessary condition to obtain an ordering of the vertices of $K_2$ that is compatible with the partial ordering given by containment for the vertices in $S_{22}$. 
Moreover, if the matrix $A(S_{22},K_2)$ is nested, then any two vertices in $S_{22}$ are either nested or disjoint. In other words, if $A(S_{22},K_2)$ is nested, then we can draw every chord corresponding to a vertex in $S\setminus V(H)$ in the same arc of the circumference. However, this is not the case in the previous example, for the vertices $s_1$ and $s_3$ are neither disjoint nor nested and thus they cannot be drawn in the same portion of the circle model. Equivalently, $A(S,K)$ is not a nested matrix. Since we have shown a circle model for $G$, it follows that the ``nestedness'' of $A(S,K)$ is not enough to determine whether or not there is a circle model for a given split graph $G=(K,S)$.

\begin{figure}[h!] 	
	\centering
	 \begin{subfigure}[b]{0.34\textwidth}
		\includegraphics[width=\textwidth]{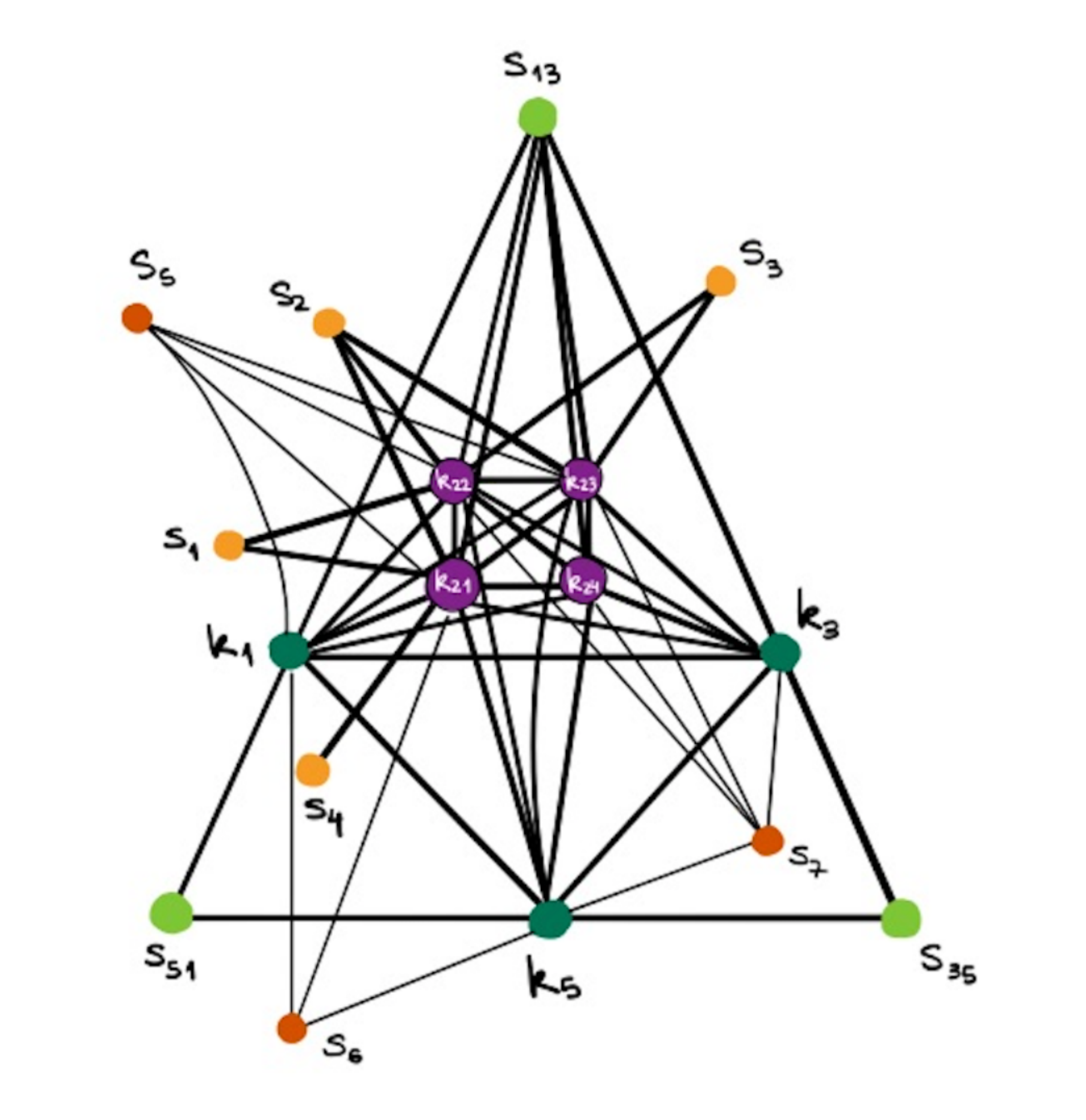}
		\caption{A split circle graph $G'$.}
		\label{fig:example_graph1}
	\end{subfigure}
	 \begin{subfigure}[b]{0.35\textwidth}
		\includegraphics[width=\textwidth]{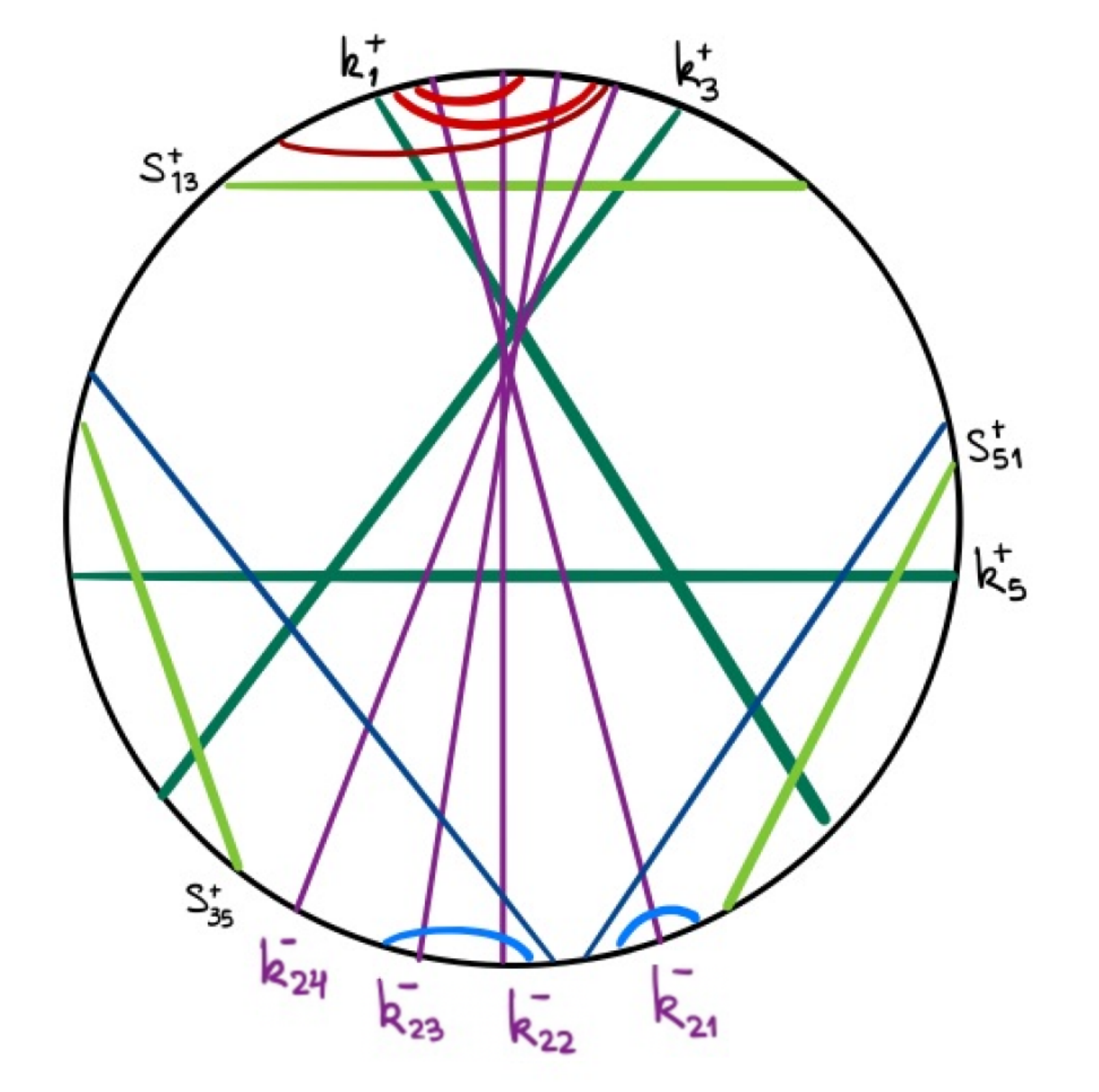}
		\caption{A circle model for $G'$.}
		\label{fig:example_model1}
	\end{subfigure}
	\caption{Graph and circle model of Example $2$.}
\end{figure}
Let us see one more example. Consider $G'$ to be the split graph depicted in Figure~\ref{fig:example_graph1}. 
Notice that $G'$ arises from $G$ by adding three new vertices.
Unlike what happens with the chords corresponding to $s_1$, $s_2$, $s_3$ and $s_4$, the chords that represent the new vertices $s_5$, $s_6$ and $s_7$ have only one of its endpoints in the arcs corresponding to the area of the circle designated for $K_2$, this is, in the arcs $k^+_1 k^+_3$ and $s_{51}^- s_{35}^+$. Furthermore, each of these new vertices has a unique possible placement 
for each endpoint of their corresponding chord --meaning that there exists a unique arc of the circumference in which we can place the endpoint. Let us consider $S' =S\setminus V(H)= \{ s_1, \ldots, s_7 \}$.
If we consider the rows given by the vertices in $S'$ and the columns given by $k_{21}, \ldots, k_{24}$, then the adjacency matrix $A(S',K_2)$ in this example is as follows:
\vspace{-.5mm}
\[
A(S',K_2) =\scriptsize{ \begin{pmatrix}
	1 1 0 0 \\
	1 1 1 0 \\
	0 1 1 0 \\
	1 0 0 0 \\
	1 1 1 0 \\
	1 0 0 0 \\
	0 1 1 1 
\end{pmatrix}}
\]

As in the previous example, $A(S',K_2)$ is not a nested matrix. Also notice that $s_5$, $s_6$ and $s_7$ all are adjacent to at least one vertex in $K \setminus K_2$. 
Let us concentrate in the placement of the endpoints of the chords corresponding to $s_5$, $s_6$ and $s_7$ that lie between the arcs $k_1 k_3$ and $s_{51} s_{35}$. Notice that the ``nested or disjoint'' property must still hold, and not only for those vertices in $K_2$. More precisely, since $s_5$ is adjacent to $k_{24}$, $k_{23}$ and $k_1$, while $s_1$ is nonadjacent to $k_1$ and adjacent to $k_{23}$ and $k_{24}$, then necessarily $s_1$ must be contained in $s_5$. Something similar occurs with $s_7$ and $s_3$, whereas $s_6$ and $s_3$ are disjoint. 

There is one situation in this example that did not occur in Example $1$. Since $s_6$ is adjacent to $k_{21}$, $k_1$ and $k_5$, the chord corresponding to the vertex $k_{21}$ is forced to be placed first within every chord corresponding to $K_2$. This follows from the fact that a chord that represents $s_6$ must have one of its endpoints placed inside the arc $s_{51}^- s_{35}^+$, for we need $k_{21}$ to be the first chord of $K_2$ that comes right after $s_{51}$. Moreover, this is once again confirmed by the fact that $s_5$ is adjacent to $k_1$ and $k_{21}$, thus the chord corresponding to the vertex $k_{21}$ must be drawn first when considering the ordering given by the neighbourhoods of those vertices in $S$ that have at least one endpoint lying in $k_1^+ k_3^+$. It follows that $k_{21}$ being the first vertex in the ordering is a necessary condition when searching for a consecutive-ones ordering for the matrix $A(S',K_2)$. See Figure~\ref{fig:example_model1}, where we give a circle model for the graph $G'$.

The previously described situations must also hold for each set $K_i$. For each $K_i\subseteq K$, there are exactly two arcs of the circumference in which we can place the endpoints of a chord corresponding to a vertex in $K_i$. Moreover, since these vertices lie in $K$, such a chord must have precisely one endpoint on each of these arcs. We denote these two arcs as $K_i^+$ and $K_i^-$. 

We translate the problem of giving a circle model to the fullfilment of some properties for each of the matrices $A(S_i,K_i)$, where $\{K_i\}$ is a partition of $K$ depending on some induced subgraph $H$ isomorphic to tent, $4$-tent or co-$4$-tent, as described in the introduction.
In order to express the corresponding properties in matrix terms, we define enriched matrices. 

\begin{definition} \label{def:enriched_matrix}
	An \emph{enriched matrix} $A$ is a $(0,1)$-matrix together with an assignment of labels and  colors to some (possibly none) of its rows such that all the following assertions hold:
	\begin{enumerate}
		\item Each row of $A$ is either unlabeled or labeled with one of the following labels: L or R or LR. We say that a row is an \emph{LR-row  (resp.\ L-row, R-row)} if it is labeled with LR (resp.\ L, R).
		\item Each row of $A$ is either uncolored or colored with either blue or red.
			\item The only rows that might be colored are those labeled with L or R, and those empty LR-rows. 
		\item All the empty LR-rows are colored with the same color. 
	\end{enumerate}
	
	The \emph{underlying matrix of an enriched matrix $A$} is the $(0,1)$-matrix with the same entries as $A$ but that has neither labeled nor colored rows.
\end{definition}

We denote the color assignment for a row with a colored bullet at the right side of the row.
The color assignment for some of the rows represents in which arc of the circle corresponding to $K_i$ we must draw one or both endpoints when considering the placement of the chords. Some of the vertices in $S$ have a unique possible placement for the endpoints of its chords 
, and some of them can \emph{a priori} be drawn in either two of the arcs $K_i^+$ and $K_i^-$. 
Moreover, the labeling of the rows indicates ``from which direction does the chord come from'' if we are standing in a particular portion of the circle. 
For example, the following is the matrix $A(S',K_2)$ for the circle model represented in Figure~\ref{fig:example_model1} turned into an enriched matrix by taking into account all the information on the placement of the chords in the model: 

\vspace{-5mm}
\[
A(S',K_2) = 
\scriptsize{ \bordermatrix{ &  \cr
	& 1 1 0 0 \cr
	& 1 1 1 0 \cr
	& 0 1 1 0 \cr
	& 1 0 0 0 \cr
	\textbf{L} & 1 1 1 0 \cr
	\textbf{L} & 1 0 0 0 \cr
	\textbf{R} & 0 1 1 1 }\,
	\begin{matrix} 
   \cr  \cr \cr  \cr \textcolor{dark-red}{\bullet} \cr \textcolor{blue}{\bullet} \cr \textcolor{blue}{\bullet} 
\end{matrix}}
\]
\vspace{.5mm}

In this example, the rows of the enriched matrix are $s_1, \ldots, s_7$, in that order. Notice that $s_5$ is adjacent to $k_1 \in K_1$ and some vertices in $K_2$, thus $s_5$ lies in $S_{12}$. The chords of the vertices in $S_{12}$ can only have its endpoints placed in the arcs $s_{13}^+ k_1^+$ and $k_1^+ k_3^+$. More precisely, in terms of intersecting the chords of $K_2$, those vertices in $S_{12}$ have one of its endpoints placed in $K_2^+$. Something similar happens with the vertex $s_6 \in S_{52}$, whose endpoint intersecting chords of $K_2$ must be placed somewhere in the arc $K_2^-$.
Hence, we need to assign distinct colors to this rows in order to denote in which of $K_2^+$ or $K_2^-$ we are able to draw the endpoint of the chord that should intersect its neighbours in $K_2$.
Furthermore, notice that both rows are labeled with L in $A(S',K)$. If we stand in $K_2^+$ and consider from which direction should the chord corresponding to $s_5$ come from, then we notice that it should come from the left. The same holds for $s_6$ if we stand in $K_2^-$. On the other hand, the chord corresponding to $s_7$ should come from the right, and this is why these three rows are labeled with L or R.


\begin{definition} \label{def:LR-orderable}
	Let $A$ be an enriched matrix. We say $A$ is \emph{LR-orderable} if there is a linear ordering $\Pi$ for the columns of $A$ such that each of the following assertions holds:
	\begin{itemize}
	    \item $\Pi$ is a consecutive-ones ordering for every non-LR-row of $A$.
	    
		\item The ordering $\Pi$ is such that the $1$'s in every nonempty row labeled with L (resp.\ R) start in the first column (resp.\ end in the last column). 
		
		\item $\Pi$ is a consecutive-ones ordering for the complements of every LR-row of $A$.
     \end{itemize} 
Such an ordering is called an \emph{LR-ordering}. For each row of $A$ labeled with L or LR and having a $1$ in the first column of $\Pi$, we define its \emph{L-block (with respect to $\Pi$)} as the maximal set of consecutive columns of $\Pi$ on which the row has a 1 starting in the first column. \emph{R-blocks} are defined on an entirely analogous way.
For each unlabeled row of $A$, we say its \emph{U-block (with respect to $\Pi$)} is the set of columns having a $1$ in the row.
The blocks of $A$ with respect to $\Pi$ are its L-blocks, its R-blocks and its U-blocks. 
\end{definition} 

\begin{definition} \label{def:block-bicoloring}
	Let $A$ be an enriched matrix with an LR-ordering. 
We say an \emph{L-block (resp.\ R-block, U-block) is colored} if there is a color assignment for every entry of the block. 
	
	A \emph{basic bi-coloring for the blocks of $A$} is a color assignment with either red or blue for some L-blocks, U-blocks and R-blocks of $A$.
		A basic bi-coloring is \emph{total }if every L-block, R-block and U-block of $A$ is colored, and is \emph{partial }if otherwise.
\end{definition}

Recall that, according to the definition, in an enriched matrix, the only rows that might be colored are those labeled with L or R and those empty LR-rows. Moreover, in an LR-ordering, every row labeled with L (resp.\ R) starts in the first column (resp.\ ends in the last column), and the $1$'s in each row appear consecutively.  
Thus, if an enriched matrix is also LR-orderable, then the given coloring induces a partial basic bi-coloring (see Figure~\ref{fig:example_LR-ord_a}) 
, in which every empty LR-row remains unchanged, whereas for every nonempty colored labeled row, we color all its $1$'s with the color given in the definition of the matrix.

\begin{figure}[h!]
    \centering
	\begin{subfigure}[b]{0.28\textwidth}
		\begin{align*}
			A = \scriptsize{ \bordermatrix{ &  \cr
			\textbf{LR} & 1  0  0  0  1 \cr
			\textbf{LR} & 1  1  0  0  1 \cr
							& 0  1 1  0  0 \cr
			\textbf L & \textcolor{dark-red}{1}  \textcolor{dark-red}{1}  \textcolor{dark-red}{1}  0  0 \cr
			\textbf{LR} & 0  0  0   0  0 \cr
			\textbf R & 0  0  \textcolor{blue}{1} \textcolor{blue}{1} \textcolor{blue}{1} }\,
			\begin{matrix} 
		   \cr  \cr  \cr \textcolor{dark-red}{\bullet} \cr \textcolor{blue}{\bullet} \cr \textcolor{blue}{\bullet} 
		\end{matrix}}
		\end{align*}
		\caption{An enriched LR-orderable matrix $A$.} 
		\label{fig:example_LR-ord_a}
	\end{subfigure}
	\begin{subfigure}[b]{0.28\textwidth}
		\begin{align*}
			A' = \scriptsize{ \bordermatrix{  &  \cr
			\textbf{LR} &  \textcolor{blue}{1}  0  0  0  \textcolor{dark-red}{1} \cr
			\textbf{LR} &  \textcolor{blue}{1}   \textcolor{blue}{1}  0  0   \textcolor{dark-red}{1} \cr
							 & 0  \textcolor{dark-red}{1}  \textcolor{dark-red}{1}  0  0 \cr
			\textbf L &  \textcolor{dark-red}{1}   \textcolor{dark-red}{1}   \textcolor{dark-red}{1}  0  0 \cr
			\textbf{LR} &  0   0   0   0   0 \cr
			\textbf R & 0  0  \textcolor{blue}{1}  \textcolor{blue}{1}  \textcolor{blue}{1} } \,
			\begin{matrix}
			\\ \\  \\  \textcolor{dark-red}{\bullet} \\ \textcolor{blue}{\bullet} \\ \textcolor{blue}{\bullet} \\
			\end{matrix}}
		\end{align*}
		\caption{A total block bi-coloring of the blocks of $A$.} 
		\label{fig:example_LR-ord_b}
	\end{subfigure}
	\begin{subfigure}[b]{0.28\textwidth}
		\begin{align*}
			B = \scriptsize{ \bordermatrix{ & \cr
			\textbf{LR} & 1  0  1  0 1 \cr
			\textbf L & \textcolor{dark-red}{1}  \textcolor{dark-red}{1}  0  0  0 \cr
			\textbf R & 0  0  0  \textcolor{blue}{1}  \textcolor{blue}{1} \cr
						& 0  0  1  1  0 }\,
			\begin{matrix} 
		   \cr  \textcolor{dark-red}{\bullet} \cr \textcolor{blue}{\bullet} \cr \cr
		\end{matrix}}
		\end{align*}
		\caption{An enriched non-LR-orderable matrix $B$.} 
		\label{fig:example_LR-ord_c}
	\end{subfigure}
\caption{Examples of enriched LR-orderable and non-LR-orderable matrices.}
\label{fig:example_LR-ord} 
\end{figure}

We now define $2$-nested matrices, which allow us to solve both the problem of ordering the columns in each adjacency matrix $A(S_i,K_i)$ of a split graph for each set $K_i$, and the problem of deciding if there is a feasible distribution of the vertices in $S_i$ between the two arcs $K_i^+$ and $K_i^-$. 
This allows to obtain a circle model for the given graph. We give a complete characterization of $2$-nested matrices by forbidden subconfigurations at the end of Section~\ref{section:2nested_matrices}. 

\begin{definition} \label{def:2-nested}
	Let $A$ be an enriched matrix. We say $A$ is \emph{$2$-nested} if there exists an LR-ordering $\Pi$ of $A$ and an assignment of colors red or blue to the blocks of $A$ (with respect to $\Pi$) such that all of the following conditions hold:
\begin{enumerate}
	\item If an LR-row has an L-block and an R-block, then they are colored with distinct colors. \label{item:2nested1}
	\item For each nonempty colored row $r$ in $A$, its only block is colored with the same color as $r$ in $A$. \label{item:2nested2}
	\item If an L-block of an LR-row is properly contained in the L-block of an L-row, then both blocks are colored with different colors. \label{item:2nested3}
	\item Every L-block of an LR-row and any R-block are disjoint. The same holds for an R-block of an LR-row and any L-block. \label{item:2nested4} 
	\item If an L-block and an R-block are not disjoint, then they are colored with distinct colors.	\label{item:2nested5} 
	\item Each two U-blocks colored with the same color are either disjoint or nested. \label{item:2nested6}
	\item If an L-block and a U-block are colored with the same color, then either they are disjoint or the U-block is contained in the L-block. The same holds replacing L-block for R-block. \label{item:2nested7}
	\item If two distinct L-blocks of non-LR-rows are colored with distinct colors, then every LR-row has an L-block. The same holds replacing L-block for R-block.  \label{item:2nested8} 
	\item If two LR-rows overlap, then the L-block of one and the R-block of the other are colored with the same color. \label{item:2nested9}
	
\end{enumerate}	

An assignment of colors red and blue to the blocks of $A$ that satisfies all these properties is called a \emph{total block bi-coloring}. 
\end{definition}

\begin{remark}
We now give some insight on what properties are behind each assertion in Definition~\ref{def:2-nested}. All these properties are necessary conditions for each matrix $A(S_i,K_i)$, in order to give a circle model for any split graph containing a tent, $4$-tent or co-$4$-tent.

The LR-rows represent those independent vertices whose chords have its endpoints placed inside distinct arcs corresponding to $K_i$. 
More precisely, the difference between an LR-row and an unlabeled row, is that one endpoint of the chord corresponding to an LR-row must be placed in $K_i^+$ and the other in $K_i^-$, whereas for an unlabeled row, either both endpoints lie in $K_i^+$ or in $K_i^-$.
Hence, assertion~\ref{item:2nested1} of Definition~\ref{def:2-nested} ensures that, when deciding where to place the chord corresponding to an LR-row, if the ordering indicates that the chord intersects some of its adjacent vertices in one arc and the other in the other arc, then the distinct blocks corresponding to the row must be colored with distinct colors. 

With assertion~\ref{item:2nested2} of Definition~\ref{def:2-nested}, we ensure that the colors that are pre-assigned remain unchanged, since they correspond to vertices in $S$ whose chords admit a placement of the endpoint corresponding to $K_i$ in either $K_i^+$ or $K_i^-$ and not in both indistinctly. 

The third property refers to the ordering given by containment for the vertices. In~\cite{BDPS20,P20}, we see that every LR-row represents a vertex that is adjacent to almost every vertex in the complete partition $K$ of $G$. Hence, when dividing the LR-rows into blocks, we need to ensure that each of its block is not properly contained in the neighbourhoods of vertices that are nonadjacent to at least one partition of $K$. Something similar holds for L-rows (resp.\ R-rows) and U-rows, and L-rows (resp.\ R-rows) and LR-rows. This property is reflected in assertions~\ref{item:2nested7} and~\ref{item:2nested8}.

Assertions~\ref{item:2nested4},~\ref{item:2nested5},~\ref{item:2nested6} and~\ref{item:2nested9} of the definition refer to the previously discussed ``nested or disjoint'' property that we need to ensure in order to give a circle model for $G$.
\end{remark}


\section{Characterization by forbidden subconfigurations} \label{section:2nested_matrices}

In this section, we begin giving some definitions that are necessary to state Theorem~\ref{teo:2-nested_charact}, which is presented at the end of this section and is the main result of this work. In Section~\ref{subsec:admissibility} we define and characterise admissible matrices, which give necessary conditions for a matrix to admit a total block bi-coloring. In Section~\ref{subsec:part2nested} we define and characterise LR-orderable and partially $2$-nested matrices, and then we prove some properties of LR-orderings in admissible matrices. Finally, in Section~\ref{subsec:teo_2nestedmatrices} we give the proof of Theorem~\ref{teo:2-nested_charact}, which characterises $2$-nested matrices by forbidden subconfigurations.

\begin{definition}\label{def:subconfiguration}
Let $A$ and $B$ be enriched matrices. We say that \emph{$B$ is a subconfiguration of $A$} if $B$ equals some submatrix of $A$ up to permutations of rows and/or columns and such that the labels and colors remain the same.
Given a subset of rows $R$ (resp.\ of columns $C$) of $A$, we say that \emph{$R$ (resp.\ $C$) induces a matrix $B$} if $B$ is a subconfiguration of the submatrix of $A$ formed by the rows in $R$ (resp.\ the columns in $C$).

Let $\mathcal{F}$ be a family of enriched matrices. We say that $A$ is $\mathcal{F}$-free if $A$ contains no $F$ as a subconfiguration, for every $F \in \mathcal{F}$.
\end{definition}

The $0$-gem, $1$-gem and $2$-gem are the following enriched matrices:
\[ 		\scriptsize{ \begin{pmatrix}
		 1 1 0 \cr
		0 1 1
		\end{pmatrix}, \qquad
		\bordermatrix{ & \cr
		 & 1 0 \cr
		 & 1 1 }\ , \qquad
		\bordermatrix{ & \cr
		\textbf{LR} & 1 1 0 \cr
		\textbf{LR} & 1 0 1  }\ } \]
respectively.


\begin{definition}  \label{def:gems}
	Let $A$ be an enriched matrix. We say that $A$ \emph{contains a gem} (resp.\ \emph{doubly-weak gem}) if it contains a $0$-gem (resp.\ a $2$-gem) as a subconfiguration.
	We say that $A$ \emph{contains a weak gem} if it contains a $1$-gem such that, either the first 
	corresponds to an L-row (resp.\ R-row) of $A$ and the second 
	corresponds to a U-row of $A$, or the first corresponds to an LR-row of $A$ and the second corresponds to a non-LR-row of $A$.
	We say that a $2$-gem is \emph{badly-colored} if the entries in the corresponding column of $A$ in which both rows have a $1$ are in blocks colored with the same color.
\end{definition}

\begin{definition} \label{def:dual}
Let $A$ be an enriched matrix. The dual matrix of $A$ is defined as the enriched matrix $\tilde{A}$ that has the same underlying matrix as $A$ and for which every row of $A$ that is labeled with L (resp.\ R) is now labeled with R (resp.\ L) and every other row remains the same. Also, the color or lack of color assignment of each row remains as in $A$.
\end{definition}

In Figures~\ref{fig:forb_D},~\ref{fig:forb_F},~\ref{fig:forb_S},~\ref{fig:forb_P} and~\ref{fig:forb_LR-orderable} we define some special matrices that play an important role in the sequel. 
We will use green and orange to represent either red and blue or blue and red, respectively. For every enriched matrix represented in the figures of this section, if a row labeled with L or R appears in black, then it may be colored with either red or blue indistinctly.
Also, if a row is labeled with ``\textbf{L (LR)}'' (resp.\ ``\textbf{R (LR)}''), then such a row is a row labeled with either L or LR (resp.\ R or LR) indistinctly and independently from one another. 

\vspace{-7mm}
\begin{figure}[H]
    \centering
    \footnotesize{
    \begin{align*}
            M_0 = \scriptsize{ \bordermatrix{ & \cr
                         & 1 0 1 1 \cr
                         & 1 1 1 0  \cr
                         & 0 1 1 1 }\ }
            &&
            M_{II}(4) = \scriptsize{ \bordermatrix{ & \cr
                         & 0 1 1 1 \cr
                         & 1 1 0 0  \cr
                         & 0 1 1 0 \cr
                         & 1 1 0 1 }\ }
            &&
            M_V = \scriptsize{ \bordermatrix{ & \cr
                         & 1 1 0 0 0 \cr
                         & 0 0 1 1 0  \cr
                         & 1 1 1 1 0 \cr
                         & 1 0 0 1 1 }\ }
            &&
            S_0(k)&= \scriptsize{ \begin{pmatrix}
                1 11...11\\
                110...00\\
                011...00\\
                .   .   .   .   . \\
                .   .   .   .   . \\
                .   .   .   .   . \\
                000...11\\
                100...01\\
            \end{pmatrix} }
    \end{align*}}
    \caption{The matrices $M_0$, $M_{II}(4)$, $M_V$ and $S_0(k) \in \{0,1\}^{((k+1)\times k}$ for any even $k \geq 4$.}
    \label{fig:forb_M_chiquitas}
    \end{figure}


\begin{figure}[H]
    \centering
    \includegraphics[scale=.33]{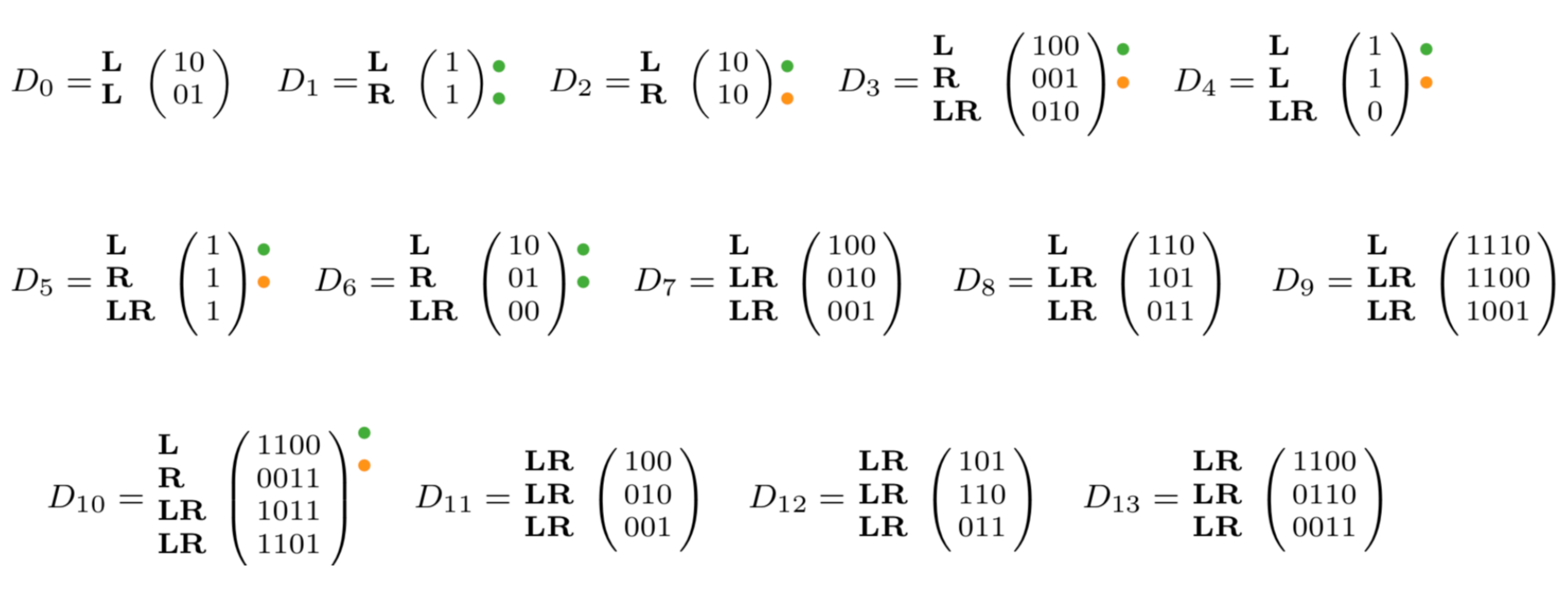}
    \caption{The family of enriched matrices $\mathcal{D}$.}
    \label{fig:forb_D}
\end{figure}

\begin{figure}[H]
\centering
\footnotesize{
    \begin{align*}
            F_0= \scriptsize{  \begin{pmatrix}
                11100\\
                01110\\
                00111\\
            \end{pmatrix} }
            &&
            F_1(k)= \scriptsize{ \begin{pmatrix}
                011...111\\
                111...110\\
                000...011\\
                000...110\\
                .   .   .   .   . \\
                .   .   .   .   . \\
                .   .   .   .   . \\
                110...000\\
            \end{pmatrix} }
            &&
            F_2(k)= \scriptsize{  \begin{pmatrix}
                0111...10\\
                1100...00\\
                0110...00\\
                .   .   .   .   . \\
                .   .   .   .   . \\
                .   .   .   .   . \\
                0000...11\\
            \end{pmatrix}  }
            &&
              F'_0=  \scriptsize{ \bordermatrix{ & \cr
                \textbf{L (LR)} & 1 1 0 0 \cr
                 &1 1 1 0 \cr
                &0 1 1 1 }\ }
            \end{align*}
            \begin{align*}
            F''_0= \scriptsize{ \bordermatrix{ & \cr
                \textbf{L} & 1 1 0 \cr
                 &1 1 1 \cr
                \textbf{R} &0 1 1 }\     }
            &&
            F'_1(k)= \scriptsize{  \bordermatrix{ & \cr
                &11\ldots1111\cr
                \textbf{L (LR)}&11\ldots1110\cr
                &00\ldots0011\cr
                &00\ldots0110\cr
                & \iddots \cr
                \textbf{L (LR)}&10 \ldots 0000 }\ }
            &&
            F'_2(k)= \scriptsize{ \bordermatrix{ & \cr
                &111\ldots10\cr
                \textbf{L (LR)}&100\ldots00\cr
                &110\ldots00\cr
                &  \ddots \cr
                &000\ldots11 }\ }
        \end{align*}}
        \vspace{-2mm}
    \caption{The enriched matrices of the family $\mathcal{F}$.}  \label{fig:forb_F}
\end{figure}

The matrices $\mathcal{F}$ represented in Figure~\ref{fig:forb_F} are defined as follows:
$F_1(k) \in \{0,1\}^{k \times (k-1)}$, $F_2(k) \in \{0,1\}^{k \times k}$, $F'_1(k) \in \{0,1\}^{k \times (k-2)}$ and $F'_2(k) \in \{0,1\}^{k \times (k-1)}$, for every odd $k \geq 5$. In the case of $F'_0$, $F'_1(k)$ and $F'_2(k)$, the labeled rows may be either L or LR indistinctly and independently from one another, 
and in the case of their dual matrices, the labeled rows may be either R or LR indistinctly and independently from one another.


The matrices $\mathcal{S}$ in Figure~\ref{fig:forb_S} are defined as follows.
If $k$ is odd, then $S_1(k) \in \{0,1\}^{(k+1) \times k}$ for $k \geq 3$, and if $k$ is even, then $S_1(k) \in \{0,1\}^{k \times (k-2)}$ for $k\geq 4$.
The remaining matrices have the same size whether $k$ is even or odd: $S_2(k) \in \{0,1\}^{k \times (k-1)}$ for $k \geq 3$,
$S_3(k) \in \{0,1\}^{k \times (k-1)}$ for $k \geq 3$,
$S_5(k) \in \{0,1\}^{k \times (k-2)}$ for $k \geq 4$,
$S_4(k) \in \{0,1\}^{k \times (k-1)}$, $S_6(k) \in \{0,1\}^{k \times k}$ for $k\geq 4$,
$S_7(k) \in \{0,1\}^{k \times (k+1)}$ for every $k \geq 3$
and $S_8(2j) \in \{0,1\}^{2j \times (2j)}$ for $j \geq 2$.
If $k$ is even, then the first and last row of $S_2(k)$ and $S_3(k)$ are colored with the same color, whereas in $S_4(k)$ and $S_5(k)$ are colored with distinct colors.

\begin{figure}[H]
    \centering
    \footnotesize{
    \begin{align*}
            S_1(2j) &= \scriptsize{ \bordermatrix{ & \cr
                \textbf{L}& 1 0 \ldots 00  \cr
                              & 1 1 \ldots 00 \cr
                               &   \ddots  \cr
                               & 0 0 \ldots 11 \cr
                \textbf{LR} &0 0  \ldots 01 \cr
                \textbf{L} & 1  1  \ldots  11 }\ }
            &
            S_1(2j+1) &= \scriptsize{  \bordermatrix{ & \cr
                \textbf{L} & 1 0 \ldots 0 0 \cr
                            & 1 1 \ldots 0 0 \cr
                            & \ddots \cr
                            & 0 0 \ldots 1 1 \cr
                \textbf{LR} &0 0 \ldots 0 1 }\ }
            &
        S_2(k) &= \scriptsize{ \bordermatrix{ & \cr
                \textbf{L}& 1 0 \ldots 0 0 \cr
                            & 1 1  \ldots 0 0 \cr
                            &  \ddots \cr
                            &0 0 \ldots 1 1 \cr
                \textbf{L} & 1 1 \ldots 1 0 }\
                \begin{matrix}
             \textcolor{dark-orange}{\bullet} \\ \\ \\ \\ \\ \textcolor{dark-green}{\bullet} \\
                \end{matrix}  }
        \end{align*}
        \begin{align*}
         S_3(k) &=  \scriptsize{ \bordermatrix{ & \cr
             \textbf{L} & 1 0 \ldots 0 0 \cr
                              & 1 1  \ldots 0 0 \cr
                              &  \ddots \cr
                              &0 0 \ldots 1 1 \cr
            \textbf{R} & 0 0 \ldots 0 1 }\
                \begin{matrix}
            \textcolor{dark-orange}{\bullet} \\ \\ \\ \\ \\ \textcolor{dark-green}{\bullet} \\
                \end{matrix} }
                &
        S_4(k) &= \scriptsize{ \bordermatrix{ & \cr
        \textbf{LR}& 1 1 \ldots 1 1 \cr
        \textbf{L} & 1 0 \ldots 0 0 \cr
                        & 1 1  \ldots  0 0 \cr
                        &  \ddots \cr
                        & 0 0  \ldots 1 1 \cr
          \textbf{R}& 0 0 \ldots 0 1  }\
                \begin{matrix}
            \\ \textcolor{dark-orange}{\bullet} \\ \\ \\ \\ \\ \textcolor{dark-green}{\bullet} \\
                \end{matrix} }
        &
        S_5(k) &= \scriptsize{ \bordermatrix{ & \cr
        \textbf{L}& 1 0 \ldots 0 0 \cr
                      & 1 1 \ldots 0 0 \cr
                      &  \ddots \cr
                      & 0 0  \ldots 1 1 \cr
        \textbf{LR}& 1 1 \ldots 1 0 \cr
        \textbf{L}& 1 1 \ldots 1 1  }\
                \begin{matrix}
             \textcolor{dark-orange}{\bullet} \\ \\ \\ \\ \\ \\  \textcolor{dark-green}{\bullet} \\
                \end{matrix} }
        \end{align*}
        \begin{align*}
        S_6(3) &= \scriptsize{ \bordermatrix{ & \cr
        \textbf{LR} & 1 1 0 \cr
          \textbf{R} & 0 1 1  \cr
                          & 1 1 0 }\ }
        &
        S_6'(3) &= \scriptsize{ \bordermatrix{ & \cr
        \textbf{LR} & 1 1 0 \cr
          \textbf{R} & 0 1 1  \cr
                          & 1 1 1 }\ }
         &
        S_6(k) &= \scriptsize{ \bordermatrix{ & \cr
        \textbf{LR}& 1 1 1 \ldots 1 1 0 \cr
          \textbf{R}& 0 1 1 \ldots 1 1 1 \cr
                        & 1 1 0  \ldots 0 0 0  \cr
                        &  \ddots \cr
                        & 0 0 0 \ldots 0 1 1 }\
                \begin{matrix}
            \\ \\ \textcolor{dark-orange}{\bullet} \\ \\ \\ \\ \\ \\
                \end{matrix} }
    \end{align*}
    \begin{align*}
        S_7(3) &= \scriptsize{ \bordermatrix{ & \cr
        \textbf{LR} & 1 1 0 0 1 \cr
        \textbf{LR} & 1 0 0 1 1  \cr
                         & 1 1 1 0 0 }\ }
        &
        S_7(2j) &= \scriptsize{ \bordermatrix{ & \cr
        \textbf{LR}& 1 1 0 0 \ldots 0 0 0 \cr
        \textbf{LR}& 1 0 0 0  \ldots 0 0 1 \cr
                        & 0 1 1 0 \ldots 0 0 0  \cr
                        &  \ddots \cr
                        & 0 0 0 0 \ldots 0 1 1 }\ }
        &
        S_8(2j) &= \scriptsize{ \bordermatrix{ & \cr
        \textbf{LR}& 1 0 0 \ldots 0 0 1 \cr
                        & 1 1 0  \ldots 0 0 0 \cr
                        &  \ddots \cr
                        & 0 0 0 \ldots 0 1 1 }\ }
    \end{align*}}
\caption{The family of matrices $\mathcal{S}$ for every $j\geq2$ and every odd $k \geq 3$} \label{fig:forb_S}
\end{figure}


In the matrices $\mathcal{P}$, the integer $l$ represents the number of unlabeled rows between the first row and the first LR-row. The matrices $\mathcal{P}$ described in Figure~\ref{fig:forb_P} are defined as follow:
$P_0(k,0) \in \{0,1\}^{k \times k}$ for every $k \geq 4$, $P_0(k,l) \in \{0,1\}^{k \times (k-1)}$ for every $k \geq 5$ and $l >0$;
$P_1(k,0) \in \{0,1\}^{k \times (k-1)}$ for every $k \geq 5$, $P_1(k,l) \in \{0,1\}^{k \times (k-2)}$ for every $k \geq 6$, $l > 0$;
$P_2(k,0) \in \{0,1\}^{k \times (k-1)}$ for every $k \geq 7$, $P_2(k,l) \in \{0,1\}^{k \times (k-2)}$ for every $k \geq 8$ and $l > 0$.
If $k$ is even, then the first and last row of every matrix in $\mathcal{P}$ are colored with distinct colors.

\begin{figure}[H]
    \centering
    \footnotesize{
        \begin{align*}
        P_0(k,0) =  \scriptsize{ \bordermatrix{ & \cr
        \textbf{L} & 1 1 0 0 0  \ldots 0 0 0   \cr
        \textbf{LR} & 1 0 0 1 1 \ldots 1 1 1 \cr
                         & 0 0  1  1  0  \ldots 0 0 0 \cr
                         &  \ddots  \cr
                        & 0 0 0 0 0 \ldots  0 1 1  \cr
        \textbf{R} & 0 0 0 0 0 \ldots 0 0 1  }\
     \begin{matrix}
      \textcolor{dark-green}{\bullet} \\ \\ \\ \\ \\ \\ \textcolor{dark-green}{\bullet}
      \end{matrix} }
        &&
        P_0(k,l) = \scriptsize{  \bordermatrix{ & \cr
        \textbf{L}& 1 0 0 \ldots 0 0 0 0 \ldots 0 \cr
                      & 1 1 0 \ldots 0 0 0 0  \ldots 0 \cr
                      &  \ddots \cr
                     & 0 0 0 \ldots 1 1 0 0 \ldots 0 \cr
    \textbf{LR} & 1 1 1 \ldots 1 0 0 1 \ldots 1 \cr
                     & 0 0 0 \ldots 0 0 1 1  \ldots 0 \cr
                     & \ddots \cr
                    & 0 0 0 \ldots 0 0 \ldots 0 1 1 \cr
        \textbf{R} & 0 0 0 \ldots 0 0 \ldots 0 0 1 }\
    \begin{matrix}
      \textcolor{dark-green}{\bullet} \\ \\ \\ \\ \\ \\ \\ \\ \\ \\ \textcolor{dark-green}{\bullet}
      \end{matrix} }
       \end{align*}
    \begin{align*}
        P_1(k,0) = \scriptsize{ \bordermatrix{ & \cr
        \textbf{L} & 1 1 0 0  \ldots 0 0 0   \cr
        \textbf{LR} & 1 0 1 1 \ldots 1 1 1 \cr
        \textbf{LR} & 1 1 0 1 \ldots 1 1 1 \cr
                         & 0 0  1  1  0  \ldots 0 0 0 \cr
                         &  \ddots  \cr
                        & 0 0 0 0 0 \ldots  0 1 1  \cr
        \textbf{R} & 0 0 0 0 \ldots 0 0 1  }\
     \begin{matrix}
      \textcolor{dark-green}{\bullet} \\ \\ \\ \\ \\ \\ \\ \textcolor{dark-green}{\bullet}
      \end{matrix} }
    &&
            P_1(k,l) = \scriptsize{ \bordermatrix{ & \cr
        \textbf{L}& 1 0 0 \ldots 0 0 0 0 \ldots 0 \cr
                      & 1 1 0 \ldots 0 0 0 0  \ldots 0 \cr
                      &  \ddots \cr
                     & 0 0 0 \ldots 1 1 0 0 \ldots 0 \cr
    \textbf{LR} & 1 1 1 \ldots 1 0 1 1 \ldots 1 \cr
    \textbf{LR} & 1 1 1 \ldots 1 1 0 1 \ldots 1 \cr
                     & 0 0 0 \ldots 0 0 1 1  \ldots 0 \cr
                     & \ddots \cr
                    & 0 0 0 \ldots 0 0 \ldots 0 1 1 \cr
        \textbf{R} & 0 0 0 \ldots 0 0 \ldots 0 0 1 }\
    \begin{matrix}
      \textcolor{dark-green}{\bullet} \\ \\ \\ \\ \\ \\ \\ \\ \\ \\ \\ \textcolor{dark-green}{\bullet}
      \end{matrix} }
       \end{align*}
    \begin{align*}
        P_2(k,0) =  \scriptsize{ \bordermatrix{ & \cr
        \textbf{L} & 1 1 0 0 0 0 \ldots 0 0 0 \cr
        \textbf{LR} & 1 0 1 1 1 1 \ldots 1 1 1 \cr
        \textbf{LR} & 1 1 1 0 1 1 \ldots 1 1 1 \cr
        \textbf{LR} & 1 1 0 1 1 1 \ldots 1 1 1 \cr
        \textbf{LR} & 1 1 1 0 0 1 \ldots 1 1 1 \cr
                         & 0 0  0 0 1  1 \ldots 0 0 0 \cr
                         &  \ddots  \cr
                        & 0 0 0 0 0 \ldots  0 1 1  \cr
        \textbf{R} & 0 0 0 0 0 \ldots 0 0 1  }\
     \begin{matrix}
      \textcolor{dark-green}{\bullet} \\ \\ \\ \\ \\ \\ \\ \\ \\ \\ \textcolor{dark-green}{\bullet}
      \end{matrix} }
    &&
            P_2(k,l) = \scriptsize{ \bordermatrix{ & \cr
        \textbf{L}& 1 0 0 \ldots 0 0 0 0 0 \ldots 0 \cr
                      & 1 1 0 \ldots 0 0 0 0 0  \ldots 0 \cr
                      &  \ddots \cr
                     & 0 0 0 \ldots 1 1 0 0 0  \ldots 0 \cr
    \textbf{LR} & 1 1 1 \ldots 1 0 0 1 1 \ldots 1 \cr
    \textbf{LR} & 1 1 1 \ldots 1 1 1 0 1 \ldots 1 \cr
    \textbf{LR} & 1 1 1 \ldots 1 1 0 1 1 \ldots 1 \cr
    \textbf{LR} & 1 1 1 \ldots 1 1 0 0 1 \ldots 1 \cr
                     & 0 0 0 \ldots 0 0 0 1 1 \ldots 0 \cr
                     & \ddots \cr
                    & 0 0 0 \ldots 0 0 0 \ldots 0 1 1 \cr
        \textbf{R} & 0 0 0 \ldots 0 0 0  \ldots 0 0 1 }\
    \begin{matrix}
      \textcolor{dark-green}{\bullet} \\ \\ \\ \\ \\ \\ \\ \\\ \\ \\ \\ \\ \\ \\ \textcolor{dark-green}{\bullet}
      \end{matrix} }
    \end{align*}}
    \caption{The family of enriched matrices $\mathcal{P}$ for every odd $k$.}
    \label{fig:forb_P}
    \end{figure}


\begin{figure}[h]
    \centering
  \footnotesize{
    \begin{align*}
            M_2'(k) = \scriptsize{ \bordermatrix{ & \cr
                        & 1 1 1 \ldots 1 1 1\cr
           \textbf L &  1 0 0 \ldots 0 0 0\cr
                        & 1  1 0  \ldots  0 0 0\cr
                        & \ddots  \cr
                        & 0 0 0 \ldots 1 1 0 \cr
           \textbf L & 1 1 1 \ldots 1 0 1 }\ }
            &&
            M_2''(k) = \scriptsize{ \bordermatrix{ & \cr
        \textbf R & 1  1  1 \ldots  1 1 1\cr
        \textbf L & 1  0 0 \ldots  0 0 0\cr
                    & 1  1 0  \ldots  0 0 0\cr
                        & \ddots \cr
                        & 0 0 0  \ldots  1 1 0 \cr
         \textbf R &  0 0 0 \ldots  0 1  0 \cr
     \textbf{L} & 1  1 1 \ldots  1 1  1  }\  }
             \end{align*}
            \begin{align*}
            M_3'(k) = \scriptsize{ \bordermatrix{ & \cr
            \textbf L &  1 0  0   \ldots   0   0   0 \cr
                &   1  1 0   \ldots   0   0   0 \cr
                & \ddots \cr
                & 0 0  0   \ldots   1   1   0 \cr
                & 1 1 1   \ldots   1   0   1  }\     }
            &&
            M_3''(k) = \scriptsize{ \bordermatrix{ & \cr
                & 1   1  0  \ldots  0 0  \cr
                & 0  1  1   \ldots  0 0  \cr
                & \ddots \cr
                & 0 0 0   \ldots  1 1 \cr
                \textbf R & 0 1 1  \ldots  1  0  }\ }
            &&
            M_4' =  \scriptsize{ \bordermatrix{ & \cr
                    \textbf L & 1  0  0  0 0 \cr
                    & 0  1   1   0  0 \cr
                    & 0  0  0  1   1 \cr
                    & 1  0  1  0  1 }\ }
            \end{align*}
            \begin{align*}
            M_4'' = \scriptsize{  \bordermatrix{ & \cr
                    \textbf L & 1  0  0  0 \cr
                    \textbf R & 0  1  0  0 \cr
                                & 0  0  1  1 \cr
                                 &  1  1   0   1 }\ }
            &&
            M_5' = \scriptsize{ \bordermatrix{ & \cr
                    & 1   1   0   0   \cr
                    & 0   0   1   1   \cr
                    \textbf R & 1   0   0   1  \cr
                    & 1   1   1   }\ }
            &&
            M_5'' = \scriptsize{ \bordermatrix{ & \cr
                    \textbf L & 1  0   0   0 \cr
                    & 0  1   1   0 \cr
                    & 1   0   1   1 \cr
                    \textbf L & 1  1   1  0 }\ }
    \end{align*}
       }
    \caption{The enriched matrices in family $\mathcal{M}$: $M_{2}'(k)$, $M_{3}'(k)$, $M_{3}''(k)$, $M_{3}'''(k)$ for $k \geq 4$, and $M_{2}''(k)$ for $k \geq 5$. }  \label{fig:forb_LR-orderable}
\end{figure}


\begin{definition} \label{def:A*}
	Let $A$ be an enriched matrix and let $\Pi$ be a LR-ordering. 
	We define $A^*$ as the enriched matrix that arises from $A$ by:
\begin{itemize}
	\item replacing each LR-row by its complement, and
	\item adding two distinguished rows: one labeled with L and the other labeled with R, and both of them having a $1$ in every column. 
\end{itemize}	 
\end{definition}

In \cite{T72}, Tucker characterised all the minimal forbidden submatrices for the C$1$P, later known as \emph{Tucker matrices}. 

\begin{theorem}\label{teo:tucker}
The $(0,1)$--matrix $M$ has the consecutive-ones property for columns and rows if and only if no submatrix of $M$, or of the transpose of $M$, is a member of the families depicted in Figure~\ref{fig:tucker_matrices}.
\end{theorem}
 
\begin{figure}[h!] 
	\centering
	\begin{align*}
			M_I(k)&= \scriptsize{ \begin{pmatrix}
				110...00\\
				011...00\\
				.   .   .   .   . \\
				.   .   .   .   . \\
				.   .   .   .   . \\
				000...11\\
				100...01\\
			\end{pmatrix}}
			&
			M_{II}(k)&= \scriptsize{  \begin{pmatrix}
				011...111\\
				110...000\\
				011...000\\
				.   .   .   .   . \\
				.   .   .   .   . \\
				000...110\\
				111...101\\
			\end{pmatrix}}
			&
			M_{III}(k)&= \scriptsize{  \begin{pmatrix}
				110...000\\
				011...000\\
				.   .   .   .   . \\
				.   .   .   .   . \\
				000...110\\
				011...101\\
			\end{pmatrix}}
			\\
			\end{align*}
			\begin{align*}
			M_{IV}&= \scriptsize{ \begin{pmatrix}
				110000\\
				001100\\
				000011\\
				010101\\
			\end{pmatrix}}
			&
			M_{V}&= \scriptsize{  \begin{pmatrix}
				11000\\
				00110\\
				11110\\
				10011\\
			\end{pmatrix}}
	\end{align*}
	\caption{Tucker matrices $M_{I}(k) \in \{0,1\}^{k \times k}$, $M_{III}(k) \in \{0,1\}^{k \times (k+1)}$ with $k \geq 3$, and $M_{II}(k) \in \{0,1\}^{k \times k}$ with $k \geq 4$} \label{fig:tucker_matrices}
\end{figure} 

Now we can state Theorem~\ref{teo:2-nested_charact}, which characterises $2$-nested matrices by forbidden subconfigurations and is the main result of this work. The proof for this theorem will be given at the end of the section.
 
\begin{theorem} \label{teo:2-nested_charact}
	Let $A$ be an enriched matrix. Then, $A$ is $2$-nested if and only if $A$ contains none of the following listed matrices or their dual matrices as subconfigurations: 
	\begin{itemize}		
	\item $M_0$, $M_{II}(4)$, $M_V$ or $S_0(k)$ for every even $k$ (see Figure~\ref{fig:forb_M_chiquitas})
	\item every enriched matrix in the family $\mathcal{D}$ (see Figure~\ref{fig:forb_D})  
	\item every enriched matrix in the family $\mathcal{F}$ (see Figure~\ref{fig:forb_F}) 
	\item every enriched matrix in the family $\mathcal{S}$ (see Figure~\ref{fig:forb_S})
	\item every enriched matrix in the family $\mathcal{P}$ (see Figure~\ref{fig:forb_P})
	\item monochromatic gems, monochromatic weak gems, badly-colored doubly-weak gems 
	\end{itemize}
and $A^*$ contains no Tucker matrices (see Figure~\ref{fig:tucker_matrices}) and none of the enriched matrices in $\mathcal{M}$ or their dual matrices as subconfigurations (see Figure~\ref{fig:forb_LR-orderable}).

\end{theorem} 

Throughout the following sections we give some definitions and char\-ac\-ter\-i\-za\-tions that will allow us to prove the above theorem. In Section~\ref{subsec:admissibility} we define and characterise the notion of admissibility, which encompasses all the properties we need to consider when coloring the blocks of an enriched matrix. In Section~\ref{subsec:part2nested}, we give a characterization for LR-orderable matrices by forbidden subconfigurations. Afterwards, we define and characterise partially $2$-nested matrices, which are those enriched matrices that admit an LR-ordering and for which the given coloring 
of those labeled rows induces a partial block bi-coloring.
These definitions and characterizations allow us to prove Lemmas~\ref{lema:2-nested_if} and~\ref{lema:2-nested_onlyif}, which are of fundamental importance for the proof of Theorem~\ref{teo:2-nested_charact}.

\subsection{Admissibility} \label{subsec:admissibility}

In this section we define the notion of admissibility for an enriched $(0,1)$-matrix, which allows us to describe the necessary conditions for an enriched matrix to admit a block bi-coloring. 
Notice that the existence of a block bi-coloring for an enriched matrix is a property that can be defined and characterised by forbidden subconfigurations. 

Let us consider the matrices defined in~\ref{fig:forb_D}. These matrices are all examples of enriched matrices that do not admit a total block bi-coloring (recall Definition~\ref{def:2-nested}).
For example, let us consider $D_0$. In order to have a total block bi-coloring of $D_0$, it is necessary that $D_0$ admits an LR-ordering of its columns. In particular, in such an ordering every row labeled with L starts in the first column. Hence, if there is indeed an LR-ordering for $D_0$, then the existence of two distinct non-nested rows labeled with L is not possible. The same holds if both rows are labeled with R.
We can use similar arguments to see that neither $D_2$, $D_3$, $D_7$ and $D_{11}$ admit an LR-ordering and thus they do not admit a total block bi-coloring. 

If instead we consider the enriched matrix $D_1$, then it is straightforward to see that assertion~\ref{item:2nested5} of Definition~\ref{def:2-nested} does not hold for this matrix.
Consider now the matrix $D_4$. It follows from assertion~\ref{item:2nested8} of Definition~\ref{def:2-nested} that if an enriched matrix has two distinct rows labeled with L and colored with distinct colors, then every LR-row has an L-block, and thus $D_4$ does not admit a total block bi-coloring.
Suppose now that $D_4$ is a submatrix of some enriched matrix and that the corresponding LR-row is nonempty in $A$. 
Notice that, if the LR-row has an L-block then it is properly contained in both rows labeled with L. It follows from this and assertion~\ref{item:2nested3} of Definition~\ref{def:2-nested} that the L-block of the LR-row must be colored with a distinct color than the one given to each row labeled with L. However, each of these rows is colored with a distinct color, thus a total block bi-coloring is not possible in that case. 
If we consider the enriched matrix $D_5$, then it follows from assertion~\ref{item:2nested4} of Definition~\ref{def:2-nested} that there is no possible LR-ordering such that the L-block of the LR-row does not intersect the L-row, and the same follows for the R-block of the LR-row and the R-row of $D_5$.

Let us consider the enriched matrix in which we find $D_6$ as a subconfiguration. If the LR-row has an L-block, then it is contained in the L-row, and the same holds for the R-block of the LR-row and the R-row. By assertion~\ref{item:2nested3} of Definition~\ref{def:2-nested}, the L-block must be colored with a distinct color than the L-row, and the R-block must be colored with a distinct color than the R-row. Equivalently, the L-block and the R-block of the LR-row are colored with the same color. However, this is not possible by assertion~\ref{item:2nested1} of Definition~\ref{def:2-nested}.
Similarly, we can see that $D_8$, $D_9$, $D_{10}$, $D_{12}$ and $D_{13}$ do not admit a total block bi-coloring, also having in mind that assertion~\ref{item:2nested9} of Definition~\ref{def:2-nested} must hold pairwise for LR-rows.

\begin{definition} \label{lista:prop_adm}
Let $A$ be an enriched matrix. We define the following properties: 

\begin{enumerate}[(a)]  
	    \item If two rows are labeled both with L or both with R, then they are nested.  \label{item:0_def_adm}
	    
	 	\item If two rows with the same color are labeled one with L and the other with R, then they are disjoint. \label{item:1_def_adm}

	 	\item If two rows with distinct colors are labeled one with L and the other with R, then either they are disjoint or there is no column where both have $0$ entries. \label{item:2_def_adm}

		\item If two rows $r_1$ and $r_2$ have distinct colors and are labeled one with L and the other with R, then any LR-row with at least one non-zero column has nonempty intersection with either $r_1$ or $r_2$. \label{item:3_def_adm}
		
	 	\item If two rows $r_1$ and $r_2$ with distinct colors are labeled both with L or both with R, then for any LR-row $r$, $r_1$ is contained in $r$ or $r_2$ is contained in $r$. \label{item:4_def_adm}		

	
		\item If two non-disjoint rows $r_1$ and $r_2$ with distinct colors, one labeled with L and the other labeled with R, then any LR-row is disjoint with regard to the intersection of $r_1$ and $r_2$. \label{item:5_def_adm}

	
		\item If two rows with the same color are labeled one with L and the other with R, then for any LR-row $r$ one of them is contained in $r$. 
		Moreover, the same holds for any two rows with distinct colors and labeled with the same letter. \label{item:6_def_adm}
		

		\item For each three non-disjoint rows such that two of them are LR-rows and the other is labeled with either L or R, two of them are nested. \label{item:7_def_adm}

		\item If two rows $r_1$ and $r_2$ with distinct colors are labeled one with L and the other with R, and there are two LR-rows $r_3$ and $r_4$ such that $r_1$ is neither disjoint or contained in $r_3$ and $r_2$ is neither disjoint or contained in $r_4$, then $r_3$ is nested in $r_4$ or viceversa. \label{item:8_def_adm}

		\item For each three LR-rows, two of them are nested. \label{item:9_def_adm}
    \end{enumerate} 	
    
\end{definition}
    
\vspace{2mm}
For each of the above properties, we characterise the set of minimal forbidden subconfigurations with the following Lemma.

\begin{lemma}
	For any enriched matrix $A$, all of the following assertions hold:
	\begin{enumerate}
		\item $A$ satisfies~assertion~(\ref{item:0_def_adm}) of Definition~\ref{lista:prop_adm} if and only if $A$ contains no $D_0$ or its dual matrix as a subconfiguration.
		\item $A$ satisfies~assertion~(\ref{item:1_def_adm}) of Definition~\ref{lista:prop_adm} if and only if $A$ contains no $D_1$ or its dual matrix as a subconfiguration.
		\item $A$ satisfies~assertion~(\ref{item:2_def_adm}) of Definition~\ref{lista:prop_adm} if and only if $A$ contains no $D_2$ or its dual matrix as a subconfiguration.
		\item $A$ satisfies~assertion~(\ref{item:3_def_adm}) of Definition~\ref{lista:prop_adm} if and only if $A$ contains no $D_2$, $D_3$ or their dual matrices as subconfigurations.
		\item $A$ satisfies~assertion~(\ref{item:4_def_adm}) of Definition~\ref{lista:prop_adm} if and only if $A$ contains no $D_0$, $D_4$ or their dual matrices as subconfigurations.
		\item $A$ satisfies~assertion~(\ref{item:5_def_adm}) of Definition~\ref{lista:prop_adm} if and only if $A$ contains no $D_5$ or its dual matrix as a subconfiguration.
		\item $A$ satisfies~assertion~(\ref{item:6_def_adm}) of Definition~\ref{lista:prop_adm} if and only if $A$ contains no $D_0$, $D_1$, $D_4$, $D_6$ or their dual matrices as subconfigurations.
		\item $A$ satisfies~assertion~(\ref{item:7_def_adm}) of Definition~\ref{lista:prop_adm} if and only if $A$ contains no $D_7$, $D_8$, $D_9$ or their dual matrices as subconfigurations.		
		\item $A$ satisfies~assertion~(\ref{item:8_def_adm}) of Definition~\ref{lista:prop_adm} if and only if $A$ contains no $D_5$, $D_9$, $D_{10}$ or its dual matrix as a subconfiguration.		
		\item $A$ satisfies~assertion~(\ref{item:9_def_adm}) of Definition~\ref{lista:prop_adm} if and only if $A$ contains no $D_{11}$, $D_{12}$, $D_{13}$ or their dual matrices as subconfigurations.	
	\end{enumerate}
\end{lemma}

\begin{proof}

\vspace{1mm}
First we find every forbidden subconfiguration corresponding to assertion~\ref{item:0_def_adm} of Definition~\ref{lista:prop_adm}.
Let $f_1$ and $f_2$ be two rows labeled both with L or both with R, and suppose they are not nested. Thus, there is a column in which $f_1$ has a $1$ and $f_2$ has a $0$, and another column in which $f_2$ has a $1$ and $f_1$ has a $0$. 
In this way, we find $D_0$ or its dual as a
subconfiguration of $A$.

Let us find now every forbidden subconfiguration corresponding to assertion~\ref{item:1_def_adm} of Definition~\ref{lista:prop_adm}.
Let $f_1$ and $f_2$ be rows labeled one with L and the other with R and colored with the same color. If $f_1$ and $f_2$ are not disjoint, then there is a column in which both rows have a $1$. In this case, we find $D_1$ or its dual as a 
subconfiguration of $A$.

For assertion~\ref{item:2_def_adm} of Definition~\ref{lista:prop_adm}, let $f_1$ and $f_2$ be two rows labeled one with L and the other with R and colored with distinct colors, and suppose they are not disjoint and there is a column $j_1$ such that both rows have a $0$ in column $j_1$. Thus, there is a column $j_2 \neq j_1$ such that both rows have a $1$ in column $j_2$.
Hence, $D_2$ is a 
subconfiguration of $A$.

With regard to assertion~\ref{item:3_def_adm} of Definition~\ref{lista:prop_adm}, let $f_1$ and $f_2$ be two rows labeled one with L and the other with R and colored with distinct colors. Let $f_3$ be a nonempty LR-row. Suppose that $f_3$ is disjoint with both $f_1$ and $f_2$. 
Hence, there is a column $l_1$ such that $f_1$ and $f_2$ have a $0$ and $f_3$ has a $1$. Moreover, either there are two distinct columns $j_1$ and $j_2$ such that the column $j_i$ has a $1$ in row $f_i$ and a $0$ in the other rows, for $i=1, 2$, or there is a column $l_2$ such that $f_1$ and $f_2$ both have a $1$ in column $l_2$ and $f_3$ has a $0$.
If the latter holds, we find $D_2$ as a subconfiguration induced by the rows $f_1$ and $f_2$. 
If instead there are two distinct columns $j_1$ and $j_2$ as described above, then we find $D_3$ as a 
subconfiguration of $A$.

For assertion~\ref{item:4_def_adm} of Definition~\ref{lista:prop_adm}, let $f_1$ and $f_2$ be two rows labeled with L and colored with distinct colors, and let $r$ be an LR-row. If $f_1$ and $f_2$ are not nested, then we find $D_0$ as a subconfiguration. Suppose that $f_1$ and $f_2$ are nested. If neither $f_1$ nor $f_2$ is contained in $r$, then there is a column $j$ in which $f_1$ and $f_2$ have a $1$ and $r$ has a $0$.
Thus, $D_4$ is a 
subconfiguration of $A$.

For assertion~\ref{item:5_def_adm} of Definition~\ref{lista:prop_adm}, let $f_1$ and $f_2$ be two non-disjoint rows colored with distinct colors, $f_1$ labeled with L and $f_2$ labeled with R. Since they are non-disjoint, there is at least one column $j$ in which both rows have a $1$. Suppose that for every such column $j$, there is an LR-row $f$ having a $1$ in that column. Then, we find $D_5$ as a subconfiguration of $A$.

For assertion~\ref{item:6_def_adm} of Definition~\ref{lista:prop_adm}, let $f$ be an LR-row and let $f_1$ and $f_2$ be two rows labeled with L and R respectively, and colored with the same color. If $f_1$ and $f_2$ are not disjoint, then we find $D_1$ as a subconfiguration. Suppose that $f_1$ and $f_2$ are disjoint.
If neither $f_1$ nor $f_2$ is contained in $f$, then there are columns $j_1 \neq j_2$ such that $f_i$ has a $1$ and $f$ has a $0$, for $i=1, 2$. Thus, we find $D_6$ as a subconfiguration of $A$. If instead $f_1$ and $f_2$ are both labeled with L and colored with distinct colors, and neither is contained in $f$, then we find $D_4$  or $D_0$ as a subconfiguration in $A$ depending on whether or not $f_1$ and $f_2$ are nested.

Suppose that $A$ satisfies~assertion~\ref{item:7_def_adm} of Definition~\ref{lista:prop_adm}. Let $f_1$ be a row labeled with L, and $f_2$ and $f_3$ two distinct LR-rows such that none of them is nested in the others. Thus, we have three possibilities. If there are three columns $j_i$, $i=1,2,3$, such that $f_i$ has a $1$ and the other rows have a $0$, then we find $D_7$ as a subconfiguration of $A$.
If instead there are three rows $j_i$, $i=1,2,3$, such that $f_i$ and $f_{i+1}$ have a $1$ and $f_{i+2}$ has a $0$ in $j_i$ (where subdindices are modulo 3), then we find $D_8$ as a subconfiguration.
The remaining possibility is that there are 4 columns $j_1, j_2, j_3, j_4$ such that $f_1$ and $f_2$ have a $1$ and $f_3$ has a $0$ in $j_1$, $f_1$ has a $1$ and $f_2$ and $f_3$ have a $0$ in $j_2$, $f_3$ has a $1$ and $f_1$ and $f_2$ have a $0$ in $j_3$, and $f_2$ and $f_3$ have a $1$ and $f_1$ has a $0$ in $j_4$. Moreover, since all three rows are pairwise non-disjoint, either there is a fifth column for which $f_1$ and $f_3$ have a $1$ and $f_2$ has a $0$ (in which case we find $D_8$ as a subconfiguration), or $f_2$ has a $1$ and $f_1$ and $f_3$ have a $0$ (in which case we have $D_7$ as a subconfiguration), or all three rows have a $1$ in such column. In this case, we find $D_9$ as a subconfiguration of $A$.

For assertion~\ref{item:8_def_adm} of Definition~\ref{lista:prop_adm}, let $f_1$ and $f_2$ be two rows labeled with L and R, respectively, and colored with distinct colors. Let $f_3$ and $f_4$ be two LR-rows such that $f_1$ is neither disjoint nor contained in $f_3$ and $f_2$ is neither disjoint nor contained in $f_4$. 
If $f_1$ is also not contained in $f_4$ or $f_2$ is not contained in $f_3$, then we find $D_9$ as a subconfiguration. Thus, suppose that $f_1$ is contained in $f_4$ and $f_2$ is contained in $f_3$. Moreover, we may assume that for any column such that $f_1$ and $f_3$ have a $1$, $f_2$ has a $0$ (and analogously for $f_2$ and $f_4$ having a $1$ and $f_1$), for if not we find $D_5$ as a subconfiguration.
Hence, there is a column $j_1$ in $A$ having a $1$ in $f_1$ and $f_4$ and having a $0$ in $f_3$ and $f_2$, and another column $j_2$ having a $1$ in $f_2$ and $f_3$ and having a $0$ in $f_1$ and $f_4$. Moreover, since $f_1$ and $f_3$ are not disjoint and $f_2$ and $f_4$ are not disjoint (and $f_1$ is nested in $f_4$ and $f_2$ is nested in $f_3$), there are columns $j_3$ and $j_4$ such that $f_1$, $f_3$ and $f_4$ have a $1$ and $f_2$ has a $0$ in $j_3$ and  $f_2$, $f_3$ and $f_4$ have a $1$ and $f_1$ has a $0$. Therefore, we find $D_{10}$ as a subconfiguration of $A$.

It follows by using a similar argument as for the previous assertions that, if $A$~does not satisfy assertion \ref{item:9_def_adm}, then that there is $D_{11}$, $D_{12}$ or $D_{13}$ as a subconfiguration in $A$, and this finishes the proof.
\end{proof}

\begin{corollary}
Every enriched matrix $A$ that admits a total block bi-coloring contains none of the matrices in $\mathcal{D}$. Equivalently, if $A$ admits a total block bi-coloring, then every property listed in~\ref{lista:prop_adm} hold. 
\end{corollary}

Another example of families of enriched matrices that do not admit a total block bi-coloring are $\mathcal{S}$ and $\mathcal{P}$, which are the matrices shown in Figures~\ref{fig:forb_S} and~\ref{fig:forb_P}, respectively.
Therefore, since the existence of a total block bi-coloring is a property inherited by subconfigurations, if an enriched matrix $A$ admits a total block bi-coloring, then $A$ contains none of the matrices in $\mathcal{S}$ or $\mathcal{P}$. 
With this in mind, we give the following definition which is also a characterization by forbidden subconfigurations.

\begin{definition} \label{teo:caract_admissible}
	Let $A$ be an enriched matrix. We say $A$ is \emph{admissible} if and only if $A$ is $ \{ \mathcal{D}, \mathcal{S}, \mathcal{P} \}$-free.
\end{definition}

\subsection{Partially 2-nested matrices} \label{subsec:part2nested}

This section is organized as follows. First, we give some definitions that will help us obtain a characterization of LR-orderable matrices (see Definition~\ref{def:LR-orderable}) by forbidden subconfigurations. 
Afterwards, we define and characterise partially $2$-nested matrices, which are those enriched matrices that admit an LR-ordering and for which the given coloring of the labeled rows of $A$ induces a partial block bi-coloring.

\begin{definition}
A \emph{tagged matrix} is a matrix $A$, each of whose rows is either uncolored or colored with blue or red, together with a set of at most two dis\-tin\-guished columns of $A$. These dis\-tin\-guished columns will be refered to as \emph{tag columns}.
\end{definition}

\begin{definition}  \label{def:tagged_matrixA}
Let $A$ be an enriched matrix. The \emph{tagged matrix of $A$}, denoted by $A_\tagg$, is a tagged matrix whose underlying matrix is obtained from $A$ by adding two columns, $c_L$ and $c_R$, such that:
\begin{enumerate}[(1)]
\item the column $c_L$ has a $1$ if $f$ is labeled L or LR and $0$ otherwise, 
\item the column $c_R$ has a $1$ if $f$ is labeled R or LR and $0$ otherwise, and 
\item the set of distinguished columns of $A_{\tagg}$ is $\{ c_L, c_R\}$. 
\end{enumerate}
We denote $A^*_\tagg$ to the tagged matrix of $A^*$ (recall Definition~\ref{def:A*}). By simplicity we will consider column $c_L$ as the first and column $c_R$ as the last column of $A_\tagg$ and $A^*_\tagg$. We consider all the rows of $A_\tagg$ and $A^*_\tagg$ colored with the same (or no) color as in $A$.
\end{definition}

\begin{figure}[h!]
	\begin{align*}
		A_\tagg = \scriptsize{ \bordermatrix{& c_L \hspace{7mm}  c_R \cr
		& \pmb 1 1 0 0 0 1 \pmb 1 \cr
		& \pmb 1 1 1 0 0 1 \pmb 1 \cr
	 	& \pmb 0 0 1 1 0 0 \pmb 0 \cr
		& \pmb 1 1 1 1 0  0  \pmb 0 \cr
		& \pmb 1 0 0 0 0 0 \pmb 1 \cr
		& \pmb 0 0  0  1 1 1 \pmb 1 }\ 
		\begin{matrix} 
   \\  \\ \\  \textcolor{dark-red}{\bullet} \\ \textcolor{blue}{\bullet} \\ \textcolor{blue}{\bullet} 
\end{matrix}}
		&&
		A^{*}_\tagg = \scriptsize{ \bordermatrix{& c_L \hspace{7mm} c_R \cr
		& \pmb 0 \pmb 1 \pmb 1 \pmb 1 \pmb  1 \pmb 1  \pmb 1 \cr
		& \pmb 1 \pmb 1 \pmb 1 \pmb 1 \pmb 1 \pmb 1 \pmb 0 \cr
		& \pmb 0 0 1 1  1 0  \pmb 0 \cr
		& \pmb 0 0 0 1 1 0  \pmb 0 \cr
	 	& \pmb 0 0 1 1 0 0 \pmb 0 \cr
		& \pmb 1 1 1 1 0 0 \pmb 0 \cr
		& \pmb 0 0 0 1 1 1 \pmb 1  }\ 
		\begin{matrix} 
   \\  \\ \\  \\ \textcolor{dark-red}{\bullet} \\ \textcolor{blue}{\bullet} \\ \textcolor{blue}{\bullet} 
\end{matrix}}
	\end{align*}
	\caption{Example of $A_\tagg$ and $A^*_\tagg$ for the matrix $A$ in Figure~\ref{fig:example_LR-ord_a}. The first two bold rows of $A^*_\tagg$ are its distinguished rows, added by definition.} \label{fig:example_tagg}
\end{figure}


The following remarks allow us to simplify the proof of the characterization of LR-orderable matrices.

\begin{remark}
	If $A^*_{\tagg}$ has the C$1$P, then the distinguished rows force the tag columns $c_L$ and $c_R$ to be the first and last columns of $A^*_{\tagg}$, respectively. 
\end{remark}

\begin{remark}
	An admissible matrix $A$ is LR-orderable if and only if the tagged matrix $A^*_{\tagg}$ has the C$1$P.
\end{remark}

\begin{lemma} \label{lema:LR-orderable_caract_bymatrices}
	An admissible matrix $A$ is LR-orderable if and only if the tagged matrix $A^*_{\tagg}$ contains neither Tucker matrices nor the tagged matrices of the family $\mathcal{M}$ depicted in Figure~\ref{fig:forb_LR-orderable} as subconfigurations.  
\end{lemma} 

\begin{proof}
$\Rightarrow )$ This follows from the last remark and Theorem~\ref{teo:tucker}.

$\Leftarrow )$ Suppose that the tagged matrix $A^*_{\tagg}$ does not contain any of the above listed matrices as subconfigurations, and still the C$1$P does not hold for the rows of $A^*_{\tagg}$. 

Hence, there is a Tucker matrix $M$ such that $M$ is a subconfiguration of $A^*_{\tagg}$. Notice that, if $M$ does not have any tag column of $A^*_{\tagg}$, then $M$ is a subconfiguration of $A$. Hence, $A$ does not have the C$1$P and therefore the result follows. Henceforth, we assume that at least one column of $M$ is a tag column of $A^*_{\tagg}$. 
Suppose without loss of generality that, if $M$ intersects only one tag column, then this tag column is $c_L$, since the analysis is symmetric if assumed otherwise and gives as a result in each case the dual matrix.

\vspace{3mm}
\begin{mycases}
\case \textit{Suppose first that $M$ intersects one or both of the distinguished rows.} 
Thus, $M$ 
is \textit{$M_V$, $M_I(3)$, or $M_{II}(k)$ for some $k \geq 3$.} We consider each case separately. 

\vspace{2mm}
\subcase $M= M_V$.
In this case, the distinguished row is $(1,1,1,1,0)$ and thus the last column of $M$ 
is a tag column. 
Hence we find $M_5'$ in $A^*_{\tagg}$, 
which results in a contradiction. 
 
\subcase $M=M_I(3)$. 
If $(1,1,0)$ is a distinguished row, then we find $D_0$ as a 
subconfiguration given by the second and third rows. It is symmetric if the distinguished row is either the second or the third row, and therefore this case is not possible.

\subcase $M= M_{II}(k)$. 
In this case, the distinguished rows of $A^*_{\tagg}$ can correspond only to the first and the last row of $M$. 
Suppose only the first row $(0, 1, \ldots, 1)$ of $M$ corresponds to a distinguished row of $A^*_{\tagg}$. Thus, the first column of $M$ is precisely a tag column of $A^*_{\tagg}$. 
Hence, $M_2'(k)$ is a subconfiguration of $A_{\tagg}$, and this results in a contradiction. The same holds if instead the last row is the sole distinguished row.

Finally, suppose both the first and the last row of $M$ correspond to the distinguished rows of $A^*_{\tagg}$. If this is the case, then the columns $1$ and $k-1$ of $M$ correspond to the tag columns of $A^*_{\tagg}$.
Suppose first that $M = M_{II}(4)$. In this case, every row of $M$ is colored, since every row corresponds to a row of $A$ labeled with either L or R. Moreover, the first and second row of $M$ are colored with distinct colors, for if not we find $D_1$ as a subconfiguration of $A$. The same holds for the second and third row of $M$, and also for the third and fourth row of $M$.  
However, this implies that the second and third row of $M$ induce $D_2$ in $A$, hence this case is not possible.

If instead $M = M_{II}(k)$ for $k \geq 5$, then $M_2''(k)$ is a subconfiguration of $A^*_{\tagg}$, and thus we reach a contradiction. 
 
\case \textit{Suppose that $M$ does not intersect any distinguished rows of $A^*_{\tagg}$.}

Suppose first that exactly one of the columns in $M$ is a tag column.

\subcase $M= M_I(k)$.
Notice that, if any of the columns of $M$ is a tag column of $A^*_{\tagg}$, then we find $D_0$ as a subconfiguration of $A$, which results in $A$ not being admissible and thus reaching a contradiction.

\subcase $M=M_{II}(k)$. 
As in the previous case, some of the columns of $M$ cannot be tag columns of $A^*_{\tagg}$. If there is only one tag column, then he only remaining possibilities for tag columns in $M$ are column $1$ or column $k-1$, for in any other case we find $D_0$ as a subconfiguration of $A$. Analogously, if instead $M$ intersects both tag columns of $A^*_{\tagg}$, then such columns are also columns $1$ and $k-1$.
However, if $c_L$ is either column $1$ or column $k-1$, then $M_2''(k)$ is a subconfiguration of $A^*_{\tagg}$. 
Analogously, if $c_R$ is either column $1$ or $k-1$ of $M$, then we find the dual matrix of $M_2'(k)$ as a subconfiguration.

Finally, suppose that two columns are tag columns.
Notice that the first and second rows of $M$ are colored with distinct colors, for if not we find $D_1$ as a subconfiguration of $A$. The same holds for the last two rows of $M$.
Hence, if $k = 4$, then we find $D_2$ as a subconfiguration od $A$ given by the second and third rows. If instead $k > 5$, then $M_2''(k)$ is a subconfiguration of $A^*_{\tagg}$, which results once more in a contradiction.

\subcase $M=M_{III}(k)$. 
In this case, the only possibilities for tag columns in $M$ are column $1$, column $k-1$ and column $k$, for if not we find $D_0$ as a subconfiguration of $A$. 

Suppose first that the tag column of $A^*_{\tagg}$ is the first column of $M$. In that case, we find $M_3'(k)$ as a subconfiguration of $A^*_{\tagg}$, which also results in a contradiction. 
If instead the tag column is column $k$, then we use an analogous reasoning to find $M_3''(k)$ as a subconfiguration of $A^*_{\tagg}$ and thus reaching a contradiction.

Suppose now that both the first column and the last column of $M$ are tag columns.
Since $A$ is admissible, this case is not possible for the first and last row induce $D_1$ or $D_2$ as subconfigurations, depending on whether the rows are colored with the same color or with distinct colors, respectively.

\subcase $M=M_{IV}$. 
In this case, the columns of $M$ that could be tag columns are column $1$, column $3$ and column $5$, for if any other column of $M$ is a tag column of $A^*_{\tagg}$, we find $D_0$ as a subconfiguration of $A$, thus contradicting the hypothesis of admissibility of $A$. 
Further\-more, the choice of the tag column is symmetric since there is a reordering of the rows that allows us to obtain the same matrix if the tag column is either column $1$, column $3$ or column $5$, regardless of the choice of the column. 
Hence, there are two possibilities: when column $1$ is the only tag column of $A^*_{\tagg}$ in $M$, and when the two tag columns of $A^*_{\tagg}$ are columns $1$ and $3$ of $M$.
If column $1$ is the only tag column, then we find $M_4'$ as a subconfiguration of $A^*_{\tagg}$.
If instead the columns $1$ and $3$ are both tag columns, then the first and the second row of $M$ are colored with the same color, for if not there is $S_3(3)$ as a subconfiguration of $A$ and this is not possible since $A$ is admissible. Thus, in this case we find $M_4''$ as a subconfiguration of $A^*_{\tagg}$. 

\subcase $M=M_{V}$.
Once more and using the same argument, the only columns that could be tag columns of $A^*_{\tagg}$ are columns $2$, $3$ or $5$. Moreover, if the second column of $M$ is the sole tag column, then there is a reordering of the rows such that the matrix obtained is the same as the matrix when the third column is the tag column.
If column $5$ is the only tag column, then we find $M_5'$ as a subconfiguration as in Case 1.1.
If instead column $2$ is the only tag column, then the first and second rows of $M$ have the same color, for if not we find $S_2(3)$ as a subconfiguration of $A$, and thus we find $M_5''$ as a subconfiguration of $A^*_{\tagg}$. 
Finally, if columns $2$ and $5$ are both tag columns of $A^*_{\tagg}$, then the first and last row of $M$ induce $D_2$ as a subconfiguration, disregarding the coloring of the rows and thus this case is also not possible. 

\end{mycases}

Therefore, we reached a contradiction by assuming that the C$1$P does not hold for $A^*_{\tagg}$ if $A^*_{\tagg}$ contains the listed subconfigurations. 
\end{proof}


Let $A$ be an enriched matrix and let $A_{LR}$ be the enriched submatrix of $A$ consisting precisely of all the LR-rows of $A$. We give a useful property for this enriched submatrix when $A$ is admissible.

\begin{lemma} \label{lema:A_LR_2nested}
If $A$ is admissible, then $A_{LR}$ contains no $F_1(k)$ or $F_2(k)$, for every odd $k \geq 5$.
\end{lemma} 

\begin{proof}
Toward a contradiction, suppose that $A_{LR}$ contains $F_1(k)$ or $F_2(k)$ as a subconfiguration, for some odd $k \geq 5$. Since $k\geq5$, we find the following subconfiguration in $A$:
\begin{align*}
	\scriptsize{ \bordermatrix{ & \cr
		\textbf{LR} & 1 1 0 0 \cr
		\textbf{LR} & 0 1 1 0 \cr
		\textbf{LR} & 0 0 1 1 }\ }
\end{align*}
Since these three rows induce $D_{13}$, we reach a contradiction. 
It follows analogously that there is no $F_2(k)$ in $A_{LR}$. 
Therefore, $A_{LR}$ contains neither $F_1(k)$ nor $F_2(k)$, for every odd $k \geq 5$. 
\end{proof}

\begin{remark} \label{rem:A_LR_2nested}
	It follows from Lemma~\ref{lema:A_LR_2nested} that, if $A$ is admissible, then there is a partition of the LR-rows of $A$ into two subsets $S_1$ and $S_2$ such that every pair of rows in each subset are either nested or disjoint. Moreover, since $A$ contains no $D_{11}$ as a subconfiguration, every pair of LR-rows that lie in the same subset $S_i$ are nested, for each $i=1,2$. Equivalently, the LR-rows in each subset $S_i$ are totally ordered by inclusion, for each $i=1,2$.
\end{remark}

When giving the guidelines to draw a circle model for any split graph $G=(K,S)$, not only is it important that the matrix $A(S_i,K_i)$ for each partition $K_i$ of $K$ results admissible and LR-orderable. We also need to ensure that there exists an LR-ordering in which every LR-row that is split into an L-block and an R-block satisfies a couple more properties than the ones given by the LR-ordering itself.
The following definition states the necessary conditions for the LR-ordering that we need to consider to obtain a circle model. We call this a \emph{suitable LR-ordering}. The lemma that follows ensures that, if a matrix $A$ is admissible and LR-orderable, then we can always find a suitable LR-ordering for the columns of $A$.

\begin{definition} \label{def:suitable_ordering}
	An LR-ordering $\Pi$ is \emph{suitable} if the L-blocks of those LR-rows with exactly two blocks are disjoint with every R-block, the R-blocks of those LR-rows with exactly two blocks are disjoint with the L-blocks and, for each LR-row, the intersection with any U-block is empty with either its L-block or its R-block.
\end{definition}

\begin{lemma} \label{lema:hay_suitable_ordering}
If $A$ is admissible, LR-orderable and contains no $M_0$, $M_{II}(4)$, $M_V$ or $S_0(k)$ as a subconfiguration for every even $k \geq 4$, then it has at least one suitable LR-ordering. 
\end{lemma} 

\begin{proof}
	Let $A$ be an admissible LR-orderable matrix. Toward a contradiction, suppose that every LR-ordering is non-suitable.
	If $\Pi$ is an LR-ordering of $A$, then either (1) there is a U-block $u$ and an LR-row $f_1$ such that $u$ is not disjoint with the L-block and the R-block of $f_1$, or (2) there is an LR-row $f_1$ such that its L-block is not disjoint with some R-block, since $\Pi$ is non-suitable. In both cases, there is no possible reordering of the columns to obtain a suitable LR-ordering. 	
	
	\vspace{1mm}
	Since $A$ is admissible, it follows from Lemma~\ref{rem:A_LR_2nested} that the LR-rows can be split into a two set partition such that the LR-rows in each set are totally ordered. Moreover, any two LR-rows for which the L-block of one intersects the R-block of the other are in distinct sets of the partition and thus the columns may be reordered by moving the portion of the block that one of the rows has in common with the other all the way to the right (or left). Hence, if two such blocks intersect and there is no possible LR-reordering of the columns, then there is at least one non-LR-row blocking the reordering.
	Throughout the proof and for simplicity, we will say that a row or block \emph{$a$ is chained to the left (resp.\ to the right) of another row or block $b$} if $a$ and $b$ overlap and $a$ intersects $b$ in column $l(b)$ (resp.\ $r(b)$). 
	
\begin{mycases}	
	\vspace{1mm}
	\case
	\textit{Let $a_1$ be the L-block of $f_1$ and $b_1$ be the R-block of $f_1$.
	Suppose first there is a U-block $u$ such that $u$ intersects both $a_1$ and $b_1$.}
	
	Let $j_1 = r(a_1)+1$, this is, the first column in which $f_1$ has a $0$, $j_2= r(a_1)$ and $j_3= l(b_1)$ in which both rows $f_1$ and $u$ have a $1$.
	Since it is not possible to rearrange the columns to obtain a suitable LR-ordering, in particular, there are two columns $j_4 < j_2$ and $j_5 > j_3$ in which $u$ has $0$, one before and one after the string of $1$'s of $u$. Moreover, there is at least one row $f_2$ distinct to $f_1$ and $u$ blocking the reordering of the columns $j_1$, $j_2$ and $j_3$. 
	\subcase Suppose $f_2$ is the only row blocking the reordering. Notice that $f_2$ is neither disjoint nor nested with $u$ and there is at least one column in which $f_1$ has a $0$ and $f_2$ has a $1$. We may assume without loss of generality that this is column $j_1$. Suppose $f_2$ is unlabeled.
	The only possibility is that $f_2$ overlaps with $u$, $a_1$ and $b_1$, for if not we can reorder the columns to obtain a suitable LR-ordering. In that case, we find $M_0$ as a subconfiguration in $A$.
	If instead $f_2$ is labeled with either L or R, then we find $S_6'(3)$ as a subconfiguration in $A$ considering columns $j_4$, $j_2$, $j_1$, $j_3$, $j_5$ and both tag columns. If $f_2$ is an LR-row and $f_2$ is the only row blocking the reordering, then either the L-block of $f_2$ is nested in the L-block of $f_1$ and the R-block of $f_2$ contains the R-block of $f_1$, or vice versa. However, in that case we can move the portion of the L-block of $f_1$ that intersects $u$ to the right and thus we find a suitable LR-ordering, therefore this case is not possible.
	
	\subcase Suppose now there is a sequence of rows $f_2, \ldots, f_k$ for some $k \geq 3$ blocking the reordering such that $f_i$ and $f_{i+1}$ overlap for each $i \in \{2, \ldots, k\}$. 
	Moreover, there is either --at least-- one row that overlaps $a_1$ or $b_1$. We may assume without loss of generality that $f_2$ is such a row and that $f_2$ and $b_1$ overlap.
	Suppose that $f_2$ and $f_3$ are unlabeled rows. Notice that, either all the rows are chained to the left of $f_2$ or to the right. Furthermore, since $A$ contains no $M_0$ as a subconfiguration and we assumed that $b_1$ and $f_2$ overlap, if $f_i$ is chained to the left of $f_2$, then $f_i$ is contained in $b_1$ for every $i \geq 3$, and if $f_i$ is chained to the right of $f_2$, then $f_i$ is contained in $u$ for every $3 \leq i <k$.
	In either case, we find $M_{II}(4)$ as a subconfiguration considering the columns $j_2$, $j_1$, $j_3$ and $j_5$. 
	Suppose that $f_2$ is the only labeled row in the sequence and that $f_2$ is labeled with R. If $u$ and $f_2$ overlap, then we find $S_6'(3)$ as a subconfiguration as in the previous paragraphs.
	Thus, we assume $u$ is nested in $f_2$. Since the sequence of rows is blocking the reordering, the rows $f_3, \ldots, f_k$ are chained one to one to the right and $f_k = u$, therefore we find $S_6(k)$ as a subconfiguration.
	The only remaining possibility is that there are two labeled rows in the sequence blocking the reordering. Since $A$ contains neither $D_1$ nor $S_3(3)$ as a subconfiguration, then either these two rows are labeled with the same letter and nested, or they are labeled one with L and the other with R and are disjoint. We may assume without loss of generality that $f_2$ and $f_k$ are such labeled rows.
	
	If $f_2$ and $f_k$ are both labeled with L, then necessarily one is nested in the other, for $\Pi$ is an LR-ordering. In that case, one has a $0$ in column $j_1$ and the other has a $1$ there, for if not we can reorder the columns moving $j_1$ --and maybe some other columns in which $f_1$ has a $0$-- to the right. Hence, in this case we find $S_5(k)$ as a subconfiguration induced by rows $f_1, f_2, \ldots, f_k$
	It is analogous if $f_2$ and $f_k$ are labeled with R.

	If instead $f_2$ and $f_k$ are labeled one with L and the other with R, then there are two possibilities. Either $f_2, \ldots, f_{k-1}$ are nested in $a_1$, or $f_2$ is chained to the right of $u$ and $f_3$ is chained to the left. In either case, if $f_2$ or $f_3$ have a $1$ in some column in which $f_1$ has a $0$ and $u$ has a $1$, then we find $S_6'(3)$ as a subconfiguration. If instead $f_3$ is nested in $a_1$ and $f_2$ is nested in $b_1$, then we find $M_V$ as a subconfiguration considering the columns $j_4$, $j_2$, $j_1$, $j_3$ and $j_5$.

	\vspace{2mm}	
	\case
	\textit{Suppose now that there is a row $f_2$ such that the L-block $a_1$ of $f_1$ and the R-block $b_2$ of $f_2$ are not disjoint.} Notice that, by definition of R-block, $f_2$ is labeled with either R or LR.
	Once more, we consider $j_1=r(a_1) + 1$, the first column in which $f_1$ has a $0$.

	Since $a_1$ and $b_2$ intersect, there is a column $j_2 < j_1$ such that $a_1$ and $b_2$ both have a $1$ in column $j_2$.
	\subcase Suppose first that there is exactly one row $f_3$ blocking the possibility of reordering the columns to obtain a suitable LR-ordering.
	Notice that, for a row to block the reordering of the columns, such row must have a $1$ in $j_2$ and at least one column with a $0$. We have three possible cases:

 	\subsubcase Suppose first that $f_3$ is unlabeled. If $f_2$ is labeled with LR and $f_3$ does not intersect the L-block of $f_2$, then we can move to the R-block of $f_1$ those columns in which $f_3$ has $0$ and $a_1$ has $1$. If $f_3$ intersects the L-block of $f_2$, then this is precisely as in the previous case. Thus, we assume $f_2$ is labeled with R. 
	If $f_3$ is not nested in either $f_1$ nor $f_2$, then there is a column $j_3$ in which $f_3$ and $f_2$ have a $1$ and $f_1$ has a $0$, and a column $j_4$ in which $f_3$ and $f_1$ have a $1$ and $f_2$ has a $0$. 
	In that case, we find $S_6(3)$ as a subconfiguration considering the columns $j_1$, $j_2$, $j_3$, $j_4$ and both tag columns.
	 If $f_3$ is nested in $f_2$, then we can rearrange the columns by moving to the right all the columns in which $a_1$ and $f_2$ both have $1$ and mantaining those columns in which $f_3$ has a $1$ together.
	 If instead $f_3$ is nested in $f_1$, then we find $S_6'(3)$ as a subconfiguration. 
	 
	\vspace{1mm}	
	\subsubcase Suppose now that $f_3$ is labeled with L. If $f_2$ is labeled with R, then $f_2$ and $f_3$ are colored with distinct colors, for if not we find $D_1$ as a subconfiguration. Thus, we find $D_5$ as a subconfiguration induced by $f_1$, $f_2$, $f_3$.
	 Moreover, notice that, if $f_3$ is also labeled with R, then it is possible to move all those columns of $a_1$ that have a $1$ and intersect $f_2$ (and $f_3$) in order to obtain a suitable LR-ordering and thus $f_3$ did not block the reordering.
	If instead $f_2$ is an LR-row, then we find either $D_7$, $D_8$ or $D_9$ as a subconfiguration, depending on where is the string of $0$'s in row $f_3$. Also notice that it is indistinct in this case if $f_3$ is labeled with R.
	\vspace{1mm} 
	 \subsubcase Suppose $f_3$ is labeled with LR. Since $A$ is admissible, if $f_2$ is an LR-row, then either $f_3$ is nested in $f_1$ or $f_3$ is nested in $f_2$ (we may assume this since it is analogous if $f_3$ contains $f_1$ or $f_2$: we will see that $f_3$ is not blocking the reordering). If $f_3$ is nested in $f_2$, then we can move the part of the L-block $a_1$ that intersects $b_2$ all the way to the right and then we have a suitable reordering. It is analogous if $f_3$ is nested in $f_1$.
	 If $f_2$ is labeled with R, then we may assume that $f_2$ is not nested in $f_3$, for if not we have a similar situation as in the previous paragraphs. The same holds if $f_1$ and $f_3$ are nested LR-rows. 
	 We know that the L-block $a_3$ of $f_3$ intersects the R-block $b_2$ 
	 Hence, in the column $j_3=r(a_3)+1$ the row $f_3$ has a $0$ and $f_2$ has a $1$, and in the column $j_4=l(b_2)-1$ the row $f_3$ has a $1$ and $f_2$ has a $0$. Moreover, since $f_1$ and $f_3$ are not nested, then there is a column greater than $j_2$ in which $f_1$ has a $0$ and $f_2$ and $f_3$ have a $1$. In this case, we find $D_8$ as a subconfiguration.

	\vspace{1mm}
	\subcase Suppose now that it is not possible to reorder the columns to obtain a suitable LR-ordering, since there is a sequence of rows $f_3, \ldots, f_k$, with $k > 3$, blocking --in particular-- the reordering of the columns $j_1=r(a_1)+1$ and $j_2=r(a_1)$.
	
	
	We may assume that the sequence of rows is either chained to the right --and thus $f_k$ is labeled with R-- or to the left --and thus $f_k$ is labeled with L, for if not we find $M_V$ as a subconfiguration as in the first case.
	Suppose that $f_2$ is labeled with R. If the sequence $f_3, \ldots, f_k$  is chained to the left, then we find $S_4(k)$ as a subconfiguration. If instead the sequence $f_3, \ldots, f_k$ is chained to the right, then we find $S_1(k)$ as a subconfiguration.
	Suppose now that $f_2$ is an LR-row. 
	Since the L-block of $f_1$ and the R-block of $f_2$ intersect, then these rows are not nested. Whether the sequence is chained to the right or to the left, we may assume that $f_3$ is nested in $a_1$ and is disjoint with $a_2$. Let $k$ be the number of $0$'s between the L-block and the R-block of $f_2$. Depending on whether $k$ is odd or even, we find $S_0(k)$ or $S_8(k)$, respectively, as a subconfiguration of the subconfiguration given by considering the rows $f_1, f_2, \ldots, f_{k+3}$. This finishes the proof.
\end{mycases}

\end{proof}

\begin{definition} \label{def:partially_2-nested}
	Let $A$ be an enriched matrix. 
	We say $A$ is \emph{partially $2$-nested} if all the following assertions hold:
	 \begin{enumerate}
	 	\item $A$ is admissible, LR-orderable and contains no $M_0$, $M_{II}(4)$, $M_V$ or $S_0(k)$ as a subconfiguration for any even $k \geq 4$.
	 	\item Each pair of non-LR-rows colored with the same color are either disjoint or nested in $A$.
	 	\item If an L-block (resp.\ R-block) of an LR-row is colored, then any non-LR-row colored with the same color is either disjoint or contained in such L-block (resp.\ R-block).
	  	\item If an L-block (resp.\ R-block) of an LR-row $f_1$ is colored and there is a distinct LR-row $f_2$ for which its L-block (resp.\ R-block) is also colored with the same color, then $f_1$ and $f_2$ are nested in $A$.
	 \end{enumerate}
\end{definition}

\begin{remark} \label{obs:partially2nested_hasnogems}
The second assertion of the definition of partially $2$-nested implies that there are no monochromatic gems or monochromatic weak gems in $A$ because $A$ is admissible and thus any two labeled non-LR-rows induce no $D_1$ as a subconfiguration. 
Moreover, the third assertion implies that there are no monochromatic weak gems in $A$.
Furthermore, the last statement implies that there are no badly-colored doubly-weak gems in $A$.
\end{remark}

The following Corollary is a straightforward consequence of Remark~\ref{obs:partially2nested_hasnogems} and Lemma~\ref{lema:LR-orderable_caract_bymatrices}.

\begin{corollary} \label{cor:partially_2-nested_caract}
	An admissible matrix $A$ is partially $2$-nested if and only if $A$ contains as a subconfiguration no $M_0$, $M_{II}(4)$, $M_V$, monochromatic gems nor monochromatic weak gems nor badly-colored doubly-weak gems and the tagged matrix $A^*_{\tagg}$ does not contain any Tucker matrices, $M_2'(k)$, $M_2''(k)$, $M_3'(k)$, $M_3''(k)$ for $k\geq 3$, $M_4'$, $M_4''$, $M_5'$, $M_5''$ as a subconfiguration.
\end{corollary}

\subsection{A characterization of $2$-nested matrices} \label{subsec:teo_2nestedmatrices}

We begin this section by stating and proving a lemma that characterises when a partial $2$-coloring can be extended to a total proper $2$-coloring, for every partially $2$-colored connected graph $G$.
Then, we give the definition and some properties of the auxiliary matrix $A+$. These properties will be helpful throughout the proof of Theorem~\ref{teo:2-nested_charact} at the end of the section.

\begin{lemma} \label{lema:2-color-extension}
	Let $G$ be a connected graph with a partial proper $2$-coloring of the vertices. Then, the partial $2$-coloring can be extended to a total proper $2$-coloring of the vertices of $G$ if and only if all of the following conditions hold:
	\begin{itemize}
		\item There are no even induced paths such that the only colored vertices of the path are its endpoints, and they are colored with distinct colors.
		\item There are no odd induced paths such that the only colored vertices of the path are its endpoints, and they are colored with the same color.
		\item There are no induced uncolored odd cycles.
		\item There are no induced odd cycles with exactly one colored vertex.
		\item There are no induced odd cycles with exactly two consecutive colored vertices.
	\end{itemize}	  
\end{lemma}

\begin{proof}
The `only if' part is trivial.

On the other hand, for the `if' part, suppose all of the conditions hold. Notice that, since $G$ has a given proper partial $2$-coloring, then there are no adjacent vertices pre-colored with the same color. We denote $(G,f)$ to refer to $G$ and the given partial proper $2$-coloring $f$ of its vertices.

Let $H$ be the subgraph of $G$ induced by its uncolored vertices and let $f$ be the given proper partial $2$-coloring of $V(G)$. The proof is by induction on the number of vertices of $H$.

For the base case, this is to say when $|H| = 1$, let $v$ be in $H$. If $v$ cannot be colored, then there are two distinct vertices $x_1$ and $x_2$ such that $x_1$ and $x_2$ have different colors. Thus, the set $\{ x_1, v, x_2 \}$ either induces an even path in $G$ of length 2 with the endpoints colored with distinct colors, or an induced $C_3$ with exactly one uncolored vertex, which results in a contradiction.

For the inductive step, suppose that we can extend the partial $2$-coloring of $G$ to a proper $2$-coloring if $|V(H)| \leq k$. 
Suppose that $|V(H)| = k+1$. If $H$ is not connected, then the result follows by applying the induction hypothesis to each of its connected components. Moreover, if $H$ is disconnected from $G\left[G-H\right]$, then $H$ is an uncolored graph with no odd cycles, thus we can extend $f$ using any proper $2$-coloring for $H$. Henceforth, we assume $H$ is connected and there is a vertex in $H$ having at least one neighbour in $V(G-H)$.

Let $v$ in $H$ be any vertex such that $N(v) \cap V(G-H)\neq \emptyset$. Every vertex $w$ in $N(v)\cap V(G-H)$ must be colored with the same color, for if not we find either an induced cycle of length $3$ with exactly one uncolored vertex or an even induced path with its endpoints colored with distinct colors. Suppose that such a color is red. Thus, we extend the given partial proper $2$-coloring $f$ to $f'$ assigning the color blue to $v$. 
We will see that the graph $G$ with the partial proper $2$-coloring $f'$ fulfills all the assertions.

It is straightforward that there are no uncolored odd cycles in $(G,f')$, for there were no odd uncolored cycles in $(G,f')$. Furthermore, using the same argument, we see that there are no induced odd cycles with exactly one colored vertex nor induced odd cycles with exactly two consecutive colored vertices, for this would imply that there is either an odd uncolored cycle or an odd cycle with exactly one colored vertex in $(G,f)$. 

Since all the assertions hold for $(G,f)$, if there was an even induced path $P=<v_1, \ldots, v_j>$ such that the only colored vertices are its endpoints and they are colored with distinct colors, then the only possibility is that one of such endpoints is $v$. Let $v=v_1$. Notice that $v_2$ lies in $H\setminus \{v\}$, thus it is uncolored in both $(G,f)$ and $(G,f')$.
Let $w \in N(v) \cap V(G-H)$. Notice that $w v_j \not\in E(G)$, for they are both colored with red.
If $w$ is nonadjacent to every vertex in $P$ with the exception of $v=v_1$, then $<P,w>$ is an odd induced path in $(G,f)$ such that the only colored vertices are its endpoints and are both colored with the same color.
Suppose to the contrary that $w$ is adjacent to at least one vertex in $P$.
If $w$ is adjacent to a vertex in $P$ with an even index, then let us consider $v_i$ to be the neighbour of $w$ in $P$ with the smaller even index. In this case, we find an induced odd cycle  $<w, v=v_1, \ldots, v_i, w>$ with exactly one colored vertex in $(G,f')$, and this is equivalent to having an induced odd cycle with exactly one colored vertex in $(G,f)$, which results in a contradiction.
If instead $w$ is adjacent to a vertex in $P$ having an odd index, then we consider $v_i$ to be the neighbour of $w$ in $P$ with the largest odd index. In this case, we find $<w,v_i, \ldots, v_j>$ which is an induced path of length $j-i+1$ --which is odd-- in $(G,f)$ such that the only colored vertices are $w$ and $v_j$ and they are colored with the same color, and this contradicts one of the assertions.

The same argument holds if there is an odd induced path in $(G,f')$.
\end{proof}

	 Let $A$ be an admissible matrix, let $S_1$ and $S_2$ be a partition of the LR-rows of $A$ such that every pair of rows in $S_i$ is nested, for each $i=1,2$. 
	 Since $A$ contains no $D_0$ as a subconfiguration, there is a row $m_L$ such that $m_L$ is labeled with L and contains every L-block of those rows in $A$ that are labeled with L. Analogously, we find a row $m_R$ such that every R-block of a row in $A$ labeled with R is contained in $m_R$. Moreover, there are two rows $m_1$ in $S_1$ and $m_2$ in $S_2$ such that every row in $S_i$ is contained in $m_i$, for each $i=1,2$.
	 This property allows us to define the following auxiliary matrix, which will be helpful throughout the proof of Theorem~\ref{teo:2-nested_charact}.
	 
\begin{definition} \label{def:A+}
Let $A$ be an admissible matrix and let $\Pi$ be a suitable 
LR-ordering of $A$. The enriched matrix $A+$ is the result of applying the following rules to $A$:
\begin{itemize}
	\item Every empty row is deleted.

	\item Each LR-row $f$ with exactly one block is replaced by a row labeled with either L or R, depending on whether it has an L-block or an R-block.

	\item Each LR-row $f$ with exactly two blocks, is replaced by two uncolored rows, one having a $1$ in precisely the columns of its L-block and labeled with L, and another having a $1$ in precisely the columns of its R-block and labeled with R. We add a column $c_f$ with $1$ in precisely these two rows and $0$ in the remaining ones. 

	\item If there is at least one row labeled with L or R in $A$, then each LR-row $f$ whose entries are all $1$'s is replaced by two uncolored rows, one having a $1$ in precisely the columns of the maximum L-block and labeled with L, and another having a $1$ in precisely the complement of the maximum L-block and labeled with R. We add a column $c_f$ with $1$ in precisely these two rows and $0$ in the remaining ones.
\end{itemize}
Notice that every non-LR-row remains the same. See Figure~\ref{fig:example_B+} for an example.

\end{definition} 

\begin{figure}
\begin{align*}
	B = \scriptsize{\bordermatrix{ ~ &  ~ \cr
	\textbf{LR} & 1 0 0 0 0  \cr
	\textbf{LR} & 1 1 0 0 1 \cr
 	 \textbf{LR} & 1 1 1 1 1  \cr
					 & 0 1  1  0  0 \cr
	\textbf L & 1 1 1 0 0 \cr
	\textbf{LR} & 0 0 0 0 0 \cr
	\textbf R & 0 0 0 1 1  }\,
	\begin{matrix}
	\\ \\ \\ \\ \textcolor{dark-red}{\bullet} \\ \textcolor{blue}{\bullet} \\ \textcolor{blue}{\bullet} 
	\end{matrix} }
	&&
B+ = \scriptsize{ \bordermatrix{ & \cr
	\textbf{L} & 1 0 0 0 0 0 0 \cr
	\textbf{L} & 1 1 0 0 0 1 0  \cr
	\textbf{R} & 0 0 0 0 1 1 0 \cr
	\textbf{L} & 1 1 1 0 0 0 1 \cr
	\textbf{R} & 0 0 0 1 1  0  1\cr
					 & 0 1  1  0  0  0  0 \cr
	\textbf L & 1 1 1 0 0 0 0 \cr
	\textbf R & 0 0 0 1 1 0 0 }\,
	\begin{matrix}
	\\ \\ \\ \\ \\ \\ \textcolor{dark-red}{\bullet} \\ \textcolor{blue}{\bullet} \\ 
	\end{matrix} }
\end{align*}
\caption{Example of an enriched admissible matrix $B$ and $B+$. The last two columns of $B+$ are $c_{r_2}$ and $c_{r_3}$.} \label{fig:example_B+}
\end{figure}

\begin{remark}\label{obs:props_de_A+_suitable}
Let $A$ be a partially $2$-nested matrix. Since $A$ is admissible, LR-orderable and contains no $M_0$, $M_{II}(4)$, $M_V$ or $S_0(k)$ for every even $k \geq 4$ as a subconfiguration, then by Lemma~\ref{lema:hay_suitable_ordering} we know that there exists a suitable LR-ordering $\Pi$. Hence, whenever we define the auxiliary matrix $A+$ for such a matrix $A$, we will always consider a suitable LR-ordering $\Pi$ to do so.

Let us consider $A+$ as defined in Definition~\ref{def:A+} according to a suitable LR-ordering $\Pi$. Suppose there is at least one LR-row in $A$. Recall that, since $A$ is admissible, the LR-rows may be split into two disjoint subsets $S_1$ and $S_2$ such that the LR-rows in each subset are totally ordered by inclusion. This implies that there is an inclusion-wise maximal LR-row $m_i$ for each $S_i$, $i=1,2$. If we assume that $m_1$ and $m_2$ overlap, then either the L-block of $m_1$ is contained in the L-block of $m_2$ and the R-block of $m_1$ contains the R-block of $m_2$, or vice versa. Since $\Pi$ is suitable and $A$ contains neither $D_1$ nor $D_4$ as a subconfiguration and has no uncolored rows labeled with either L or R, if there is at least one LR-row in $A$, then the following holds:
	\begin{itemize}
		\item There is an inclusion-wise maximal L-block $b_L$ in $A+$ such that every R-block in $A+$ is disjoint with $b_L$. 
		\item There is an inclusion-wise maximal R-block $b_R$ in $A+$ such that every L-block in $A+$ is disjoint with $b_L$. 
	\end{itemize}
	
Therefore, when defining $A+$ we replace each LR-row having two blocks by two distinct rows, one labeled with L and the other labeled with R, such that the new row labeled with L does not intersect with any row labeled with R and the new row labeled with R does not intersect with any row labeled with L. 
\end{remark}

Notice that $A$ differs from $A+$ only in its LR-rows, which are either deleted or replaced in $A+$ by labeled uncolored rows. 



\begin{definition} \label{def:proper2coloring}
	A color assignment to some (eventually none) of the blocks of an enriched matrix $A$ using two colors is a \emph{proper $2$-coloring} if $A$ is admissible, the L-block and R-block of each LR-row of $A$ are colored with distinct colors, and $A$ contains no monochromatic gems, weak monochromatic gems or badly-colored doubly-weak gems as subconfigurations. If there is no (resp.\ at least one) uncolored row in $A$, then we say that the proper $2$-coloring is \emph{total} (resp.\ \emph{partial}).
	
	Given an assignment of two colors to some (eventually none) of the blocks an enriched matrix $A$ using two colors, we say it is a \emph{proper $2$-coloring of $A+$} if it is a proper $2$-coloring of $A$. 
	
\end{definition}  

\begin{remark} \label{obs:admisible_espartialproper2color}
Let $A$ be an enriched matrix. If $A$ is admissible, then the given coloring of the blocks is a proper $2$-coloring. This follows from the fact that every colored row is either labeled with L or R, or is an empty LR-row, thus there are no monochromatic gems, monochromatic weak gems or badly-colored weak gems in $A$ for they would induce $D_1$.
\end{remark}

The following is a straightforward consequence of Remark~\ref{obs:partially2nested_hasnogems}.

\begin{lemma}\label{lema:part2nested_is_2_colored}
	Let $A$ be an enriched matrix. If $A$ is partially $2$-nested, then the given coloring of $A$ is a proper partial $2$-coloring. 
	Moreover, if $A$ is partially $2$-nested and admits a total $2$-coloring, then $A$ with such $2$-coloring is partially $2$-nested.
\end{lemma}

\begin{lemma}  \label{lema:2-nested_if}
	Let $A$ be an enriched matrix. Then, $A$ is $2$-nested if $A$ is partially $2$-nested and the given partial block bi-coloring of $A$ can be extended to a total proper $2$-coloring of $A$. 
\end{lemma}

\begin{proof}
Let $A$ be an enriched matrix that is partially $2$-nested and for which the given coloring of the blocks can be extended to a total proper $2$-coloring of $A$. In particular, this induces a total block bi-coloring for $A$. Indeed, we want to see that a proper $2$-coloring induces a total block bi-coloring for $A$. 
Notice that the only pre-colored rows may be those labeled with L or R and those empty LR-rows.
Let us see that each of the assertions of Definition~\ref{def:2-nested} hold.

\begin{enumerate}
	\item Since $A$ is an enriched matrix and the only rows that are not pre-colored are the nonempty LR-rows and those that correspond to U-blocks, then there is no ambiguity when considering the coloring of the blocks of a pre-colored row (Assertion~\ref{item:2nested2} of Definition~\ref{def:2-nested}). 
	
	\item If $A$ is partially $2$-nested, then in particular is admissible, LR-orderable and contains no $M_0$, $M_{II}(4)$ or $M_V$ as a subconfiguration. Thus, by Lemma~\ref{lema:hay_suitable_ordering}, there is a suitable LR-ordering $\Pi$ for the columns of $A$. We consider $A$ ordered according to $\Pi$ from now on. Since $\Pi$ is suitable, then every L-block of an LR-row and an R-block of a non-LR-row are disjoint, and the same holds for every R-block of an LR-row and an L-block of a non-LR-row (assertion~\ref{item:2nested4} of Definition~\ref{def:2-nested}). 

	\item Since $A$ is admissible, it contains as a subconfigurations no matrix in $\mathcal{D}$. Moreover, since $A$ is partially $2$-nested, by Corollary~\ref{cor:partially_2-nested_caract} there are no monochromatic gems or weak gems and no badly-colored doubly-weak gems induced by pre-colored rows. It follows from this and the fact that the LR-ordering is suitable, that assertion~\ref{item:2nested8} of Definition~\ref{def:2-nested} holds.

	\item The pre-coloring of the blocks of $A$ can be extended to a total proper $2$-coloring of $A$. This induces a total block bi-coloring for $A$, for which we can deduce the following assertions:
	\begin{itemize}
	\item Since there is a total proper $2$-coloring of $A$, in particular the L-block and R-block of each LR-row are colored with distinct colors. (Assertion~\ref{item:2nested1} of Definition~\ref{def:2-nested}).
	
	\item Each L-block and R-block corresponding to distinct LR-rows with nonempty intersection are also colored with distinct colors since there are no badly-colored doubly-weak gems in $A$ (Assertion~\ref{item:2nested9} of Definition~\ref{def:2-nested}).
	
	\item Since $A$ is admissible, every L-block and R-block corresponding to distinct non-LR-rows are colored with different colors since there is no $D_1$ in $A$ (Assertion~\ref{item:2nested5} of Definition~\ref{def:2-nested}).
	
	\item Since there are no monochromatic weak gems in $A$, an L-block of an LR-row and an L-block of a non-LR-row that contains the L-block must be colored with distinct colors. Furthermore, if any L-block and a U-block are not disjoint and are colored with the same color, then the U-block is contained in the L-block. (Assertions~\ref{item:2nested3} and ~\ref{item:2nested7} of Definition~\ref{def:2-nested}).
	
	\item There is no monochromatic gem in $A$, then each two U-blocks colored with the same color are either disjoint or nested. (Assertion~\ref{item:2nested6} of Definition~\ref{def:2-nested}).
	\end{itemize}
\end{enumerate}
\end{proof}

\begin{lemma} \label{lema:if_suitable_noM}
Let $A$ be an enriched matrix. If $A$ admits a suitable LR-ordering, then $A$ contains no $M_0$, $M_{II}(4)$, $M_V$ or $S_0(k)$ as a subconfiguration for any even $k \geq 4$. 
\end{lemma}

\begin{proof}
The result follows trivially if $A$ contains no LR-rows because $A$ admits an LR-ordering. Thus, if we consider $A$ without its LR-rows, that submatrix has the C$1$P and hence it contains no Tucker matrix as a subconfiguration. Toward a contradiction, suppose that $A$ contains either $M_0$, $M_{II}(4)$, $M_V$ or $S_0(k)$ as a subconfiguration for some even $k \geq 4$.
Since it contains no $M_I(k)$ as a subconfiguration for every $k \geq 3$, in particular there is no $M_0$ or $S_0(k)$ where at most one of the rows is an LR-row. Moreover, it is easy to see that, if we reorder the columns of $M_0$, then there is no possible LR-ordering in which every L-block and every R-block are disjoint.
Similarly, consider $S_0(4)$, whose first row has a $1$ in every column. We may assume that the last row is an LR-row for any other reordering of the columns yields an analogous situation with one of the rows. However, whether the first row is unlabeled or not, the first and the last row prevent a suitable LR-ordering. The reasoning is analogous for any even $k > 4$.

Suppose that $A$ contains $M_V$ as a subconfiguration, and let $f_1$,$ f_2$, $f_3$ and $f_4$ be the rows of $M_V$ depicted as follows:
\[ M_V= \scriptsize{ \bordermatrix{ &  \cr
f_1 & 1 1 0 0 0 \cr
f_2 & 0 0 1 1 0 \cr
f_3 & 1 1 1 1 0 \cr
f_4 & 1 0 0 1 1  } }
\]
If the first row is an LR-row, then either $f_3$ or $f_4$ is an LR-row, for if not we find $M_I(3)$ in $A^*$ as a subconfiguration, which is not possible because there is an LR-ordering in $A$. The same holds if the second row is an LR-row. 
If $f_3$ is an LR-row, then $f_4$ is an LR-row, for if not $f_4$ must have only one block. Thus, if $f_4$ is an unlabeled row, then it intersects both blocks of $f_3$, and if $f_4$ is an R-row, then its R-block intersects the L-block of $f_3$. 
However, if we move the columns so that the L-block of $f_3$ does not intersect the R-block of $f_4$, then we either cannot split $f_1$ into two blocks such that one starts on the left and the other ends on the right, or we cannot maintain a single block in $f_2$. It follows analogously if we assume that $f_4$ is an LR-row, thus $f_1$ is not an LR-row. By symmetry, we assume that $f_2$ is also a non-LR-row, and thus the proof is analogous if only $f_3$ and $f_4$ may be LR-rows. 

Suppose $A$ contains $M_{II}(4)$ as a subconfiguration. Let us denote $f_1$, $f_2$, $f_3$ and $f_4$ to the rows of $M_{II}(4)$ depicted as follows:
\[ M_{II}(4) = \scriptsize{ \bordermatrix{ &  \cr
f_1 & 0 1 1 1 \cr
f_2 & 1 1 0 0 \cr
f_3 & 0 1 1 0 \cr
f_4 & 1 1 0 1 } }
\]

If $f_2$ is an LR-row, then necessarily $f_3$ or $f_4$ are LR-rows, for if not we find $M_I(3)$ in $A^*$.
If only $f_2$ and $f_3$ are LR-rows, then we find $M_{II}(4)$ as a subconfiguration in $A^*$.
If instead only $f_2$ and $f_4$ are LR-rows, then --as it is-- whether $f_1$ is an R-row or an unlabeled row, the block of $f_1$ intersects the L-block and the R-block of $f_4$ (and also the L-block of $f_2$). The only possibility is to move the second column all the way to the right and split $f_2$ into two blocks and give the R-block of $f_4$ length $2$. However in this case, it is not possible to move another column and obtain an ordering that keeps all the $1$'s consecutive for $f_3$ and $f_1$ not intersecting both blocks of $f_4$ simultaneously. Thus, $f_1$ is also an LR-row. However, for any ordering of the columns, either it is not possible to simultaneously split the string of $1$'s in $f_1$ and keep the L-block of $f_2$ starting on the left, or it is not possible to simultaneously maintain the string of $1$'s in $f_3$ consecutive and the L-block of $f_1$ disjoint with the R-block of $f_4$. 
It follows analogously if both $f_3$ and $f_4$ are LR-rows. 
Hence, $f_2$ is a non-LR-row, and by symmetry, we may assume that $f_3$ is also a non-LR-row. Suppose now that $f_1$ is an LR-row. If $f_4$ is not an LR-row, then there is no possible way to reorder the columns and having a consecutive string of $1$'s for the rows $f_2$, $f_3$ and $f_4$ simultaneously, unless we move the fourth column all the way to the left. However in that case, either $f_4$ is an L-row and its L-block intersects the R-block of $f_1$ of it is an unlabeled row that intersects both blocks of $f_1$. Moreover, the same holds if $f_4$ is an LR-row, with the difference that in this case the R-block of $f_4$ intersects the L-block of $f_1$ or the string of $1$'s in $f_2$ and $f_3$ is not consecutive.\end{proof}

\begin{lemma} \label{lema:2-nested_onlyif}
	Let $A$ be an enriched matrix with a block bi-coloring. 
	If $A$ is $2$-nested, then $A$ is partially $2$-nested and the total block bi-coloring induces a proper total $2$-coloring of $A$.
\end{lemma}

\begin{proof}

If $A$ is $2$-nested, then in particular there is an LR-ordering $\Pi$ for the columns. Moreover, by assertions~\ref{item:2nested4} and ~\ref{item:2nested7} of Definition~\ref{def:2-nested},
such an ordering is suitable.

Suppose first there is a monochromatic gem in $A$. Such a gem is not induced by two unlabeled rows since in that case assertion~\ref{item:2nested6} of Definition~\ref{def:2-nested} would not hold. Hence, such a gem is induced by at least one labeled row. Moreover, if one is a labeled row and the other is an unlabeled row, then assertion~\ref{item:2nested7} of Definition~\ref{def:2-nested} would not hold. Thus, both rows are labeled. By assertion~\ref{item:2nested5} of Definition~\ref{def:2-nested}, if the gem is induced by two non-disjoint L-block and R-block, then it is not monochromatic, disregarding on whether they correspond to LR-rows or non-LR-rows. Hence, exactly one of the rows is an LR-row. However, by assertion~\ref{item:2nested4} of Definition~\ref{def:2-nested}, an L-block of an LR-row and an R-block of a non-LR-row are disjoint, thus they cannot induce a gem. 

Suppose there is a monochromatic weak gem in $A$, thus at least one of its rows is a labeled row. It is not possible that exactly one of its rows is a labeled row and the other is an unlabeled row, since assertion~\ref{item:2nested7} of Definition~\ref{def:2-nested} holds. Moreover, these rows do not correspond to rows labeled with L and R, respectively, for assertions~\ref{item:2nested4} and~\ref{item:2nested5} of Definition~\ref{def:2-nested} hold. Furthermore, both rows of the weak gem are LR-rows, since if exactly one is an LR-row, then assertions~\ref{item:2nested3},~\ref{item:2nested4} and~\ref{item:2nested7} of Definition~\ref{def:2-nested} hold and thus it is not possible to have a weak gem. However, in that case, assertion~\ref{item:2nested5} of Definition~\ref{def:2-nested} guarantees that this is also not possible.

Finally, there is no badly-colored doubly-weak gem since assertions~\ref{item:2nested4},~\ref{item:2nested5} and~\ref{item:2nested9} of Definition~\ref{def:2-nested} hold. 

\vspace{1mm}
Now, let us see that $A$ is admissible. Since there is an LR-ordering of the columns, there are no $D_0$, $D_2$, $D_3$, $D_6$, $D_7$, $D_8$ or $D_{11}$ contained as a subconfiguration in $A$. Moreover, by assertion~\ref{item:2nested5} of Definition~\ref{def:2-nested}, there is no $D_1$ contained in $A$ as a subconfiguration.  	
As we have previously seen, there are no monochromatic gems or monochromatic weak gems. Hence, it is easy to see that if there is a total block bi-coloring, then $A$ contains none of the matrices in $\mathcal{S}$ or $\mathcal{P}$ as a subconfiguration.
Suppose $A$ contains $D_4$ as a subconfiguration. By assertion~\ref{item:2nested8} of Definition~\ref{def:2-nested}, if there are two L-blocks of non-LR-rows colored with distinct colors, then every LR-row has a nonempty L-block, and in this case such an L-block is contained in both rows labeled with L. However, by assertion~\ref{item:2nested3} of Definition~\ref{def:2-nested}, the L-block of the LR-row is properly contained in the L-blocks of the non-LR-rows, thus it must be colored with a distinct color than the color assigned to each L-block of a non-LR-row, and this leads to a contradiction.
By assertion~\ref{item:2nested3} of Definition~\ref{def:2-nested}, there is no $D_5$ contained in $A$ as a subconfiguration. 
Let us suppose $A$ contains $D_9$ as a subconfiguration induced by the rows $f_1$, $f_2$ and $f_3$, where $f_1$ is labeled with L and $f_2$ and $f_3$ are LR-rows. Suppose that $f_1$ is colored with red. Since the L-block of $f_2$ is contained in $f_1$, by assertion~\ref{item:2nested3} of Definition~\ref{def:2-nested}, then the L-block of $f_2$ is colored with blue. The same holds for the L-block of $f_3$. However, $f_2$ and $f_3$ are not nested, thus by by assertion~\ref{item:2nested9} of Definition~\ref{def:2-nested}, the L-blocks of $f_2$ and $f_3$ are colored with distinct colors, which results in a contradiction.

Let us suppose $A$ contains $D_{10}$ as a subconfiguration induced by the rows $f_1$, $f_2$, $f_3$ and $f_4$, where $f_1$ is labeled with L and colored with red, $f_2$ is labeled with R and colored with blue, and $f_3$ and $f_4$ are LR-rows. 
Since the L-block of $f_3$ is properly contained in $f_1$, then by assertion~\ref{item:2nested3} of Definition~\ref{def:2-nested}, it is colored with blue. By assertion~\ref{item:2nested1} of Definition~\ref{def:2-nested}, the R-block of $f_3$ is colored with red. Using a similar argument, we assert that the R-block of $f_4$ is colored with red and the L-block of $f_4$ is colored with blue. However, $f_3$ and $f_4$ are non-disjoint and non-nested, thus the L-block of $f_3$ and the R-block of $f_4$ are colored with distinct colors, which results in a contradiction.

By Lemma~\ref{lema:if_suitable_noM}, since there is a suitable LR-ordering, then $A$ contains  no $M_0$, $M_{II}(4)$, $M_V$ or $S_0(k)$ as a subconfiguration for any even $k \geq 4$. 

It follows from assertion~\ref{item:2nested9} of Definition~\ref{def:2-nested} and the fact that there is an LR-ordering, that $A$ contains neither $D_{12}$ nor $D_{13}$ as a subconfiguration.
Therefore $A$ is partially $2$-nested.

\vspace{1mm}
Finally, we will see that the total block bi-coloring for $A$ induces a proper total $2$-coloring of $A$. 
Since all the assertions of Definition~\ref{def:2-nested} hold, it is straightforward that there are no monochromatic gems or monochromatic weak gems or badly-colored weak gems in $A$. For more details on this, see Remark~\ref{obs:partially2nested_hasnogems} and Lemma~\ref{lema:2-nested_if} since the same arguments are detailed there. 
Moreover, since assertion~\ref{item:2nested1} of Definition~\ref{def:2-nested} holds, the L-block and R-block of the same LR-row are colored with distinct colors. Therefore, it follows that a total block bi-coloring of $A$ induces a proper total $2$-coloring of $A$. 
\end{proof}

\vspace{2mm}
The following corollary is a straightforward consequence of the previous results.

\begin{corollary} \label{lema:B_ext_2-nested} 
	Let $A$ be an enriched matrix. If $A$ is partially $2$-nested and $B$ is obtained from $A$ by extending its partial coloring to a total proper $2$-coloring, then $B$ is $2$-nested if and only if for each LR-row its L-block and R-block are colored with distinct colors and $B$ contains no monochromatic gems, monochromatic weak gems or badly-colored doubly-weak gems as subconfigurations.
	
	
\end{corollary}


The main result of this work is Theorem~\ref{teo:2-nested_charact}, whose proof follows directly from Theorem~\ref{teo:case1_2nested} --which was proven in~\cite{PDGS18}-- and Theorem~\ref{teo:case2_2nested}, whose proof will be given below.

\begin{theorem}[\cite{PDGS18}] \label{teo:case1_2nested}
Let $A$ be an enriched matrix such that every row of $A$ is unlabeled. Then, $A$ is $2$-nested if and only if $A$ contains no $F_0$, $F_1(k)$, $F_2(k)$ or their dual matrices for every odd $k\geq 5$ as subconfigurations (see Figure~\ref{fig:forb_F})	
\end{theorem}

Theorem~\ref{teo:case1_2nested} is a particular case of Theorem~\ref{teo:2-nested_charact}. More precisely, this theorem gives necessary and suficient conditions for an enriched matrix with all unlabeled --and thus uncolored-- rows to be $2$-nested. Notice that such an enriched matrix is simply a $(0,1)$-matrix. Moreover, the condition of being $2$-nested for this particular case reduces to admiting a C$1$P ordering and a $2$-coloring assignment for the rows such that both the red and blue subconfigurations are nested matrices.

For the remaining enriched matrices --this is, for every enriched matrix with at least one labeled row-- we have the following lemma that characterises all the forbidden subconfigurations.

\begin{theorem} \label{teo:case2_2nested}
Let $A$ be an enriched matrix such that $A$ has at least one labeled row. Then, $A$ is $2$-nested if and only if $A$ contains none of the following listed matrices or their dual matrices as subconfigurations: 
	\begin{itemize}		
	\item $M_0$, $M_{II}(4)$, $M_V$ or $S_0(k)$ for every even $k$ (See Figure~\ref{fig:forb_M_chiquitas})
	\item Every enriched matrix in the family $\mathcal{D}$ (See Figure~\ref{fig:forb_D}) 
	\item Every enriched matrix in the family $\mathcal{F}$ (See Figure~\ref{fig:forb_F}) 
	\item Every enriched matrix in the family $\mathcal{S}$ (See Figure~\ref{fig:forb_S})
	\item Every enriched matrix in the family $\mathcal{P}$ (See Figure~\ref{fig:forb_P})
	\item Monochromatic gems, monochromatic weak gems, badly-colored doubly-weak gems 
	\end{itemize}
and $A^*$ contains no Tucker matrices and none of the enriched matrices in $\mathcal{M}$ or their dual matrices as subconfigurations (See Figure~\ref{fig:forb_LR-orderable}).
\end{theorem}

The proof is organized as follows. The `only if' part follows almost immediately using Lemma~\ref{lema:2-nested_onlyif} and the char\-ac\-ter\-i\-za\-tions of admissibility, LR-orderable and partially $2$-nested given throughout this section. 
For the `if' part, we will define an auxiliary graph $H(A)$ that is partially $2$-colored according to the pre-coloring of the blocks of $A$. Toward a contradiction, we suppose that $H(A)$ is not bipartite. Using the characterization given in Lemma~\ref{lema:2-color-extension}, we know there is one of the $5$ possible kinds of paths or cycles. We analyse each case and reach a contradiction.

\begin{proof}	
	Let $A$ be an enriched matrix with at least one labeled row in $A$ (labeled with L, R or LR), and suppose $A$ is $2$-nested. In particular, $A$ is partially $2$-nested with the given coloring and the block bi-coloring induces a total proper $2$-coloring of $A$. 
	Thus, by Corollary~\ref{cor:partially_2-nested_caract}, $A$ is admissible and contains no $M_0$, $M_{II}(4)$, $M_V$, $S_0(k)$ for any even $k \geq 4$, monochromatic gems, monochromatic weak gems or badly-colored doubly-weak gems as subconfigurations, and $A^*_\tagg$ contains no Tucker matrices, $M_4'$, $M_4''$, $M_5'$, $M_5''$, $M'_2(k)$, $M''_2(k)$, $M_3'(k)$, $M_3''(k)$, $M_3'''(k)$, for any $k \geq 4$ as subconfigurations. 
	In particular, since $A$ is admissible, there is no $D_{13}$ induced by any three LR-rows.
	
	Moreover, notice that every pair of consecutive rows of any of the matrices $F_0$, $F_1(k)$, and $F_2(k)$ for all odd $k \geq 5$ induces a gem, and there is an odd number of rows in each matrix. Thus, if one of these matrices is a submatrix of $A_\tagg$, then there is no proper $2$-coloring of the blocks. Therefore, $A$ contains no $F_0$, $F_1(k)$, and $F_2(k)$ for any odd $k \geq 5$ as subconfigurations. A similar argument holds for $F'_0$, $F_1'(k)$, $F_2'(k)$, changing 'gem' by 'weak gem' whenever one of the two rows considered is a labeled row. 
	
	\vspace{1mm}	
	Conversely, suppose $A$ is not $2$-nested. Henceforth, we assume that $A$ is admissible. 
	
	If $A$ is not partially $2$-nested, then either $A$ contains $M_0$, $M_{II}(4)$, $M_V$, $S_0(k)$ for some even $k \geq 4$ as a subconfiguration, or there is a subconfiguration of $M$ in $A^*_\tagg$ such that $M$ is one of the forbidden subconfigurations for partially $2$-nested stated above. 
	Henceforth, we assume that $A$ is partially $2$-nested.
	
	If $A$ is partially $2$-nested but is not $2$-nested, then the pre-coloring of the rows of $A$ (which is a proper partial $2$-coloring of $A$ since $A$ is admissible) cannot be extended to a total proper $2$-coloring of $A$. 
	

We wish to extend the partial pre-coloring given for $A$. By Corollary~\ref{lema:B_ext_2-nested}, if $B$ is obtained by extending the pre-coloring of $A$ and $B$ is $2$-nested, then neither two blocks corresponding to the same LR-row are colored with the same color, nor there are monochromatic gems, monochromatic weak gems or badly-colored doubly-weak gems in $B$. 
	Let us consider the auxiliary matrix $A+$, defined from a suitable LR-ordering $\Pi$ of the columns of $A$.
	Notice that, if there is at least one labeled row in $A$, then there is at least one labeled row in $A+$ and these labeled rows in $A+$ correspond to rows of $A$ that are labeled with either L, R, or LR. 
	
\vspace{1mm}
	We now define a graph $H = H(A+)$, whose vertices are the rows of $A+$. We say a vertex is an \emph{LR-vertex (resp.\ non-LR vertex)} if it corresponds to a block of an LR-row (resp.\ non-LR-row) of $A$. The adjacencies in $H$ are as follows:
	\begin{itemize}
		\item Two non-LR vertices are adjacent in $H$ if the underlying uncolored submatrix of $A$ determined by these two rows contains a gem or a weak gem as a subconfiguration. 
		\item Two LR-vertices corresponding to the same LR-row in $A$ are adjacent in $H$.
		\item Two LR-vertices $v_1$ and $v_2$ corresponding to distinct LR-rows are adjacent if $v_1$ and $v_2$ are labeled with the same letter in $A+$ and the LR-rows corresponding to $v_1$ and $v_2$ overlap in $A$. 
		\item An LR-vertex $v_1$ and a non-LR vertex $v_2$ are adjacent in $H$ if the rows corresponding to $v_1$ and $v_2$ are not disjoint and $v_2$ is not contained in $v_1$. 
	\end{itemize}	  
	The vertices of $H$ are partially colored with the pre-coloring given for the rows of $A$.
	
	Notice that every pair of vertices corresponding to the same LR-row $f$ induces a gem in $A+$ that contains the column $c_f$, and two adjacent LR-vertices $v_1$ and $v_2$ in $H$ do not induce any kind of gem in $A+$, except when considering both columns $c_{r_1}$ and $c_{r_2}$.
	
	Let us consider a path $P$ in the graph $H$. We say a subpath $Q$ of $P$ is an \emph{LR-subpath} if every vertex in $Q$ is an LR-vertex. We say an LR-subpath $Q$ in $P$ is \emph{maximal} if $Q$ is not properly contained in any other LR-subpath of $P$.
We say that two LR-vertices $v_i$ and $v_j$ are \emph{consecutive }in the path $P$ (resp.\ in a cycle $C$ in $H$) if either $j=i+1$ or $v_l$ is unlabeled for every $l=i+1, \ldots, j-1$. 
			
The following claims will be useful throughout the proof.

\begin{claim}  \label{claim:1_teo2nested} 
	Let $C$ be a cycle in $H= H(A+)$. Then, there are at most 3 consecutive LR-vertices labeled with the same letter. The same holds for any path $P$ in $H$. 
\end{claim}

Let $v_1$, $v_2$ and $v_3$ be 3 consecutive LR-vertices in $H$, all labeled with the same letter. Notice that any subset in $H$ of LR-vertices labeled with the same letter in $A+$ corresponds to a subset of the same size of distinct LR-rows in $A$. By definition, two LR-vertices are adjacent in $H$ only if they are labeled with the same letter and the corresponding rows in $A$ contain a gem, or equivalently, if they are not nested. 
Moreover, notice that once the columns of $A$ are ordered according to $\Pi$, these rows have a $1$ in the first non-tag column and a $1$ in the last non-tag column.
Hence, if there are 4 consecutive LR-vertices $v_1$, $v_2$, $v_3$ and $v_4$ in the cycle $C$ of $H$ and all of them are labeled with the same letter, then $v_1$ and $v_2$ are not nested, $v_2$ and $v_3$ are not nested and $v_1$ and $v_3$ must be nested. Thus, since $v_2$ and $v_4$ and $v_1$ and $v_4$ are also nested, then $v_4$ either contains $v_1$ and $v_2$, or is nested in both. In either case, since $v_3$ and $v_4$ are not nested, then $v_1$ and $v_3$ are not nested and this results in a contradiction. \QED

\begin{claim} \label{claim:2_teo2nested}
	There are at most 6 uncolored labeled consecutive vertices in $C$. The same holds for any path $P$ in $H$.
\end{claim} 

This follows from the previous claim and the fact that every pair of uncolored labeled vertices labeled with distinct rows are adjacent only if they correspond to the same LR-row in $A$. \QED

\vspace{2mm}
If $A$ is not $2$-nested, then the partial $2$-coloring given for $H$ cannot be extended to a total proper $2$-coloring of the vertices. Notice that the only pre-colored vertices are those labeled with either L or R, and those LR vertices corresponding to an empty row, which we are no longer considering when defining $A+$. 
According to Lemma~\ref{lema:2-color-extension} we have 5 possible cases. 

\begin{mycases}
	\vspace{1mm} 
	\case  \textit{There is an odd  induced path $P =  v_1, v_2, \ldots, v_k $ such that the only colored vertices are $v_1$ and $v_k$, and they are colored with the same color.
}	

We assume without loss of generality throughout the proof that $v_1$ is labeled with L, since it is analogous otherwise by symmetry. Recall that an odd path is induced by an even number of vertices.

If $v_2, \ldots, v_{k-1}$ are unlabeled rows, then we find either $S_2(k)$ or $S_3(k)$ as a subconfiguration which is not possible since $A$ is admissible.  	
Suppose there is at least one LR-vertex in $P$.	Recall that, an LR-vertex and a non-LR-vertex are adjcent in $H$ only if the rows in $A+$ are both labeled with the same letter and the LR-row is properly contained in the non-LR-row. 

Suppose that every LR-vertex in $P$ is nonadjacent with each other. Let $v_i$ be the first LR-vertex in $P$, and suppose first that $i = 2$. Since $v_2$ is an LR-vertex and is adjacent to $v_1$, then $v_2$ is labeled with L and $v_2 \subsetneq v_1$. Hence, since we are assuming there are no adjacent LR-vertices in $P$ and $k \geq 4$, then $v_3$ is not an LR-vertex, thus it is unlabeled because we are considering a suitable LR-ordering to define $A+$. Let $v_3, \ldots, v_j$ be the maximal sequence of consecutive unlabeled vertices in $P$ that starts in $v_3$. Thus, $v_l \subseteq v_1$ for every $3 \leq l \leq j$. 
	
	Notice that there are no other LR-vertices in $P$: toward a contradiction, let $v_j$ be the next LR-vertex in $P$. If $v_j$ is labeled with L, then $v_j$ is adjacent to $v_1$ since $v_3$ is nested in $v_1$, which is not possible. It is analogous if $v_j$ is labeled with R. Thus, $v_l$ is unlabeled for every $3 \leq l \leq k-1$.
	Moreover, the vertex $v_k$ is labeled with L, for if not we find $D_1$ in $A$ induced by $v_1$ and $v_k$ and this is not possible since $A$ is admissible. However, in that case we find $S_5(k)$ as a subconfiguration.
	
	Hence, if $v_i$ is an isolated LR-vertex (this is, nonadjacent to other LR-vertices), then $i >2$. It follows that $v_2$ is an unlabeled vertex.
	Notice that a similar argument as in the previous paragraph proves that there are no more LR-vertices in $P$: since $v_{i+1}$ is nested in $v_{i-1}$, it follows that any other LR-vertex is adjacent to $v_{i-1}$. 
	Suppose first that $v_i$ is labeled with L and let $v_2, \ldots, v_{i-1}$ be the maximal sequence of unlabeled vertices in $P$ that starts in $v_2$.
	
	 Since $v_i$ is the only LR-vertex in $P$, if $v_k$ is labeled with L, then necessarily $i=k-1$ for if not $v_k$ is adjacent to $v_{i-1}$. However, since in that case $v_k \supsetneq v_{k-1}= v_i$ and $v_k$ is nonadjacent to every other vertex in $P$, then we find $S_5(k)$ as a subconfiguration. Analogously, if $v_k$ is labeled with R, since $v_j \subseteq v_{i-1}$ for every $j > i$, then $v_k$ is adjacent to $v_{i-1}$ which leads to a contradiction. 
	
	Suppose now that $v_i$ is labeled with R and remember that $i>2$. Furthermore, $v_j$ is unlabeled for every $j > i$. Moreover, $v_j$ is nested in $v_{i-1}$ for every $j>i$, for if not $v_k$ would be adjacent to $v_i$. However, in that case $v_k$ is adjacent to $v_{i-1}$, whether labeled with R or L, and this results in a contradiction. 
	
	Notice that we have also proven that, when considering an admissible matrix and a suitable LR-ordering to define $H$, there cannot be an isolated LR-vertex in such a path $P$, disregarding of the parity of the length of $P$. This last part follows from the previous discussion 
	and the fact that, if the length is 2 
	and $P$ has one LR-vertex, then 
	we find $D_4$ as a subconfiguration if the endpoints are labeled with the same letter and we find $D_5$ if the endpoints are labeled one with L and the other with R, since the endpoints are colored with distinct colors. Moreover, the ordering would not be suitable, which is a necessary condition for $A+$ and thus $H$ to be well-defined. If the length of $P$ is even and greater than 2, then the arguments are analogous as in the odd case.
	The following claim is a straightforward consequence of the previous discussion. 
	
\begin{claim}   \label{claim:3_teo2nested}
If there is an isolated LR-vertex in $P$, then it is the only LR-vertex in $P$. Moreover, there are no two nonadjacent LR-vertices in $P$. Equivalently, every LR-vertex in $P$ lies in a sequence of consecutive LR-vertices.
\end{claim}


 	
\vspace{1mm} 	
It follows from Claims~\ref{claim:1_teo2nested},~\ref{claim:2_teo2nested} and~\ref{claim:3_teo2nested} that there is one and only one maximal LR-subpath in $P$. Thus, we have one subcase for each possible length of such maximal LR-subpath of $P$, which may be any integer between $1$ and $5$, both inclusive. 
	
	\subcase Let $v_i$ and $v_{i+1}$ be the two adjacent LR-vertices that induce the maximal LR-subpath. Suppose first that both are labeled with L and that $i=2$. Since $v_2$ is an LR-vertex, $v_2$ is nested in $v_1$ and $v_3$ contains $v_1$.
	Moreover, $v_4$ is labeled with R, for if not $v_4$ is also adjacent to $v_2$. This implies that the R-block of the LR-row corresponding to $v_2$ contains $v_4$ in $A$, for if not we find $D_6$ as a subconfiguration. However, either the R-block of $v_2$ intersects the L-block of $v_3$ --which is not possible since we are considering a suitable LR-ordering--, or $v_3$ is disjoint with $v_4$ since the LR-rows corresponding to $v_2$ and $v_3$ are nested, and thus we find $D_6$ as a subconfiguration. Hence, $k>4$. 
	
	By Claim~\ref{claim:3_teo2nested} and since there is no other LR-vertex in the maximal LR-subpath, there are no other LR-vertices in $P$. Equivalently, $v_4, \ldots, v_{k-1}$ are unlabeled vertices. Moreover, this sequence of unlabeled vertices is chained to the right, since if it was chained to the left, then every left endpoint of $v_j$ for $j=4, \ldots, k-1$ would be greater than $r(v_1)$ and thus $v_k$ results adjacent to $v_2$. Hence, we find $P_0(k-1,0)$ in $A$ as a subconfiguration induced by $v_1$, $v_2$, $v_4, \ldots, v_k$
	which is not possible since $A$ is admissible. 
	The proof is analogous if $i>2$, with the difference that we find $P_0(k-1,i)$ in $A$ as a subconfiguration.
	Furthermore, the proof is analogous if $v_i$ and $v_{i+1}$ are labeled with distinct letters.
	
\vspace{1mm}

\subcase Let $Q= < v_i, v_{i+1},v_{i+2}>$ be the maximal LR-subpath of $P$. Suppose first that not every vertex in $Q$ is labeled with the same letter. 

If $v_i$ is labeled with R, then 
$v_{i+1}$ is labeled with R, since there is a sequence of unlabeled vertices between $v_1$ and $v_i$. This follows from the fact that if not, $v_{i+1}$ would be adjacent to either $v_1$ or some vertex in the unlabeled chain. Moreover, it follows analogously that, if $v_i$ is labeled with R, then $v_{i+2}$ is also labeled with R. 
Since we are assuming that not every vertex in $Q$ is labeled with the same letter, $v_{i}$ is labeled with L and we have the following claim. 

\begin{claim} \label{claim:1_case2-1_teo2nested}
For every maximal LR-subpath of $P$, the first vertex is labeled with L.
\end{claim}

Suppose $v_i$ and $v_{i+1}$ are both labeled with L and $v_{i+2}$ is labeled with R. Notice that, if $i=2$, then $v_2$ is labeled with L, $v_4$ is labeled with R and $v_3$ may be labeled with either L or R.

Since $v_{i+1}$ and $v_{i+2}$ are labeled with distinct letters, they correspond to the same LR-row in $A$. Notice that $v_i$ is contained in $v_{i+1}$. Thus, since $v_i$ and $v_{i+1}$ are adjacent, the R-block corresponding to $v_i$ in $A$ contains $v_{i+2}$. Therefore, we find $P_0(k,i)$ or $P_1(k,i)$ as a subconfiguration of the submatrix induced by $P$. 

If instead $v_{i+1}$ and $v_{i+2}$ are both labeled with R, then $v_i$ and $v_{i+1}$ are the two blocks of the same LR-row in $A$. Hence, since $v_{i+1}$ and $v_{i+2}$ are adjacent and $v_k$ is nonadjacent to $v_{i+1}$, then $v_{i+1}$ contains $v_{i+2}$ and thus the L-block of the LR-row corresponding to $v_{i+2}$ contains $v_i$. Once again, we find either $P_0(k,i)$ or $P_1(k,i)$ as a subconfiguration in $A$.

\vspace{1mm}
Suppose now that all vertices in $Q$ are labeled with the same letter and suppose first that $i=2$. Since $v_1$ and $v_2$ are adjacent, then every vertex in $Q$ is labeled with L.
Notice that $k>4$ since $v_5$ is uncolored and the endpoints of $P$ are colored with the same color.
Since $v_2$ is adjacent to $v_1$, then $v_1 \subsetneq v_3$ and $v_4 \subsetneq v_3$.
Since $k$ is even and $k>4$, then $v_5$ is an unlabeled vertex. Moreover, for every unlabeled vertex $v_j$ such that $j >4$, $l(v_j) > r(v_1)$ and $r(v_j) \leq r(v_3)$, for if not $v_j$ and $v_3$ would be adjacent. However, $v_k$ is not labeled with L for in that case it would be adjacent to $v_3$. Furthermore, if $v_k$ is labeled with R, then we find $D_8$ as a subconfiguration, which is not possible since we assumed $A$ to be admissible.

Suppose now that $i>2$. In this case, there is a sequence of unlabeled vertices between $v_1$ and $v_i$. 
If every vertex in $Q$ is labeled with L, since $v_1$ and $v_i$ are nonadjacent (and thus $v_1$ is nested in $v_i$) and $v_{i+1}$ is nonadjacent with $v_{i-1}$, then $v_i \subsetneq v_{i+1}$, $v_{i+2} \subsetneq v_{i+1}$. It follows that $v_j$ is contained between $r(v_i)$ and $r(v_{i+1})$ for every $j>i+2$ and therefore $v_k$ is adjacent either to $v_{i+1}$ or $v_i$, which results in a contradiction. 

If every vertex in $Q$ is labeled with R, then $v_{i+1} \subsetneq v_{i}$ and $v_{i+1} \subsetneq v_{i+2}$ for if not $v_{i+1}$ would be adjacent to $v_{i-1}$ and $v_{i+2}$. Hence, if $i+2=k-1$, then $v_k$ would be adjacent also to $v_{i+1}$. Hence, there is at least one unlabeled vertex $v_j$ with $j>i+2$. Moreover, for every such vertex $v_j$, it holds that $l(v_j) < l(v_i)$ and $r(v_j) > l(v_{i+1})$. Hence, if $v_k$ is labeled with R, then $v_k$ is adjacent to $v_{i+1}$. If instead $v_k$ is labeled with L, then we find $D_8$ as a subconfiguration of $A$ induced by $v_k$ and the LR-rows corresponding to $v_{i}$ and $v_{i+1}$. 

\subcase Let $Q= <v_i, v_{i+1}, v_{i+2}, v_{i+3}>$ be the maximal LR-subpath of $P$. Notice that either 2 vertices 
are labeled with L and 2 vertices are labeled with R, or $1$ vertex is labeled with L and 3 vertices are labeled with R, or viceversa. Moreover, by Claim~\ref{claim:1_case2-1_teo2nested} we know that $v_i$ is labeled with L. Every vertex $v_j$ such that $1<j<i$ or $i+3<j<k$ is an unlabeled vertex. 

Suppose first that $v_i$ is the only vertex in $Q$ labeled with L. Thus, $v_{i+1}$ is the R-block of the LR-row in $A$ corresponding to $v_i$. Hence, either $v_{i+1} \subsetneq v_{i+2}$ or vice versa.
Notice that there is at least one unlabeled vertex $v_j$ between $v_{i+3}$ and $v_k$, for if not $v_k$ is adjacent to $v_{i+1}$ or $v_{i+2}$. Moreover, either $v_j$ is contained in $v_{i+2} \setminus v_{i+3}$ or in $v_{i+3} \setminus v_{i+2}$ for every $j>i+4$. In any case, $v_k$ results adjacent to either $v_{i+2}$ or $v_{i+3}$, which results in a contradiction. 

Hence, at least $v_i$ and $v_{i+1}$ are labeled with L. Suppose that $v_{i+2}$ is labeled with R --and thus $v_{i+3}$ is labeled with R. Notice that, if $v_{i+3} \supsetneq v_{i+2}$, then there is no possible label for $v_k$ for, if $v_k$ is labeled with R, then $v_k$ is adjacent to $v_{i+2}$, whereas if $v_k$ is labeled with L, then $v_k$ is adjacent to $v_i$ and $v_{i+1}$.
However, the same holds if $v_{i+2} \supsetneq v_{i+3}$ because there is at least one unlabeled vertex $v_j$ with $j>i+3$ and thus for every such vertex holds $l(v_j)>l(v_{i+2})$ and therefore this case is not possible.


Finally, suppose that $v_i$, $v_{i+1}$ and $v_{i+2}$ are labeled with L and thus $v_{i+3}$ is labeled with R. Thus, $v_k$ is labeled with R and is nested in $v_{i+3}$. Moreover, there is a chain of unlabeled vertices $v_j$ between $v_{i+3}$ and $v_k$ such that $v_j$ is nested in $v_{i+3}$ for every $j>i+4$. 
Furthermore, $v_i \subsetneq v_{i+1}$ and $v_i \subseteq v_{i+2} \subsetneq v_{i+1}$: if $i=2$, then $v_2 \subsetneq v_1$ and, since $v_3$ and $v_4$ are nonadjacent to $v_1$, $v_3, v_4 \supseteq v_1$. If instead $i>2$, then for every unlabeled vertex $v_j$ between $v_1$ and $v_i$, $r(v_j) < r (v_i)$, except for $j=i-1$ for which holds $r(v_{i-1}) > r(v_i)$. Hence, since $v_{i+1}$ and $v_{i+2}$ are nonadjacent to every such vertex, $v_j \subset v_{i+1}, v_{i+2}$ for $1<j<i$. We find $P_0(k-3,i)$ in $A$ as a subconfiguration 
since the R-block corresponding to $v_i$ is contained in $v_{i+3}$ and thus the R-block intersects the chain of vertices between $v_{i+3}$ and $v_k$. 

We have the following as a consequence of the previous arguments.
\begin{claim} \label{claim:2_case2-1_teo2nested}
Let $v_i$ and $v_{i+1}$ be the first LR-vertices that appear in $P$. If $v_{i+1}$ is also labeled with L, then $v_i \subsetneq v_{i+1}$. Moreover, if $v_{i+2}$ is also an LR-vertex that is labeled with L, then $v_{i+2} \subsetneq v_{i+1}$.
\end{claim} 

\subcase 
Let $Q=<v_i, \ldots, v_{i+4}>$ be the maximal LR-subpath of $P$. By Claim~\ref{claim:1_case2-1_teo2nested}, $v_i$ is labeled with L. Moreover, either (1) $v_i$ and $v_{i+1}$ are labeled with L and $v_{i+2}$, $v_{i+3}$ and $v_{i+4}$ are labeled with R, or (2) $v_i$, $v_{i+1}$ and $v_{i+2}$ are labeled with L and $v_{i+3}$ and $v_{i+4}$ are labeled with R. It follows from Claim~\ref{claim:2_case2-1_teo2nested} that $v_i \subsetneq v_{i+1}$. 

Let us suppose the (1) holds. If $v_{i+3} \subsetneq v_{i+4}$, then there is at least one unlabeled vertex in $P$ between $v_{i+4}$ and $v_k$, for if not $v_k$ would be adjacent to $v_{i+2}$. Since every vertex $v_j$  for $i+5<j \leq k$ is contained in $v_{i+4} \setminus v_{i+3}$, it follows that $v_k$ is adjacent to $v_{i+2}$ and thus this is not possible. Hence, $v_{i+3} \supsetneq v_{i+4}$. 
Furthermore, $v_{i+3} \supsetneq v_{i+2}$, and, since $v_k$ is nonadjacent to $v_{i+2}$, $v_{i+2} \supsetneq v_{i+4}$. Since there is a sequence of unlabeled vertices between $v_{i+4}$ and $v_k$, we find $P_2(k,i-2)$ as a subconfiguration if $v_{i+4}$ is nested in the R-block of $v_{i+2}$, or we find $P_0(k-3,i-2)$ otherwise.

Suppose now (2) holds. By Claim~\ref{claim:2_case2-1_teo2nested}, $v_i \subsetneq v_{i+1}$ and $v_{i+2} \subsetneq v_{i+1}$. Furthermore, since $v_k$ is nonadjacent to $v_{i+3}$, it follows that $v_{i+3} \supsetneq v_{i+4}$. In this case, we find $P_2(k,i-2)$ as a subconfiguration if $v_{i+4}$ is nested in the R-block of $v_{i+2}$, or we find $P_0(k-3,i-2)$ otherwise. 

\vspace{1mm}
\subcase  
Suppose by simplicity that the length of $P$ is 7 (the proof is analogous if $k>8$), and thus let $Q=<v_2, \ldots, v_7>$ be the maximal LR-subpath of $P$ of length 5 
Notice that $v_8$ is labeled with R and colored with the same color as $v_1$.
Hence, $v_2$, $v_3$ and $v_4$ are labeled with L and $v_5$, $v_6$ and $v_7$ are labeled with R. By Claim~\ref{claim:2_case2-1_teo2nested}, $v_2 \subsetneq v_3$ and $v_4 \subsetneq v_3$. It follows that $v_2 \subsetneq v_4$, since $v_1$ and $v_4$ are nonadjacent. Using an analogous argument, we see that $v_5 \subsetneq v_6$, $v_6 \supsetneq v_5, v_7$ and $v_7 \subsetneq v_5$ for if not it would be adjacent to $v_8$.
Since consecutive LR-vertices are adjacent, the LR-rows corresponding to $v_{i+3}$ and $v_{i+4}$ are not nested, and the same holds for the LR-rows in $A$ of $v_3$ and $v_2$. Since $A$ is admissible, the LR-rows of $v_6$ and $v_3$ are nested. This implies that the L-block of the LR-row corresponding to $v_6$ contains the L-block of $v_4$ and $v_2$.
Moreover, since the LR-rows of $v_7$ and $v_5$ are nested, the LR-rows of $v_6$ and $v_7$ are not and $v_7$ is contained in  $v_6$, then the L-block of $v_7$ contains the L-block of $v_6$. Hence, $v_7$ contains $v_5$ and thus $v_8$ results adjacent to $v_5$, which is a contradiction. 
	
	\vspace{2mm}
 	\case \textit{There is an even induced path $P =  < v_1, v_2, \ldots, v_k >$ such that the only colored vertices are $v_1$ and $v_k$, and they are colored with distinct colors.}
	
	Throughout the proof of the previous case we did not take under special consideration the parity of $k$, with one exception: when $k=5$ and the maximal LR-subpath has length $1$. In other words, notice that for every other case, we find the same forbidden subconfigurations of admissibility with the appropriate coloring for those colored labeled rows.

	Suppose that $k=5$, the maximal LR-subpath has length $1$, and suppose without loss of generality that $v_2$ and $v_3$ are the LR-vertices (it is analogous otherwise by symmetry). If both are labeled with L, then $v_2$ is contained in $v_3$ and thus the R-block of $v_2$ properly contains the R-block of $v_3$. Moreover, since $v_4$ is unlabeled and adjacent to $v_5$ --which should be labeled with R since the LR-ordering is suitable--, it follows that there is at least one column in which the R-block of the LR-row corresponding to $v_3$ has a $0$ and $v_5$ has a $1$. Furthermore, there exists such a column in which also the R-block of $v_2$ has a $1$. Since $v_1$ and $v_2$ are adjacent, $v_2 \subsetneq v_1$ and thus there is also a column in which $v_2$ has a $0$, $v_3$ has a $1$ and $v_1$ has a $1$. Moreover, there is a column in which $v_1$, $v_2$ and $v_3$ have a $1$ and $v_5$ and the R-blocks of $v_2$ and $v_3$ all have a $0$, and an analogous column in which $v_1$, $v_2$ and $v_3$ have a $0$ and $v_5$ and the R-blocks of $v_2$ and $v_3$ have a $1$. It follows that there is $D_{10}$ in $A$ as a subconfiguration which is not possible since $A$ is admissible. 
	If instead $v_2$ is labeled with L and $v_3$ is labeled with R, then $v_2$ and $v_3$ are the L-block and R-block of the same LR-row $r$ in $A$, respectively. We can find a column in $A$ in which $v_1$ and $r$ have a $1$ and the other rows have a $0$, a column in which only $v_1$ has a $1$, a column in which only $v_4$ has a $1$ (notice that $v_4$ is unlabeled), and a column in which $r$, $v_4$ and $v_5$ have a $1$ and $v_1$ has a $0$. It follows that there is $P_0(4,0)$ in $A$ as a subconfiguration, which results in a contradiction.

 	\vspace{2mm}
 	\case \textit{There is an induced uncolored odd cycle $C$ of length $k$.} 

	If every vertex in $C$ is unlabeled, then the proof is analogous as in Theorem~\ref{teo:case1_2nested},
	where we considered that there are no labeled vertices of any kind.
	
	Suppose there is at least one LR-vertex in $C$. Notice that there no labeled vertices in $C$ corresponding to rows in $A$ labeled with L or R, which are the only colored rows in $A+$.
		
	Suppose $k = 3$. If 2 or 3 vertices in $C$ are LR-vertices, then $A$ contains $D_7$, $D_8$, $D_9$, $D_{11}$, $D_{12}$, $D_{13}$ or $S_7(3)$ as a subconfiguration. If instead there is exactly one LR-vertex, then 
	we find $F'_0$ in $A$ as a subconfiguration since every uncolored vertex corresponds either to an unlabeled row or to an LR-row.
		
	Suppose that $k\geq 5$ and let $C = v_1, v_2, \ldots, v_k $ be an uncolored odd cycle of length $k$. Suppose first that there is exactly one LR-vertex in $C$. We assume without loss of generality by symmetry that $v_1$ is such LR-vertex and that $v_1$ is labeled with L in $A+$.
	
 	Hence, either $v_j$ is nested in $v_1$, or $v_j$ is disjoint with $v_1$, for every $j=3, \ldots, k-1$.
 	If $v_j$ is nested in $v_1$ for every $j=3, \ldots, k-1$, then 
 	$l(v_k) < l(v_{k-2}) < l(v_{k-3}) < \ldots < l(v_2)<r(v_1)$ and $r(v_k) > r(v_1)$, since $v_k$ is adjacent to $v_1$ and nonadjacent to $v_j$ for every $j=3, \ldots, k-1$. Hence, we either find $F_1(k)$ or $F'_1(k)$ as a subconfiguration in $A$ induced by the columns $l(v_{k-1}), \ldots, r(v_k)$. 
 	 	
 	If instead $v_j$ is disjoint with $v_1$ for all $j=3, \ldots, k-1$, then $v_j$ is nested in $v_k$ for every $j=3, \ldots, k-2$. In this case, we find $F_2(k)$ or $F'_2(k)$ as a subconfiguration in $A$ induced by the columns $l(v_k)-1, \ldots, r(v_{k-1})$. 
 	
 	Now we will see what happens if there is more than one LR-vertex in $C$. First we need the following claim.
 	
\begin{claim} \label{claim:5_claim_2-nested}
	If $v$ and $w$ in $C$ are two nonadjacent consecutive LR-vertices, then one of the two paths in $C$ joining $v$ and $w$ has
	exactly one unlabeled vertex. 
\end{claim} 	

If $k=5$, then we have to see what happens if $v_1$ and $v_4$ are such vertices and $v_5$ is an LR-vertex. We are assuming that $v_2$ and $v_3$ are unlabeled since by hypothesis $v_1$ and $v_4$ are consecutive LR-vertices in $C$. Suppose that $v_1$ and $v_4$ are labeled with L and for simplicity assume that $v_1 \subsetneq v_4$. Thus, $v_5$ is labeled with L, for if not $v_5$ can only be adjacent to $v_1$ or $v_4$ and not both. Moreover, since $v_5$ is nonadjacent to $v_2$, $v_5$ is contained in $v_1$ and $v_4$. In this case, we find $F_2(5)$ as a subconfiguration in $A$.

If instead $v_1$ is labeled with L and $v_4$ is labeled with R, then $v_5$ is the L-block of the LR-row corresponding to $v_4$. In this case, we find $S_7(4)$ as a subconfiguration of $A_\tagg$.

Let $k>5$, and suppose without loss of generality that $v_1$ and $v_4$ are such LR-vertices. Thus, by hypothesis, $v_2$ and $v_3$ are unlabeled vertices. 
Suppose first that $v_1$ and $v_4$ are labeled with L and $v_1 \subsetneq v_4$. Then $l(v_2) < l(v_3)$. 
If $v_j$ is unlabeled for every $j>4$, then $v_j$ is nested in $v_3$ and thus $v_k$ cannot be adjacent to $v_1$. 
Moreover, for every $j>4$, $v_j$ is not an LR-vertex labeled with L either. Suppose to the contrary that $v_5$ is an LR-vertex labeled with L. Since $v_5$ is adjacent to $v_4$ and the LR-rows corresponding to $v_1$ and $v_4$ are nested, then $v_5$ is also adjacent to $v_1$, which is not possible since we are assuming that $k>5$. If instead $j>5$, then 
$r(v_j)>l(v_3)$ and thus it is adjacent to $v_3$, since there is a sequence of unlabeled vertices between $v_4$ and $v_j$.
By an analogous argument, we may assert that $v_j$ is not an LR-vertex for every $j>4$. The proof is analogous if $v_1 \supsetneq v_4$.

Thus, let us suppose now that $v_1$ is labeled with L and $v_4$ is labeled with R. If $v_5$ is the L-block of the LR-row corresponding to $v_4$, since $v_2$ and $v_5$ are nonadjacent, then $r(v_5)< l(v_2)$ and hence $v_5 \subsetneq v_1$. 
Moreover, $v_6$ is not an LR-vertex for in that case $v_6$ must be labeled with L and thus $v_6$ is also adjacent to $v_1$. Furthermore, since at least $v_6$ is an unlabeled vertex, then every LR-vertex $v_j$ in $C$ with $j>4$ is labeled with L, for if not $v_j$ is either adjacent to $v_4$ or nonadjacent to $v_6$ (or the maximal sequence of unlabeled vertices in $C$ that contains $v_6$).
Thus, we may assume that there no other LR-vertices in $C$, perhaps with the exception of $v_k$. 
However, if $v_k$ is an LR-vertex labeled with L, then 
it is also adjacent to $v_5$, since it is adjacent to $v_1$. And if $v_k$ is unlabeled, then $v_k$ is adjacent to $v_2$, $v_3$ or $v_4$ ($v_k$ must contain these vertices so that it results nonadjacent to them, but the first non-null column of the row corresponding to $v_4$ is the limit 
since $v_4$ is labeled with R and thus its block ends in the last column).

Analogously, if $v_5$ is unlabeled, then $v_k$ is nonadjacent to $v_1$ since it must be contained in $v_3$. Finally, if $v_5$ is an LR-vertex labeled with R, then it is contained in $v_4$. Thus, the only possibility is that $v_{k-1}$ is an LR-vertex labeled with R and $v_7$ is the L-block of the corresponding LR-row. However, since $A$ is admissible, either $v_6$ is nested in $v_5$ or $v_6$ is nested in $v_4$. In the former case, it results also adjacent to $v_4$ and in the latter case it results nonadjacent to $v_5$, which is a contradiction. 
Notice that the arguments are analogous if the number of unlabeled vertices in the paths in both directions of the cycle is more than $2$. Therefore, the claim holds. \QED

This claim follows from the previous proof.

\begin{claim} \label{cor:5_claim_2-nested}
	If $C$ is an odd uncolored cycle in $H$, then there are at most two nonadjacent LR-vertices.
\end{claim} 
	
 	Suppose that $v_1$ and $v_i$ are consecutive nonadjacent LR-vertices, where $i>2$. It follows from Claim~\ref{claim:5_claim_2-nested} that $i=3$ or $i=k-1$. We assume the first without loss of generality, and suppose that $v_1$ is labeled with L.
 	Suppose there is at least one more LR-vertex nonadjacent to both $v_1$ and $v_3$, and let $v_j$ be the first LR-vertex that appears in $C$ after $v_3$. It follows from Claim~\ref{claim:5_claim_2-nested} that $j=5$. 	
 	If $v_1$ and and $v_3$ are labeled with distinct letters, then 
 	$v_4$ is contained in $v_2$ since $v_4$ is an unlabeled vertex, and thus $v_5$ cannot be labeled with L or R for, in either case, it would be adjacent to $v_2$.
 	Thus, every LR-vertex in $C$ must be labeled with the same letter. Let us assume for simplicity that $k=5$ (the proof is analogous for every odd $k>5$) and that $v_1 \subset v_3$. Since $v_5$ is nonadjacent to $v_3$, then the corresponding LR-rows are nested. The same holds for $v_1$ and $v_3$. Moreover, $v_5$ contains both $v_1$ and $v_3$, and the R-block of $v_3$ contains the R-block of $v_1$. Furthermore, since $v_1$ and $v_5$ are adjacent, the R-block of the LR-row corresponding to $v_1$ contains the R-block of the LR-row corresponding to $v_5$ and thus the R-block of $v_3$ also contains the R-block of $v_5$, which results in $v_3$ and $v_5$ being adjacent and thus in a contradiction that arose from assuming that there were are at least three nonadjacent LR-vertices in $C$. \QED
 	
\vspace{2mm}
	We now continue with the proof of the case. 
	Notice first that, as a consequence of the previous claim and Claim~\ref{claim:2_teo2nested}, either there are exactly two nonadjacent LR-vertices in $C$ or every LR-vertex is contained in a maximal LR-subpath of length at most $6$.
	 
	\subcase Suppose there are exactly two LR-vertices in $C$ and that they are nonadjcent. Let $v_1$ and $v_3$ be such LR-vertices. Suppose without loss of generality that $v_1 \subset v_3$. Hence, every vertex that lies between $v_3$ and $v_1$ is nested in $v_2$ because they are all unlabeled vertices by assumption. 	Thus, if $v_1$ and $v_3$ are both labeled with L, then we find $F_1(k)$ as a subconfiguration in $A$ given by the columns $r(v_1), \ldots, r(v_2)$.
	If instead $v_1$ is labeled with L and $v_3$ is labeled with R, then we find $F_2(k)$ as a subconfiguration in $A$ given by the same columns.

	\subcase Suppose instead that $v_1$ and $v_2$ are the only LR-vertices in $C$. If $v_1$ and $v_2$ are the L-block and R-block of the same LR-row, then we find $S_8(k-1)$ in $A$ as a subconfiguration. 
	If instead they are both labeled with L, then every other vertex $v_j$ in $C$ is unlabeled and $v_j$ is nested in $v_1$ or $v_2$ for every $j>3$, depending on whether $v_1 \subsetneq v_2$ or viceversa. Suppose that $v_1 \subsetneq v_2$. If there is a column in which both $v_3$ and the R-block of $v_1$ have a $1$, then we find $S_8(k-1)$ in $A$ as a subconfiguration. If there is not such a column, then we find $F_2(k)$ in $A$ as a subconfiguration.
	
 	\vspace{1mm}
	\subcase Suppose that the maximal LR-subpath $Q$ in $C$ has length $2$, 
	and suppose $Q= <v_1, v_2, v_3>$.
	If $v_1$, $v_2$ and $v_3$ are labeled with the same letter, then either $v_2 \subsetneq v_1, v_3$ or $v_2 \supsetneq v_1, v_3$, and since $v_1$ and $v_3$ are nonadjacent if $k>3$, either $v_3 \subsetneq v_1$ or $v_1 \subsetneq v_3$. Suppose without loss of generality that all three LR-vertices are labeled with L, $v_2 \subsetneq v_1, v_3$ and $v_1 \subsetneq v_3$. In this case, there is a sequence of unlabeled vertices between $v_3$ and $v_1$ such that the column index of the left endpoints of the vertices decreases as the vertex path index increases. As in the previous case, if there is a column such that the R-block of $v_2$ and $v_4$ have a $1$, then we find $S_8(k-1)$ in $A$ as a subconfiguration given by the columns $r(v_1), \ldots, l(v_1)$. If instead there is not such column, then we find $F_2(k)$ as a subconfiguration of $A$ given by the same columns.
	
	If $v_1$ and $v_2$ are labeled with L and $v_3$ is labeled with R, then there is a sequence of unlabeled vertices $v_4, \ldots, v_k$ such that the column index of the left endpoints of such vertices decreases as the path index increases. Moreover, since $v_k$ is adjacent to $v_1$ and nonadjacent to $v_2$, $v_1 \supsetneq v_2$. Hence, we find $S_7(k-1)$ contained as a subconfiguration in $A$ given by the columns $r(v_1), \ldots, l(v_1)$.	 
	 	
 	\vspace{1mm}
 	\subcase Suppose that the maximal LR-subpath $Q$ in $C$ has length $3$ 
 	and that $Q= <v_1$, $v_2$, $v_3$, $v_4>$.
 	Suppose first that $v_1$ and $v_2$ are labeled with L and $v_3$ and $v_4$ are labeled with R. If $v_1 \subsetneq v_2$, then $v_k$ cannot be adjacent to $v_1$. 
 	Thus $v_2 \subsetneq v_1$ and $v_4 \supsetneq v_3$. Since there is a chain of unlabeled vertices and its left endpoints decrease as the cycle index increases, then we find $S_7(k-1)$ as a subconfiguration in $A$ induced by every row in $A$. 
	Suppose now that $v_1$ is labeled with L and the other three LR-vertices are labeled with R. Notice first that $v_2$ is the R-block of $v_1$, the LR-rows of $v_2$ and $v_4$ are nested and $v_3 \subsetneq v_2,v_4$. Moreover, $v_2 \subsetneq v_4$, for if not $v_k$ would not be adjacent to $v_1$. Thus, the left endpoint of the chain of unlabeled vertices between $v_4$ and $v_1$ decreases as the cycle index increases. Hence, if $k=5$, then we find $S_7(3)$ as a subconfiguration in $A$ induced by the LR-rows corresponding to $v_3$ and $v_4$ and the unlabeled row corresponding to $v_5$. Suppose that $k>5$. Since $v_3 \subsetneq v_2$ and $v_2$ is the R-block of $v_1$, then the L-block of the LR-row corresponding to $v_3$ contains both $v_1$ and the L-block of $v_4$. We find $S_7(k-3)$ in $A$ as a subconfiguration induced by the rows $v_3, v_4, \ldots, v_{k-1}$.	
 	
 	\subcase Suppose now that $Q=<v_1, \ldots, v_5>$ is the longest LR-subpath in $C$, and suppose that $v_1$ and $v_2$ are labeled with L and that the remaining rows in $Q$ are labeled with R. Since $v_1$ is adjacent to $v_k$, then $v_1 \supsetneq v_2$ and $v_5 \supsetneq v_4,v_3$. Since the LR-rows corresponding to $v_3$ and $v_5$ are nested, then $v_2$ is contained in the L-block corresponding to $v_5$, and since $v_4 \subsetneq v_5$, the R-block of $v_1$ is also contained in $v_5$. Thus, we find $S_7(k-3)$ in $A$ as a subconfiguration considering the LR-rows corresponding to $v_1$ and $v_5$ and $v_6, \ldots, v_k$.
 	The proof is analogous if $Q$ has length $6$, and thus this case is finished. 
 	
 	\vspace{1mm}
 	\case \textit{There is an induced odd cycle $C = v_1, v_2, \ldots, v_k, v_1$ with exactly one colored vertex. }
 	We assume without loss of generality that $v_1$ is the only colored vertex in the cycle $C$, and that $v_1$ is labeled with L.
 	Notice that, if there are no LR-vertices in $C$, then the proof is analogous as in the case in which there are no labeled vertices of any kind. Hence, we assume there is at least one LR-vertex in $C$.

\begin{claim} \label{claim:6_claim_2-nested}
	If there is at least one LR-vertex $v_i$ in $C$ and $i \neq 2$, then $v_i$ is the only LR-vertex in $C$.
\end{claim} 

	Let $v_i$ be the LR-vertex in $C$ with the minimum index, and suppose first that $v_i$ is labeled with L. Since $i \neq 2$ and $v_1$ is a non-LR-row in $A$, $v_i \supseteq v_1$, for if not they would be adjacent. Moreover, $v_l \subset v_i$ for every $l < i-1$. 
	Toward a contradiction, let $v_j$ the first LR-vertex in $C$ with $j >i$ and suppose $v_j$ is labeled with L. Notice that the only possibility for such vertex is $j=i+1$. This follows from the fact that, if $v_{i+1}$ is unlabeled, then $v_{i+1}$ is contained in $v_{i-1}$, and the same holds for every unlabeled vertex between $v_i$ and $v_j$. Hence, if there were other LR-vertex $v_j$ labeled with L such that $j>i+1$, then it would be adjacent to $v_{i+1}$ which is not possible.		
	Then, $j= i+1$ and thus $v_j$ contains $v_l$ for every $l \leq i$. However, $v_k$ and $v_1$ are adjacent, and since $v_k$ must be an unlabeled vertex, then $v_k$ is not disjoint with $v_i$, which results in a contradiction. 
	
	Suppose that instead $v_j$ is labeled with R. Using the same argument, we see that, if $j>i+1$, then every unlabeled vertex between $v_i$ and $v_j$ is contained in $v_{i-1}$ and thus it is not possible that $v_j$ results adjacent to $v_{j-1}$ if it is unlabeled. Hence, $j=i+1$. Moreover, there must be at least one more LR-vertex labeled with R for if not, it is not possible for $v_1$ and $v_k$ to be adjacent. Thus, $v_{k-1}$ must be labeled with R and $v_k$ is the L-block of the LR-row corresponding to $v_{k-1}$. Furthermore, $v_{k-1}$ is contained in $v_1$. We find $F_2(k)$ in $A$ as a subconfiguration induced by all the rows of $A$. 
	Therefore, $v_i$ is the only LR-vertex in $C$. \QED

	The following is a straightforward consequence of the previous proof and the fact that, if $v_i$ is the first LR-vertex in $C$ and $i >2$, then every unlabeled vertex that follows $v_i$ is nested in $v_{i-1}$ and thus if $v_1$ is adjacent to $v_k$ then $v_k$ must be nested in $v_2$.
	
\begin{claim} \label{claim:6_claim_2-nested_2}
	If $v_i$ in $C$ is an LR-vertex and $i \neq 2$, then $i=3$.
\end{claim}	

	 It follows from Claim~\ref{claim:2_teo2nested} that there are at most 6 consecutive LR-vertices in such a cycle $C$.
	Let $Q=<v_i, \ldots, v_j>$ be the maximal LR-subpath and suppose that $|Q|=5$ and $v_1$ is labeled with L. 
	Notice that, if $v_i$ is labeled with R,  then $v_{j-1}$ and $v_j$ are labeled with L. Moreover, since there is a sequence of unlabeled vertices between $v_1$ and $v_i$ and $v_{j-1}$ is nonadjacent to $v_2$, then $v_{j-1}$ is contained in $v_1$ and thus it results adjacent to $v_1$, which is not possible. 
	Then, necessarily $v_i$ is labeled with L and thus $v_j$ is labeled with R. Moreover, if $i>2$, then $v_i$ contains $v_1$ and every unlabeled vertex between $v_1$ and $v_{i-1}$, and if $i=2$, then $v_2 \subsetneq v_1$. In either case, $v_{i+1}$ contains $v_i$. Hence, at most $v_{i+2}$ is labeled with L and there are no other LR-vertices labeled with L for they would be adjacent to $v_i$ or $v_{i+1}$. In particular, the last vertex of the cycle $v_k$ is not labeled with L and, since it is uncolored, $v_k$ is an unlabeled vertex. 
	However, $v_k$ is adjacent to $v_1$, and this results in a contradiction. Therefore, it is easy to see that it is not possible to have more than 4 consecutive LR-vertices in $C$. Furthermore, in the case of $|Q|=4$, either $v_i$ and $v_{i+3}$ are labeled with L and $v_{i+1}$ and $v_{i+2}$ are labeled with R, or $v_i$ and $v_{i+1}$ are labeled with R and $v_{i+3}$ is labeled with L.

\begin{claim} \label{claim:7_claim_2-nested}
Suppose $v_2$ is an LR-vertex and let $v_i$ be another LR-vertex in $C$. Then, either $i=k$ or $i \in \{3,4,5\}$. Moreover, in this last case, $v_j$ is an LR-vertex for every $2 \leq j \leq i$.
\end{claim}	

Notice first that, if $v_2$ is an LR-vertex, then by definition of $H$, $v_2$ is labeled with L and $v_2 \subsetneq v_1$. 
If $i = 3$ or $i=k$, then we are done. Suppose that $i \neq k$ and there is a sequence of unlabeled vertices $v_j$ between $v_2$ and $v_i$, where $j=3, \ldots, i-1$. Hence, since $v_2 \subsetneq v_1$, necessarily $v_j \subseteq v_1$ for $j=3, \ldots, i-1$. In that case, $v_i$ is labeled with the same letter than $v_1$ and $v_2$. Moreover, since $i \neq k$, $v_1$ and $v_i$ are nonadjacent and thus $v_i \supseteq v_1$ which is not possible since $v_{i-1} \subseteq v_1$. The contradiction came for assuming that there is a sequence of unlabeled vertices between $v_2$ and $v_i$ and that $v_i \neq v_k$. 
Hence, if $i \neq 3,k$, then every vertex between $v_2$ and $v_i$ is an LR-vertex. Since we know that the maximal LR-subpath in $C$ has length at most $4$ and $v_2$ is an LR-vertex, necessarily $v_i$ must be either $v_3$, $v_4$ or $v_5$. \QED

 We now split the proof into two cases.
 \subcase \textit{$v_2$ is an LR-vertex.}
 
 	Suppose first that $v_2$ is the only LR-vertex in $C$. By definition of $H$, $v_2$ is labeled with L and $v_2 \subsetneq v_1$. Since there are no other LR-vertices in $C$, $v_j \subseteq v_1$ for every $j<k$. 
 	In this case, we find $F'_1(k)$ as a subconfiguration in $A$ induced by the columns $r(v_2), \ldots, r(v_k)$. 
 	
 	Suppose now that there is exactly one more LR-vertex $v_i$ with $i > 2$. If $i \neq 3$, then by the previous claim we know that $i=k$. If $v_k$ is labeled with L, then we find $F'_1(k)$ as a subconfiguration in $A$ induced by the columns $r(v_2), \ldots, r(v_k)$. 
 	If instead $v_k$ is labeled with R, then $r(v_1) > l(v_k)$ but this is not possible because the LR-ordering used to define $A+$ is suitable. 
 	Suppose that $i =3$. If $v_3$ is labeled with L, then $v_3 \supseteq v_1$, and if $v_3$ is labeled with R, then $v_3$ is the R-block corresponding to the same LR-row of $v_2$ in $A$. In either case, since every other vertex $v_j$ in $C$ is unlabeled, then $l(v_j) > r(v_1)$ for every $j<k$. Thus, if $v_3$ is labeled with L, then we find $F'_2(k)$ as a subconfiguration in $A$ induced by the columns $r(v_1), r(v_2), r(v_3), \ldots, r(v_k)$. 
 	If instead $v_3$ is labeled with R, then we find $S_1(k)$ as a configuration in $A$ induced by the columns $r(v_1), r(v_{k-1}), \ldots, r(v_3)$. 
 	
 	Suppose that there are exactly two LR-vertices distinct than $v_2$. As a consequence of Claim~\ref{claim:7_claim_2-nested}, we see that these vertices are necessarily $v_3$ and $v_4$.
 If $v_3$ and $v_4$ are LR-vertices and are both labeled with L, then $v_3$ and $v_4$ correspond to two distinct LR-rows that are not nested. Moreover, since $v_2 \subsetneq v_1$, then $v_3 \supseteq v_1$ and thus $v_1 \subseteq v_4 \subsetneq v_3$. Hence, since $v_5$ is unlabeled and there is at least one column for which the R-blocks of $v_2$, $v_3$ and $v_5$ have $1$, $0$ and $1$, respectively, we find $F_1(k)$ as a subconfiguration $A$ induced by the columns $1$ to $k-1$.
 
 If instead $v_3$ or $v_4$ (or both) are labeled with R, then $v_3$ corresponds to the same LR-row in $A$ as $v_2$. This follows from the fact that, if $v_3$ and $v_4$ correspond to the same LR-row in $A$, then $v_3$ is labeled with L and $v_4$ is labeled with R. Hence, since $v_3 \subseteq v_1$, $v_k$ cannot be adjacent to $v_1$ and thus this is not possible. However, if $v_3$ is the R-block of the LR-row corresponding to $v_2$, then we find $D_9$ as a subconfiguration in $A$ induced by the three 
 rows corresponding to $v_1$, $v_2$ and $v_4$.
 
  	Suppose that there are exactly three LR-vertices other than $v_2$. Hence, these vertices are $v_3$, $v_4$ and $v_5$. Recall that $v_1$ and $v_2$ are labeled with L, and that two LR-vertices labeled with distinct letters are adjacent only if they correspond to the same LR-row in $A$. In any case, $v_5$ is labeled with R. However, since $v_1$ is labeled with L and $v_3 \supseteq v_1$, necessarily $v_k$ is adjacent to $v_3$, $v_4$ or $v_5$, which is a contradiction.
 
	 \subcase \textit{$v_2$ is not an LR-vertex.} 
	 
	By Claim~\ref{claim:6_claim_2-nested_2}, if there is an LR-vertex $v$, then there are no other LR-vertices and $v = v_3$.

	Since there is a exactly one LR-vertex in $C$ (we are assuming that there is at least one LR-vertex for if not the proof is as in Theorem~\ref{teo:case1_2nested}, then $v_2$ contains $v_j$ for every $j>3$. If $v_3$ is labeled with L, then there is $F'_2(k)$ as a subconfiguration in $A$ induced by the columns $r(v_1), \ldots, l(v_2)$. 
	If instead $v_3$ is labeled with R, then we find $S_1(k)$ as a subconfiguration in $A$ induced by the same columns.

 	\case Let us suppose there is n induced odd cycle with exactly two consecutive colored vertices.
 	\subcase \textit{There is a cycle of length $3$ with exactly one uncolored vertex. }
 	
 	Let $v_1, v_2, v_3, v_1$ be a cycle of length $3$ in $H$ with exactly one uncolored vertex. We assume without loss of generality that $v_1$ and $v_3$ are the colored vertices. Since $A+$ is defined by considering a suitable LR-ordering and $v_1$ and $v_3$ are adjacent colored vertices, then $v_1$ and $v_3$ are labeled with distinct letters, for if not, the underlying uncolored matrix induced by these rows either induce $D_0$ or do not induce any kind of gem. Moreover, $v_1$ and $v_3$ are colored with distinct colors because $A$ is admissible and thus $A$ contains no $D_1$ as a subconfiguration. 
 	Furthermore, $v_2$ is unlabeled for if not it cannot be adjacent to both $v_1$ and $v_3$, since in that case $v_2$ should be nested in both $v_1$ and $v_3$. However, we find $F''_0$ as a subconfiguration of $A$, and this is a contradiction.
 	
 	\subcase \textit{There is an induced odd cycle $C= v_1, \ldots, v_j, v_1$ where the only colored vertices are $v_1$ and $v_j$.} 
 	
 	Notice that $v_1$ and $v_j$ are colored with distinct colors, thus either $v_1$ is labeled with L and $v_j$ is labeled with R or vice versa, for $A$ is admissible and two vertices corresponding to rows both labeled with L or with R in $A$ do not induce adjacent vertices in $H(A+)$. We assume without loss of generality by symmetry that $v_1$ is labeled with L  and colored with red and $v_j$ is labeled with R and colored with blue.
 	
 	Suppose first that there are no LR-vertices in $C$. In this case, since the vertices $v_2, \ldots, v_{j-1}$ correspond to unlabeled rows, the rows corresponding to $v_2, \ldots, v_{j-2}$ are nested in $v_j$ and the rows coresponding to $v_3, \ldots, v_{j-1}$ are nested in $v_1$. In this case, we find $F_1(j)$ as a subconfiguration of $A$ and this results in a contradiction.
 	
 	If instead there is at least one LR-vertex in $C$, then there is exactly one and it should be either $v_2$ or $v_{j-1}$. This follows from the fact that we used a suitable LR-ordering to obtain $A+$ and that the blocks corresponding to $v_1$ and $v_j$ intersect precisely '`in the middle of the matrix $A$''. Hence, the blocks of any LR-row cannot intersect both L-blocks and R-blocks from colored rows of $A$ and therefore, if $v_2$ (resp.\ $v_{j-1}$) is an LR-row, then the L-block of $v_2$ (resp.\ R-block of $v_{j-1}$) should be nested in $v_1$ (resp.\ $v_j$). We assume without loss of generality that $v_2$ it the only LR-vertex in $C$. Thus, the L-block of the LR-row corresponding to $v_2$ is nested in $v_1$ and does not intersect $v_j$. Moreover, the rows corresponding to $v_3, \ldots, v_{j-2}$ are nested in $v_1$ and are chained to the right. Therefore we find $S_1(j-2)$ in the subconfiguration induced by the rows corresponding to $v_2, \ldots, v_j$
 	
\end{mycases} 	
 	
 	\vspace{1mm}
 	This finishes the proof, since we have reached a contradiction by assuming that $A$ is partially $2$-nested but not $2$-nested.
\end{proof}

\section*{Acknowledgements}
We would like to thank Luciano Grippo for his engagement during the early stages of this work.







\bibliographystyle{abbrv}
\bibliography{2nested}

\end{document}